%% file: Article_JGA.tex
\documentclass{amsart}
\usepackage[T1]{fontenc}
\usepackage[utf8x]{inputenc}

\usepackage{graphicx}
\usepackage{booktabs}
\usepackage{amsmath, amssymb}
\usepackage{float}
\usepackage{subcaption}
\usepackage{caption}
\usepackage{xcolor}
\usepackage{listings}
\usepackage{booktabs}
\usepackage{enumitem}

\usepackage{adjustbox}

\usepackage{lmodern}

\usepackage{amscd,amsmath,amssymb,amsthm,amsfonts,epsfig,graphics}
\usepackage{graphicx}
\usepackage{epstopdf}
\usepackage{mathabx} 
\usepackage{hyperref}
\usepackage{subcaption}
\usepackage{comment}
\let\oldtocsection=\tocsection
\let\oldtocsubsection=\tocsubsection
\renewcommand{\tocsection}[2]{\hspace{0em}\oldtocsection{#1}{#2}}
\renewcommand{\tocsubsection}[2]{\hspace{1.8em}\oldtocsubsection{#1}{#2}}

\usepackage{mdframed} 
\usepackage{lipsum}
\usepackage{tcolorbox} 
\usepackage{comment}  

\tcbuselibrary{theorems}

\usepackage[margin=3cm]{geometry}
\newtheorem{theorem}{Theorem}[section]
\vspace{10pt}
\newtheorem{proposition}[theorem]{Proposition}
\newtheorem{lemma}[theorem]{Lemma}

\theoremstyle{definition}
\newtheorem{definition}[theorem]{Definition}
\newtheorem{examples}[theorem]{Examples}
\vspace{10pt}

\theoremstyle{remark}
\newtheorem{remark}[theorem]{Remark}
\newtheorem{corollary}[theorem]{Corollary}

\newtcbtheorem{mytheo}{Theorem}%
{colback=green!5,colframe=black!35!black,fonttitle=\bfseries}{th}

\newcommand{\di}{\mathop{}\!\mathrm{d}}

\newcommand{\df}{:=}

\newcommand{\setB}{\mathbb{B}}
\def\setR{\mathbb R}

\newcommand{\nd}{\noindent}
\newcommand{\norm}[1]{\left\lVert#1\right\rVert}

\newcommand{\cl}{\mathbb{R}}

\newcommand{\ddim}{\mathbb{R}^d}

\newcommand{\Om}{\Omega}

\def\cC{\mathcal C}
\def\cD{\mathcal D}

\def\cL{\mathcal L}

\def\cC{\mathcal C}

\begin{document}
\pagenumbering{arabic}

\title[Manifolds with kinks and graph Laplacian]{Manifolds with kinks and the asymptotic behavior of the graph Laplacian operator with Gaussian kernel}

\author{Susovan Pal}
\address{S.~Pal: Department of Mathematics and Data Science, Vrije Universiteit Brussel, Pleinlaan 2, B-1050 Elsene, Belgium.}
\email{susovan.pal@vub.be}

\author{David Tewodrose*}
\thanks{*Corresponding author.}
\address{D.~Tewodrose : Department of Mathematics and Data Science, Vrije Universiteit Brussel, Pleinlaan 2, B-1050 Elsene, Belgium.}
\email{david.tewodrose@vub.be}

\date{\today}

\maketitle

\begin{abstract}
    \nd We introduce manifolds with kinks, a class of manifolds with possibly singular boundary that notably contains manifolds with smooth boundary and corners. We derive the asymptotic behavior of the Graph Laplace operator with Gaussian kernel and its deterministic limit on these spaces as bandwidth goes to zero. We show that this asymptotic behavior is determined by the inward sector of the tangent space and, as special cases, we derive its behavior near interior and border points. Lastly, we validate our theoretical findings using numerical simulation.\\



\end{abstract}

\tableofcontents

\section{Introduction}

The connection between the graph Laplacian and the Laplace–Beltrami operator is a central theme in geometric data analysis and has attracted substantial interest from both applied and theoretical communities. On the one hand, from a theoretical perspective, this relationship offers a discretization framework for studying differential operators on manifolds via point clouds or sampled data \cite{LiShiSun,HarlimSanzYang,YanJiangHarlim,JiaoYanHarlimLu}.
On the other hand, from an applied viewpoint, this connection underpins a wide array of algorithms in machine learning and signal processing, including manifold learning techniques such as Laplacian eigenmaps \cite{BelkinNiyogi2003,BelkinNiyogi2005} or Diffusion Maps \cite{CoifmanLafon2006, SingerWu}. Understanding the convergence of the graph Laplacian to the Laplace–Beltrami operator — notably the role of scaling, sampling density, and boundary behavior — remains crucial in ensuring that discrete approximations faithfully capture the geometry and analysis of the underlying continuous space. This paper aims at tackling this question on a class of smooth Riemannian manifolds having possibly singular boundary.\\

To describe our framework and main results, let us consider independent and identically distributed (i.i.d.) random variables  $\{X_i\}_{i \ge 1} \sim X$ with common law $\mathbb{P}_X$ on a metric measure space $(M,\di,\mu)$. For any $t>0$ and $n \in \mathbb{N}\backslash \{0\}$, define
$$
L_{n,t}f(x):=\frac{1}{nt^{d/2+1}}\sum_{i=1}^{n}e^{-\frac{\di(x,X_i)^2}{t}}(f(x)-f(X_i))
$$
for any $x \in M$ and $f$ in the space $\cC(M)$ of continous functions on $M$. Here $d$ is a positive integer. We call $L_{n,t}$ the intrinsic Graph Laplacian with $d$-dimensional Gaussian kernel associated with the family $\{X_1,\ldots,X_n\}$ and the parameter $t$. In case $M$ is a subset of some ambient Euclidean space $\setR^D$, then the intrinsic distance $\di$ might be replaced by the extrinsic Euclidean norm $\|\cdot\|$, in which case we call $L_{n,t}$ the extrinsic Graph Laplacian, which then writes as
$$
L_{n,t}f(x)=\frac{1}{nt^{d/2+1}}\sum_{i=1}^{n}e^{-\frac{\|x-X_i\|^2}{t}}(f(x)-f(X_i)).
$$

We define the (intrinsic/extrinsic) $d$-dimensional Gaussian operator at time $t$ as the expected value of $L_{n,t}$ with respect to $\mathbb{P}_X$, that is
$$L_tf(x):=\mathbb{E}_X\left[\frac{1}{t^{d/2+1}}e^{-\frac{\di(x,X)^2}{t}}(f(x)-f(X))\right]$$
for any $x \in X$ and $f \in \cC(M) \cap L^1(M,\mathbb{P}_X)$. It follows from the Strong Law of Large Numbers that
$$L_{n,t}f(x) \to L_tf(x) \qquad \text{almost surely as $n \to \infty$}$$
for any $x,f,t$ as above. Note that if $\mathbb{P}_X$ is absolutely continuous with respect to $\mu$ with Radon--Nikodym derivative $p \in L^1(M,\mu)$, then the Gaussian operator rewrites as
\[
L_t f(x) = \frac{1}{t^{d/2+1}} \int_M e^{-\frac{\di(x,y)^2}{t}} (f(x) - f(y)) p(y) \, d\mu(y).
\]
We classically say that $p$ is a density if it has unit $L^1$ norm.

After the work of various authors, e.g. \cite{Hein2005,BelkinNiyogi2005}, it is by now known that when $M$ is a smooth $d$-dimensional submanifold of an Euclidean space $\mathbb{R}^D$, for any sufficiently regular $f$ and $p$, the extrinsic Gaussian operator satisfies
\begin{equation}\label{eq:int}
    L_tf(x)\to -\frac{\pi^{d/2}}{4} \, \Delta_{M,p}f(x) \qquad \text{ as $t\to 0$}
\end{equation}
 at any interior point $x \in M$, where $\Delta_{M,p}$ is the weighted Laplace operator given by
\[
\Delta_{M,p}f(x) \df  p(x)\Delta_M{f}(x) + 2 \nabla{p}(x) \cdot \nabla{f}(x),
\]
with $\Delta_M$ the classical Laplace--Beltrami operator of $M$. In \cite{belkin+} (see also \cite{AnderssonAvelin}), this pointwise result for the extrinsic Gaussian operator was extended to three types of singular subsets of $\setR^D$, namely submanifolds with boundaries, intersections of them, and submanifolds with edge-type singularities. For instance, if $x$ belongs to the non-empty boundary of some smooth submanifold $M \subset \setR^D$, then 
\begin{equation}\label{eq:bound}
L_tf(x)  = -\frac{\pi^{(d-1)/2}}{2\sqrt{t}}p(x)\partial_\nu f(x)+ o\left(\frac{1}{\sqrt{t}}\right) \qquad \text{ as $t \downarrow 0$,}
\end{equation}
where $\partial_\nu f(x)$ is the inner normal derivative of $f$ at $x$. Note that the case of isolated interior singularities has been recently addressed in \cite{Pal}.

Our main results extend both \eqref{eq:int} and \eqref{eq:bound} in the setting of smooth Riemannian manifolds with kinks. These spaces, which we introduce in Section \ref{sec:manifolds}, might be understood as manifolds with possibly non-empty and irregular boundary. To distinguish these possibly non-smooth boundaries with classical smooth ones, we use the terminology \textit{border} : a smooth manifold with kinks $M$ thus divides into an interior part $int M$ and a border part $\partial M$. Manifolds with kinks notably include manifolds without boundary, manifolds with smooth boundary, and manifolds with corners as considered e.g.~in \cite{Cerf,Douady,DouadyHerault,Michor,Melrose,Joyce}. Other relevant spaces like cones, pyramids, even cusps, fall into the framework of manifolds with kinks while they are not manifolds with corners.

If $M$ is a $d$-dimensional manifold with kinks, then any $x \in M$ admits a tangent space $T_xM \simeq \setR^d$ defined as the usual space of smooth derivations at $x$. A key concept in our analysis is the notion of inward tangent cone $$I_xM \subset T_xM$$ which consists in all the directions emanating from $x$ and pointing towards the interior of $M$. In particular, $I_xM$ coincides with $T_xM$ if $x$ is an interior point, with a half-subspace of $T_xM$ if $x$ is a classical boundary point, and with a binary fraction of $T_xM$ if $x$ is a corner point.

We also introduce on manifolds with kinks the notions of $C^k$ functions for $k \in \mathbb{N}\cup\{\infty\}$, and tangent, cotangent and tensor bundles, see Section \ref{subsec:tangent} and \ref{subsec:metric}. We then define vector fields as the sections of the tangent bundle, differential forms as alternating sections of the cotangent bundle, Riemannian metrics as the sections of the symmetric, positive definite covariant 2-tensor bundle. The canonical Riemannian distance $\di$ and volume measure $\mathrm{vol}_g$ of a Riemannian metric $g$ are also introduced in a natural manner. Like for manifolds without/with boundary, we define the differential of order $k$ of a $\cC^k$ function as the suitable symmetric covariant $k$-tensor field which acts on contravariant $k$-tensor fields. Note that we shall use $Z^{(k)}$ to denote the contravariant $k$-tensor field $Z \otimes \ldots \otimes Z$ where $Z$ is a vector field. Any Riemannian metric $g$ on a manifold with kinks $M$ allows to specify the gradient and hessian of a $\cC^2$ function $f$, through the classical relations
\[
df(Z) = g(\nabla^g f,Z), \qquad d^{(2)}f(Z \otimes Z') = g([\mathrm{Hess}^g f] Z,Z').
\]
A Riemannian metric additionally identifies those tangent vectors at a point $x$ that have norm one, whose set we denote by $S^g_xM$. We define the inward tangent sphere
\[
S^gI_xM \df I_x M \cap S^g_xM
\]
which we endow with the $(d-1)$-dimensional Hausdorff measure $\sigma$ associated with the distance induced on $T_xM$ by the scalar product $g(x)$. Finally, we define the generalized normal vector
\[
v_g(x) \df  \int_{S^gI_xM} \theta \di \sigma(\theta).
\]

For the purpose of our analysis, we need minimal regularity on the boundary of the manifolds with kinks that we consider. Thus we work under Lipschitz and Continuously Directionally Differentiable (LCDD for short)
regularity, noting that this condition is satisfied by a large class of manifolds with kinks, e.g.~manifolds with boundaries and corners, and pyramids. We get refined results under Lipschitz and Twice Continuously Directionally Differentiable (LTCDD for short). Our approach also allows for cusps, which are not LCDD. We refer to Section \ref{sec:euclidean} for the precise definitions of L(T)CDD boundary points and cusps for local Euclidean models, and to \ref{sec:manifolds} for these definitions  in the manifold context.

For any $\ell \in \mathbb{N}$, set
    $$c_\ell = \frac{1}{2}\Gamma((\ell+1)/2)$$ where $\Gamma$ denotes the classical Gamma function. Our first main result concerning the intrinsic Gaussian operator writes as follows.

\begin{theorem}\label{th:1}
    Let $M$ be a smooth $d$-dimensional manifold with kinks endowed with a $\cC^2$ Riemannian metric $g$, and let $x \in M$ be either an interior point, an LCDD border point, or a cusp. Consider $\eta \in (0,1/2)$. Then the intrinsic Gaussian operator associated with a density $p \in \cC_{\ge 0}^2(M)$ satisfies
    \begin{align*}
    L_tf(x) & = -\frac{c_d}{\sqrt{t}} \bigg( p(x) \, \partial_{v_g(x)}f(x)  + \mathrm{Err}(t)\bigg) - c_{d+1} \bigg( p(x) A_gf(x) +  [p,f]_g(x) \bigg) + O(\sqrt{t}) + O(t^{-d}e^{-t^{2\eta-1}})
    \end{align*}
    as $t \downarrow 0$, for any $f \in \cC^3(M) \cap L^1(M,p\,\mathrm{vol}_g)$, $\mathrm{Err}(t) =o(1)$ as $t\downarrow 0$,
    and
    \begin{align*}
\partial_{v_g(x)} f(x) & \df \di_{x}f\left( v_g(x) \right),\\
        A_gf(x) & \df \frac{1}{2}\int_{S^gI_xM} \di_{x}^{(2)}f \left( \theta^{(2)} \right) \di \sigma(\theta),\\
        [p,f]_g(x) & \df \int_{S^gI_xM} \di_xf(\theta)\di_xp(\theta) \di \sigma(\theta).
    \end{align*}
    Moreover :

\begin{enumerate}
    \item If $x$ is an interior point, then as $t\downarrow 0$,
    \[
    L_tf(x) = -\frac{\pi^{d/2}}{4} \, \Delta_{M,p}f(x) + O(\sqrt{t}) + O(t^{-d}e^{-t^{2\eta-1}}).
    \]
    \item If $x$ is a $\cC^1$ boundary point, then as $t\downarrow 0$,
    \begin{equation}\label{eq:c1bound}
    L_tf(x) =  -\frac{\pi^{(d-1)/2}}{2\sqrt{t}}\bigg( p(x)\partial_\nu f(x)  + \mathrm{Err}(t)\bigg) - c_{d+1} \bigg( p(x) A_gf(x) +  [p,f]_g(x) \bigg) + O(\sqrt{t}) + O(t^{-d}e^{-t^{2\eta-1}})
    \end{equation}
    \item If $x$ is LTCDD, then for any small enough $t>0$,
    \begin{equation}\label{eq:refined}
    \mathrm{Err}(t) \le  \, \sqrt{t} \, C_d \, p(x) \, \|d_xf \|_{g_x}  SC(x) .
    \end{equation}
    where $C_d = 8 \Gamma(d+2)$ and $SC(x)$ is the total slicewise curvature of $\partial \Omega$ at $x$ defined in \eqref{eq:slicewise}. In particular, if the total slicewise curvature vector is zero, then $\mathrm{Err}(t) = 0$ for any small enough $t>0$.
\end{enumerate}
\end{theorem}

Our second main result deals with the extrinsic Gaussian operator. It writes in the context of submanifolds with kinks which we introduce in Section \ref{subsec:submanifolds}. In this setting, we can replace the inward tangent sphere $S^gI_x M$ with a suitable Bouligand tangent sphere, using the local chart around $x$ given by the orthogonal projection onto the translated tangent space $x + T_xM$. We refer to Definition \ref{def:open feasible direction cone} for the definition of Bouligand tangent sphere. For technical reasons, we obtain this result under an additional constraint on the power $\eta$ ; this might be lifted by means of a different approach.

\begin{theorem}\label{th:2}
  Let $M\subset \cl^D$ be a smooth $d$-dimensional submanifold with kinks endowed with the smooth Riemannian metric $g$ induced by the Euclidean scalar product of $\cl^D$, and $x\in M$ be either an interior point, an LCDD border point or a cusp. Let $\pi$ be the orthogonal projection of $\setR^D$ onto the translated tangent space $x + T_xM$. We identify the latter with $\setR^d$ and set $$\Omega \df \pi(\mathbb{B}^D_\epsilon(x) \cap M)$$ for $\epsilon>0$ sufficiently small. We let $S^B_{0_d}\Omega$ denote the Bouligand tangent sphere of $\Omega$ at $0_d$.
  
  Consider $\eta \in (1/6,1/2)$. Then the intrinsic Gaussian operator associated with a density $p \in \cC_{\ge 0}^2(M)$ satisfies
    \begin{align*}
    L_tf(x) & = \left[ -\frac{c_d}{\sqrt{t}} \left( \tilde{p}(0_d) \, \partial_{M}\tilde{f}(0_d)  + \mathrm{Err}(t)\right) - c_{d+1} \bigg( \tilde{p}(0_d) A_M\tilde{f}(0_d) +  [\tilde{p},\tilde{f}]_M(0_d) \bigg)\right](1+o(1))\\
    & \qquad \qquad \qquad \qquad \quad  \qquad\qquad \qquad \qquad \qquad \qquad \qquad    \quad + O(t^{1/6}) + O(t^{-d}e^{-t^{2\eta-1}})
    \end{align*}
    as $t \downarrow 0$, for any $f \in \cC^3(M) \cap L^1(M,p\,\mathrm{vol}_g)$, where $\tilde{f} \df f \circ \pi^{-1}$, $\tilde{p} \df p \circ \pi^{-1}$, and
    \[
    \mathrm{Err}(t) = o(1),
    \]
    \[
    \partial_{M} \tilde{f}(0_d) \df \di_{0_d}\tilde{f}\left( \int_{S^B_{0_d}\Omega} \theta \di \sigma(\theta) \right),
    \]
    \[
    A_M\tilde{f}(0_d) \df \frac{1}{2}\int_{S^B_{0_d}\Omega} \di_{x}^{(2)}\tilde{f} \left( \theta^{(2)} \right) \di \sigma(\theta), \qquad [\tilde{p},\tilde{f}]_M(0_d)  \df \int_{S^B_{0_d}\Omega} \di_x\tilde{f}(\theta)\di_x\tilde{p}(\theta) \di \sigma(\theta).
    \]
Moreover, the same conclusion as in Theorem \ref{th:1} is true when $x$ is an interior point or a LTCDD point.
\end{theorem}

It is worth pointing out that in the previous context, the inward tangent sector $I_xM$ coincides with the Bouligand tangent cone of $\Omega$ at $0_d$, so that Theorem \ref{th:2} matches up with Theorem \ref{th:1}. Let us also stress out that, while bounded geometry of the boundary was assumed in \cite{HAL2007}, we do not make any assumption of that kind here.

The recent preprint \cite{ART} provides a result that goes in the direction of Theorem \ref{th:2}, but with less details on the differential operators involved there. 

Let us also point out that, if $x$ is a corner of depth $k$, meaning that $M$ is locally modeled on
\[
\setR^d_k \df \setR^k_+ \times \setR^{d-k}
\]
in a neighborhood of $x$, then we obtain 
\begin{equation}\label{eq:Rdk}
I_xM \simeq \bigcap_{1 \le i \le k} \{ v \in T_xM : g(\partial_{i} f(x),v)\ge 0\} 
\end{equation}
where each $\partial_{i} f(x)$ corresponds to differentiation along a corresponding half-line direction of $\setR^n_k$, and $\di_x f (v_g(x)) = \sum_i \partial_{i} f (x)$.




Theorems \ref{th:1} and \ref{th:2} open the door to a natural Neumann boundary condition on (sub)manifolds with kinks, that would be $\partial_{v_g(x)} f(x)=0$ for any $x \in \partial M$. A natural problem to tackle after that is the spectral convergence of the Graph Laplacian with Gaussian kernel generated from finitely many samples $\{X_1,\ldots,X_n\}$ to the suitable Neumann Laplacian as sample size goes to infinity. We refer to \cite{BelkinNiyogi_conv,SingerWu2,TGHS} for related results on manifolds without and with boundary, see also \cite{DT} for a preliminary investigation with the symmetrized AMV kernel.\\

Our third main result tells us how to choose bandwidth $t$ as a function of the sample size $n$ so that the difference between the scaled graph Laplace operator $\sqrt{t} L_{n,t}$ and the first-order operator
\[
\mathcal{D}f(x) \df -c_d \, p(x) \, \partial_{v_g(x)}f(x) 
\]
identified in Theorem \ref{th:1} converges to zero in probability or almost surely. We refer to Definition \ref{def:alphasub} for the notion of $\alpha$-subexponentiality for a real-valued random variable, which encompasses boundedness, subexponentiality and subgaussianity.

\begin{theorem}\label{th:3}
Let $M$ be a smooth $d$-dimensional manifold with kinks endowed with a $\cC^2$ Riemannian metric $g$, and let $x \in M$ be either an interior point, an LCDD border point, or a cusp. Consider a random variable $X$ on $M$ with law having density $p \in \cC^2(M)$. Let $\{t_n\} \subset (0,+\infty)$ be such that $t_n \to 0$ as $n \to \infty$. For $f \in \cC^3(M)\cap L^1(M, p \, \mathrm{vol}_g)$, assume that $f(X)$ is $\alpha$-subexponential for some $\alpha \in (0,2]$.
\begin{enumerate}[label=(\roman*)]
    \item If \( \sqrt{n}\, t_n^{\frac{d+1}{2} } \to \infty \) as \( n \to \infty \), then
    \[
    \sqrt{t_n}L_{n, t_n} f(x) \xrightarrow{\mathbb{P}} \mathcal{D}f(x).
    \]

    \item If 
   $ \left(\sqrt{n}\, t_n^{\frac{d+1}{2}} \right)^{\alpha}/\ln(n)\to \infty$ as $n \to \infty$, then
    \[
    \sqrt{t_n} L_{n, t_n} f(x) \xrightarrow{\text{a.s.}} \mathcal{D}f(x).
    \]
\end{enumerate}
Moreover, if $\alpha \in (0,1]$, then $\sqrt{n}$ can be replaced by $n$ in the two previous asymptotic assumptions.
\end{theorem}

Of course, Theorem \ref{th:3} also holds when $M$ is a smooth submanifold with kinks of an Euclidean space $\setR^D$, in which case $\mathcal{D}$ must be replaced by the first-order operator obtained in Theorem \ref{th:2}.

Note the difference in the sufficient condition between \textit{(i) }and \textit{(ii)} above : the former doesn't depend on the tail decay rate $\alpha$ of $f(X),$ whereas the latter does.

We obtain Theorem \ref{th:3} in Section \ref{sec:concentration} through concentration estimates. Such estimates --- studied in e.g.~\cite{Hein2005,belkin+,AnderssonAvelin} --- quantify the error in approximating the Gaussian operator by the graph Laplace operator. Our results generalize this line of work in two key directions :

\begin{enumerate}[label=(\roman*)]
    \item We work in the setting of Riemannian manifolds with kinks, whose singularities can be significantly more severe than those of manifolds with boundary considered in \cite{belkin+}.
    \item Unlike most previous works, we do not assume that the function $f : M \to \mathbb{R}$ is bounded. In particular, our results hold on noncompact manifolds with finite total volume.
\end{enumerate}

\noindent
Regarding point (2), previous concentration bounds rely on Hoeffding's or Bernstein's inequalities for bounded random variables. In contrast, we employ more general forms of these inequalities that are applicable to subgaussian or subexponential random variables with appropriate tail decay. Let us mention that Theorem \ref{th:3} guided us to choose the bandwidth parameter $t>0$ correctly in the numerical experiments presented in Section \ref{sec:numerics}.

We conclude by pointing out that this paper deals with the so-called unnormalized Graph Laplacian only. The other two popular ones, namely normalized and random walk, can be investigated in a similar manner, but we leave this for future work.

\section{Euclidean models}\label{sec:euclidean}

Throughout the paper, we write $0_d$ for the origin of $\setR^d$, we let $\setB_r^d(x)$ stand for the Euclidean ball of radius $r$ centered at $x \in \setR^d$, and we write $\setB_r^d$ and $\setB^d$ instead of $\setB_r^d(0_d)$ and $\setB_1^d(0_d)$ respectively. We let $(e_1,\ldots,e_d)$ denote the canonical basis of $\setR^d$. For $y \in \setR^d$ we will often denote by $y'$ the $d-1$ first coordinates of $y$ in the canonical basis of $\setR^d$, and we let $y_d$ be the last one. We consider the open upper half space $\mathbb{H}^d:=\{y \in \setR^d : y_d > 0\}$ and the Euclidean unit sphere $\mathbb{S}^{d-1}\df \{y \in \setR^d : y_1^2 + \ldots + y_d^2 = 1\}$. We let $\sigma$ be the usual surface measure on $\mathbb{S}^{d-1}$ which coincides with the $(d-1)$-dimensional Hausdorff measure of the Euclidean distance. We also write $\mathcal{L}^d$ for the $d$-dimensional Lebesgue measure in $\mathbb{R}^d$.

For a subset $A$ of the Euclidean space  $\setR^d$, we let $\bar{A}$ be the closure of $A$ and $int(A)$ be its interior. We recall that a rigid motion of $\setR^d$ is a linear map $y \mapsto Ay+b$ where $A\in O(d)$ and $b\in \ddim$. The epigraph and strict epigraph of $\gamma : \setR^{d-1} \to \setR$ 
are defined, respectively, as
\[
epi(\gamma) \df \{ y \in \setR^d : y_d \ge \gamma(y')\} \quad \text{and} \quad \mathring{epi}(\gamma) \df \{ y \in \setR^d : y_d > \gamma(y')\}.
\]
When $\gamma$ is continuous, we obviously have
\begin{equation}\label{eq:epi=}
epi(\gamma) = \overline{\mathring{epi}(\gamma)}.
\end{equation}
The directional derivative of $\gamma$ at $x' \in \setR^{d-1}$ in the direction $v' \in \setR^{d-1}$ is defined as
\[
\gamma'(x';v') \df \lim_{t \downarrow 0} \frac{\gamma(x'+tv')-\gamma(x')}{t}
\]
whenever the limit exists. If this is the case, then the second directional derivative of $\gamma$ at $x'$ in the direction $v'$ is defined as
\[
\gamma''(x';v') \df \lim_{t \downarrow 0} \frac{\gamma(x'+tv')-\gamma(x')-t \gamma'(x';v')}{t^2/2}
\]
whenever it exists. Lastly, the characteristic function of a subset $A$ of a set $X$ will always be denoted by $1_A$.

\subsection{Boundary points}

For an open set $\Omega \subset \setR^d$ with a non-empty boundary $\partial \Omega$, we will call boundary points of $\Omega$ the elements in $\partial \Omega$. Inspired by \cite[p.626]{Evans}, we define the regularity of boundary points as follows. 

\begin{definition}\label{C0 boundary pt}
    Let $\Omega$ be an open subset of $\ddim$ with a non-empty boundary $\partial \Omega$, and $x \in \partial \Omega$. Consider $k \in \mathbb{N}\cup \{\infty\}$ and $\alpha \in (0,1]$.
    \begin{enumerate}[label=(\roman*)]
        \item We say that $x$ is a $\cC^k$ boundary point of $\Omega$ if there exist $\delta>0$ and $\gamma \in \cC^k(\setR^{d-1})$ such that, up to applying a rigid motion,
        \begin{equation}\label{eq:boundary_points}\tag{*}
        \begin{cases}
        \bar{\Omega} \cap  \setB_\delta^d(x)  = epi(\gamma)  \cap \setB_\delta^d(x), \nonumber\\
        \Omega \cap \setB_\delta^d(x)  = \mathring{epi}(\gamma)  \cap \setB_\delta^d(x).
        \end{cases}
        \end{equation}

    \item We say that $x$ is a $\cC^{k,\alpha}$ boundary point of $\Omega$ if it satisfies \eqref{eq:boundary_points} for some $\gamma \in \cC^{k,\alpha}(\setR^{d-1})$, up to a rigid motion.
    \item We say that $x$ is a Directionally Differentiable (DD for brevity) boundary point if it is a $\cC^0$ boundary point with a function $\gamma$ as above whose directional derivative $\gamma'(x';v')$ exists for any $v' \in \mathbb{R}^{d-1}$, where we identify $x$ with its image $(x',x_d)$ through the rigid motion from (1).
    \item We say that $x$ is a Continuously Directionally Differentiable (CDD for brevity) boundary point if it is a DD point with a function $\gamma$ as above whose directional derivative $\gamma'(x';v')$ exists for any $v' \in \mathbb{R}^{d-1}$ and is continuous with respect to $v'$. We additionally say that $x$ is TCDD if the second directional derivative $\gamma''(x';v')$ exists for any $v' \in \mathbb{R}^{d-1}$ and is continuous with respect to $v'$.
    \end{enumerate}
\end{definition}

As customary, we may write smooth instead of $\cC^\infty$, and Lipschitz instead of $\cC^{0,1}$. We will essentially work with boundary points that are $\cC^0$, $\cC^1$, Lipschitz, or CDD. The combination of Lipschitz and CDD (resp.~TCDD) shall be abbreviated to LCDD (resp.~LTCDD). 

\begin{remark} It should be noted that a boundary point may not be $\cC^0$. For instance, consider $\Om:=\{(x,y) \in (0,+\infty)\times \setR: y < \sin(1/x)\}.$ Then the boundary $\partial \Omega$ is the union of $\{0\} \times (-\infty,1]$ and $\{y = \sin(1/x)\}$, which cannot be written as the graph of a continuous function locally around $0_2$, see Figure \ref{Im1}.

\begin{figure}[!h] \centering
\includegraphics[scale=0.5]{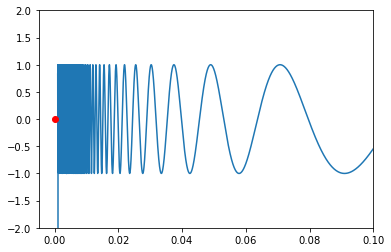}
\caption{$0_2$ is not a $\cC^0$ boundary point of $\Omega$}
\label{Im1}
\end{figure}
\end{remark}

\begin{remark}
The $\cC^0$ regularity of a boundary point is trivially preserved by local homeomorphisms. More precisely, for any $i \in \{1,2\}$, let $x_i$ be a boundary point of an open set $\Omega_i \subset \setR^d$ with a non-empty boundary. Assume that there exist open neighborhoods $V_1, V_2$ in $\setR^d$ of $x_1,x_2$ respectively, and a homeomorphism $\Phi : V_1 \cap \bar{\Omega}_1 \xrightarrow{\sim} V_2 \cap \bar{\Omega}_2$ such that $x_2 = \Phi(x_1)$. If $x_1$ is $\cC^0$, let $\gamma_1$ be the continuous map  such that \eqref{eq:boundary_points} holds around $x_1$ up to rigid motion. Then the composition of $\gamma_1$ with the homeomorphism $\Phi$ yields a continuous map $\gamma_2 \df \Phi \circ \gamma_1$ which ensures that $x_2$ is $\cC^0$.
\end{remark}

The next proposition shows that, from a topological point of view, there is no difference between a $\cC^0$ boundary point and a boundary point of the model upper half space $\mathbb{H}^d$.

\begin{proposition}\label{Small ball near the border are homeomorphic}

    Let $\Omega$ be an open subset of $\ddim$ admitting a $\cC^0$ boundary point $x \in \partial \Omega$. Then there exists $\delta>0$ such that $\bar{\Omega} \cap  \setB^d_\delta(x)$ is homeomorphic to $\bar{\mathbb{H}}^d\cap \setB^d_\delta.$ 
\end{proposition}

\begin{proof}
    Let $\delta$ and $\gamma$ be such that \eqref{eq:boundary_points} holds up to a rigid motion. Then $\Phi(y) \df  x + (y',\gamma(y')+y_{d})$ defines a homeomorphism $
    \bar{\mathbb{H}}^d\cap \setB^d_\delta \to epi(\gamma) \cap \setB^d_\delta(x)$ with inverse $\Phi^{-1}(\xi) \df (\xi', \xi_d-\gamma(\xi'))-x$. 
\end{proof}

\subsection{Diffeomorphisms and inward sectors}\label{subsec:diffeo_inward}

For $k$ a positive integer, we shall use the following notion of $\cC^k$ diffeomorphism between possibly non-open subsets of $\setR^d$, see e.g.~\cite[Section 1.5]{Melrose} for details.

\begin{definition}\label{Diffeomorphism of sets in ddim}
We say that a map $\varphi:A\to B$ between subsets $A,B\subset \ddim$ is a $\cC^k$ diffeomorphism if there exist open subsets $\tilde{A},\tilde{B} \subset \ddim$ and a $\cC^k$ diffeomorphism $\tilde{\varphi}:\tilde{A}\to \tilde{B}$ such that  such that $A \subset \tilde{A}$, $B \subset \tilde{B}$ and $\tilde{\varphi}|_A=\varphi.$
\end{definition}

We will mostly use this definition in the case where $A$ and $B$ are the closures of open subsets of $\setR^d$ with a non-empty boundary. In this regard, let us introduce a useful definition.

\begin{definition}
     Let $\Omega$ be an open subset of $\ddim$ with a non-empty boundary $\partial \Omega$. For any $x \in \bar{\Omega}$, we define the inward sector of $\Omega$ at $x$ as
     \[
     I_x\Omega \df \{c'(0) \in \setR^d : c \in \cC^1(I,\setR^d) \text{ such that $I = (-\epsilon,\epsilon)$ for some $\epsilon>0$, $c(0)=x$ and $c([0,\epsilon)) \subset \Omega$} \}.
     \]
\end{definition}

It is easily seen that the inward tangent sector at an interior point is $\setR^d$, and that it is a half space at a $\cC^1$ boundary point. Moreover, the inward tangent sector at the origin of $\setR^d_k$ is $\setR^d_k$ itself. 

Let us provide relevant properties of diffeomorphisms between closures of open sets.

\begin{lemma}\label{lem:prop_diffeo}
Let $ U, V\subset \ddim$ be open subsets and $\varphi:\bar{U}\to \bar{V}$ a $\cC^1$ diffeomorphism. Then $\varphi(U)=V$ and $\varphi(\partial U)=\partial V.$ Moreover, for any $x \in \partial U$, the differential $d_x\varphi : \setR^d \to \setR^d$ is a well-defined linear map which maps  $I_x U$ to $I_{\varphi(x)} V.$
\end{lemma}

\begin{proof}
 Let $\tilde{\varphi}$ be a $\cC^1$ diffeomorphism from an open set $\tilde{U} \supset \bar{U}$ onto another open set $\tilde{V} \supset \bar{V}$ extending $\varphi$. Then $\varphi(\partial U) = \tilde{\varphi}(\partial U) = \partial V$ where the first equality holds because $\tilde{\varphi}$ coincides with $\varphi$ on $\bar{U}$ and the second one is ensured by the continuity of $\tilde{\varphi}$ and $\tilde{\varphi}^{-1}$. Then $\varphi(U)= \tilde{\varphi}(\bar{U}\backslash \partial U) = \tilde{\varphi}(\bar{U})\backslash \tilde{\varphi}(\partial U)  = \tilde{\varphi}(\bar{V})\backslash \tilde{\varphi}(\partial V)  =\varphi(V)$.  For the well-posedness of $d_x\varphi$, note that any $\cC^1$ diffeomorphism $\tilde{\varphi} : \tilde{U} \to \tilde{V}$ extending $\varphi$ is $\cC^1$, hence the map $d\tilde{\varphi} : y \mapsto d_y\tilde{\varphi}$ is continuous on $\tilde{U}$ and coincides with $d\varphi$ on $U$; as a consequence, $d \varphi$ continuously extends to $\partial U$ in a unique way. The last point follows from the chain rule applied  to any map of the form $\tilde{\varphi} \circ c$ where $c$ defines an element $c'(0)$ of $I_x U$ and $\tilde{\varphi}$ is an extension of $\varphi$.
\end{proof}

\begin{remark}
    The first two equalities in the previous lemma do not hold when $\varphi$ is a diffeomorphism between $U$ and $V$ only. For instance, the open sets $\mathbb{B}^2$ and  $\mathbb{B}^2 \backslash ([0,1)\times \{0\})$ are biholomorphic via the Riemann mapping theorem, but their boundaries are not homeomorphic, see Figure \ref{Im8}. It should be pointed out that, in this case, the diffeomorphism given by the Riemann map does not extend to the boundary.

\begin{figure}[!h] \centering
\includegraphics[scale=0.4]{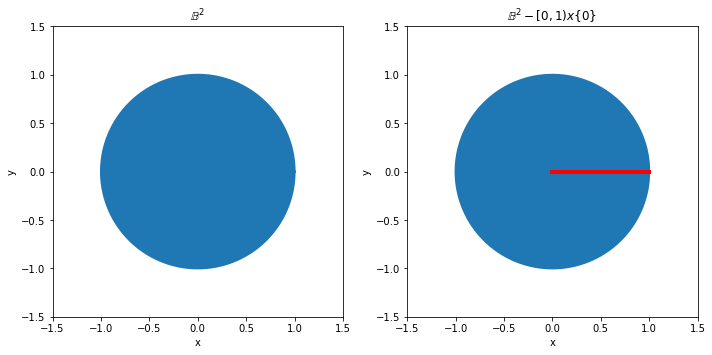}
\caption{Two diffeomorphic open sets whose boundaries are not diffeomorphic}
\label{Im8}
\end{figure}
    
\end{remark}

\subsection{Boundary cones of open Euclidean subsets}\label{sec:boundary_cones}

In this section, we provide classical notions from convex analysis, see e.g.~\cite{RW} for more details. We recall that $C\subset \ddim$ is called a cone if it is invariant under multiplication by a positive number, i.e.~if $\lambda v \in C$ for any $v\in C$ and $\lambda > 0.$ 

\begin{definition}\label{def:open feasible direction cone}
Consider $A \subset \ddim$ and $a \in A$.
\begin{enumerate}[label=(\roman*)]
    \item The Bouligand tangent cone of $A$ at $a$ is defined as $$T^B_aA:=\{0_d\}\cup \left\{v\in \ddim \backslash \{0_d\}: \text{there exists $\{a_n\}\subset A$ such that $ a_n \to a$  and }  \frac{a_n-a}{\norm{a_n-a}} \stackrel{}{\to} \frac{v}{\norm{v}} \right\}.$$
    \item The Bouligand tangent sphere of $A$ at $a$ is defined as $S^B_aA:=T^B_aA\cap \mathbb{S}^{d-1}.$
    \item The feasible direction cone of $A$ at $a$ is defined as $$F_a(A):=\{v\in \ddim: \text{there exists $t_v>0$ such that $a+tv \in A$ for any $0 \le t < t_v$}\}.$$
    
    \item The open feasible direction cone of $A$ at $a$ is defined as
    $$
    \tilde{F}_a(A):= \{v\in \ddim:   \text{there exists $t_v>0$ such that $a+tv \in int(A)$ for any $0 < t < t_v$}\}.
    $$
    \end{enumerate}
\end{definition}

    The Bouligand tangent cone is also known as tangent cone or contingent cone in the literature. It is a closed set. One can easily check that our definition is equivalent to the one given in \cite[Def.~6.1]{RW}. Note also that $$\tilde{F}_a(A) 
 \subset F_a(A) \subset T^B_aA$$ 
 with possibly strict inclusion : for instance, $\tilde{F}_{0_d}(\bar{\mathbb{H}}^d) =  \mathbb{H}^d \varsubsetneq \bar{\mathbb{H}}^d
 = F_{0_d}(\bar{\mathbb{H}}^d)$. However, if $x$ belongs to an open set $\Omega \subset \setR^d$, then
 \begin{equation}\label{eq:interiorF}
     \tilde{F}_x(\bar{\Omega}) 
= F_x(\bar{\Omega}) = T^B_x\bar{\Omega} = \setR^d.
 \end{equation}

As for boundary points, the following preliminary result holds.

\begin{lemma}\label{prop:open feasible direction cone non-empty}
    Let $x$ be a $\cC^0$ boundary point of an open set $\Omega \subset \setR^d$ with a non-empty boundary. Then $$\tilde{F}_x({\bar{\Omega}})\ne \emptyset.$$
\end{lemma}

\begin{proof}
    Let $\delta$ and $\gamma$ be such that \eqref{eq:boundary_points} holds up to a rigid motion. Since $x \in \partial \Omega$, we have $x_d = \gamma(x')$ so that $x_d + t > \gamma(x')$ for any $t \in (0,\delta)$. Thus $x+te_d \in \Omega \cap \setB_\delta(x) \subset \Omega$ for any such a $t$, so that $e_d\in \tilde{F}_x(\bar{\Omega})$.
\end{proof}

For boundary points with additional regularity, the Bouligand tangent cone satisfies the following properties.

\begin{proposition}\label{lem:tangent cones as epigraphs}
    Let $x$ be a LCDD boundary point of an open set $\Omega \subset \setR^d$ with a non-empty boundary, and let $\gamma$ be such that \eqref{eq:boundary_points} holds up to a rigid motion. We identify $x$ with its image $(x',x_d)$ through this rigid motion. Then 
    \begin{equation}\label{eq:epi}
    T^B_x{\bar{\Omega}}=epi(\gamma'(x';\cdot)).
    \end{equation} 
    Moreover, \begin{equation}\label{eq:epi2}  T^B_x{\bar{\Omega}} = \overline{\tilde{F}_x(\bar{\Omega})}, \qquad int(T^B_x{\bar{\Omega}})=int({ \tilde{F}_x(\bar{\Omega})  }), \qquad 
\partial(T^B_x{\bar{\Omega}})=\partial({ \tilde{F}_x(\bar{\Omega})  }).\end{equation}
Lastly,
\begin{equation}\label{eq:epi3}
    \mathcal{L}^d(\partial(T^B_x{\bar{\Omega}})) = 0.
\end{equation}
\end{proposition}

\begin{proof}
    With no loss of generality, we assume that $x =0_d$. Set $g(v')\df \gamma'(0_{d-1},v')$ for any $v' \in \setR^{d-1}$.
    
    Consider $v \in epi(g)$. Then $g(v')<v_d$. For any $\epsilon_n \downarrow 0$,
    \begin{align*}
    \gamma(\epsilon_n v') = g(\epsilon_n v') + o(\epsilon_n \|v'\|) = \epsilon_n g(v') + o(\epsilon_n \|v'\|) & \le \epsilon_n v_d + o(\epsilon_n \|v'\|)\\
    & =\epsilon_n(v_d + o(\|v'\|)) ) =: \tau_n,
    \end{align*}
    so that each $x_n \df  (\epsilon_n v', \tau_n)$ belongs to $epi(g)$ and the sequence $(x_n)$ satisfies $x_n \to 0_d$ with $$\frac{x_n}{\|x_n\|} = \frac{(v',v_d+o(\|v'\|))}{\|(v',v_d+o(\|v'\|))\|} \to \frac{v}{\|v\|} \, \cdot $$
    Thus $v \in T_{0_d}^B\bar{\Omega}$. This shows the reverse inclusion in \eqref{eq:epi}.

    Consider $v\in T^B_{0_d}{\bar{\Omega}}$ and $\{x_n\} \subset \bar{\Omega}$ such that $x_n \to 0_d$ and $x_n/\|x_n\| \to v/\|v\|$. By \eqref{eq:boundary_points}, we can assume that $\{x_n\} \subset epi(\gamma)$. Set $\epsilon_n \df \|x_n\|/\|v\|$ for any $n$. Then
    \[
    \gamma(\epsilon_n v') = g(\epsilon_n v') + o(\epsilon_n)
    \]
    by definition of $g$. Moreover, since $\gamma$ is Lipschitz, there exists $L \ge 1$ such that
    \[
    |\gamma(x_n') - \gamma(\epsilon_n v')| \le L \|x_n' - \epsilon_n v'\| = L \left|\frac{x_n'\|v\|}{\|x_n\|} - v'\right| \epsilon_n = o(\epsilon_n).
    \]
    As a consequence,
    \[
    \delta_n \df |g(\epsilon_n v') - \gamma(x_n')| = o(\epsilon_n).
    \]
    Thus $\epsilon_n g(v') = g(\epsilon_n v') \le \gamma(x_n') + \delta_n \le (x_n)_d+ \delta_n$ because $x_n \in epi(\gamma)$. Then
    \[
    g(v') \le \frac{(x_n)_d+ \delta_n}{\epsilon_n} = \frac{\|v\|(x_n)_d}{\|x_n\|} + o(1)  \to v_d.
    \]
     This shows the direct inclusion in \eqref{eq:epi}. 

      For clarity of the exposition, in the rest of the proof we do not assume $x=0_d$ anymore. Let us establish the first equality in \eqref{eq:epi2}. Using successively \eqref{eq:epi} and \eqref{eq:epi=} (the latter being available because $x$ is CDD), we can write:
\begin{align*}
T^B_x\bar{\Omega} &= \overline{\mathring{epi}(g)}.
\end{align*}
But
\begin{align*}
    \mathring{epi}(g) &  = \left\{v \in \setR^d : v_d > \lim_{t \downarrow 0} \frac{\gamma(x' + t v') - \gamma(x')}{t} \right\}\\
& \subset \left\{ v \in \setR^d : \exists \, t_v> 0\,  \text{ s.t. } v_d > \frac{\gamma(x' + t v') - \gamma(x')}{t}  \, \forall \,  t \in (0,t_v) \right\}\\
&= \left\{v \in \setR^d : \exists \, t_v> 0\,  \text{ s.t. }  x_d + t v_d > \gamma(x' + t v') \ \forall\,  t \in (0,t_v) \right\}  \quad  (\text{using that } x_d=\gamma(x')) \\
&= \tilde{F}_x(\bar{\Omega}).
\end{align*}

\nd This shows $T^B_x\bar{\Omega}  \subset \overline{\tilde{F}_x(\bar{\Omega})}.$ The converse inclusion is obvious since the Bouligand tangent cone is closed and contains $\tilde{F}_x(\bar{\Omega})$.

The third equality in \eqref{eq:epi2} is an obvious consequence of the first and second ones. So we are left with proving the second one. Since $\tilde{F}_x(\bar{\Omega}) \subset T^B_x{\bar{\Omega}}$ we clearly have $ int(\tilde{F}_x(\bar{\Omega})) \subset int(T^B_x{\bar{\Omega}}),$ hence it is enough to show the converse inclusion $int(T^B_x{\bar{\Omega}}) \subset int(\tilde{F}_x(\bar{\Omega}))$. By \eqref{eq:epi} and the continuity of $g$, this amounts to showing that $\mathring{epi}(g) \subset int(\tilde{F}_x(\bar{\Omega}))$. Consider $v \in \mathring{epi}(g)$. Then there exists $\eta>0$ such that $g(w')<w_d$ for any $w \in \setB^d_\eta(v)$. By definition of $g$, for any such a $w$, there exists $t_w>0$ such that for any $t \in (0,t_w)$,
\[
\frac{\gamma(x'+tw')-\gamma(x')}{t} < w_d.
\]
Since $\gamma(x') = x_d$ the previous rewrites as
\[
\gamma(x'+tw') < x_d + tw_d.
\]
Thus $w \in \tilde{F}_x(\bar{\Omega})$. This means that $\setB^{d}_\eta(v) \subset \tilde{F}_x(\bar{\Omega})$, hence $v \in int(\tilde{F}_x(\bar{\Omega}))$ as claimed.

To conclude, let us explain why \eqref{eq:epi3} holds. By \eqref{eq:epi}, the set $\partial(T^B_x{\bar{\Omega}})$ coincides with the graph of $g$. Since the latter is continuous, its graph has $d$-dimensional Lebesgue measure equal to $0$.

\end{proof}

\begin{remark}\label{rem:feasible directions don't correspond in general under diffeomorphisms} 
It should be noted that feasible direction cones are not preserved by diffeomorphisms. Indeed, consider the closed subsets $A \df \{1 \le x \le \sqrt{y}\}$ and $B \df \{1\le x/2+1/2\le y\}$ of $(0,+\infty)^2$, and $\Phi : A \to B$  mapping $(x,y)$ to $(x,y^2)$. Then $\Phi$ is a diffeomorphism in the sense of Definition \ref{Diffeomorphism of sets in ddim} since it trivially extends to a diffeomorphism of $(0,+\infty)$ onto itself. But $F_{(1,1)}(A) = B \backslash \{x = y\}$ is not closed while $F_{(1,1)}(B) = B$ is, so the linear isomorphism $\di_{(1,1)}\Phi$ cannot map $F_{(1,1)}(A)$ to $F_{(1,1)}(B)$. See Figure \ref{boundaries}.

\begin{figure}[!h] \centering \begin{subcaptionblock}{.4\textwidth} \centering \includegraphics[scale=0.3]{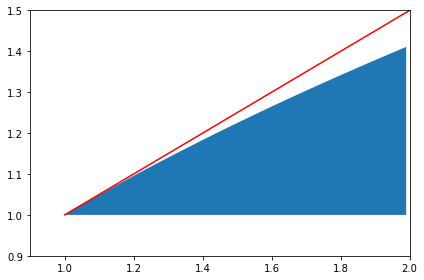} \caption{} \end{subcaptionblock}
\begin{subcaptionblock}{.4\textwidth} \centering \includegraphics[scale=0.3]{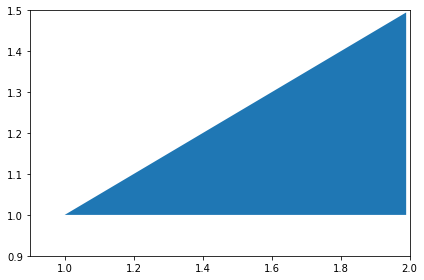} \caption{}\end{subcaptionblock}
\caption{Sets $A$ and $B$ with $F_{(1,1)}(B)\backslash F_{(1,1)}(A)=B \cap \{x=y\}$ in red}\label{boundaries} \end{figure}

\end{remark}

The following proposition relates the inward tangent sector and the Bouligand tangent cone at $\cC^0$ boundary points.

\begin{proposition}\label{prop:at boundary of open sets, inward sector equals tangent cone} 
    Let $\Omega\subset \ddim$ be open with a non-empty boundary, and $x\in \partial{\Omega}$ be a $\cC^0$ boundary point. Then $$I_x\bar{\Omega}\subset T^B_x\bar{\Omega}.$$ Moreover, if $x$ is LCDD, then $$\mathcal{L}^d(T^B_x\bar{\Omega} \backslash I_x\bar{\Omega})=0.$$
\end{proposition}

\begin{proof}

Let us first show $\subset.$ Take $c'(0) \in I_x \bar{\Omega}$ where $c \in \cC^1(I,\setR^d)$ is such that $I = (-\epsilon, \epsilon)$ for some $\epsilon>0$, $c(0)=x$, and $c([0,\epsilon)]) \subset \bar{\Omega}$. Set $x_n:=c(1/n)$ for any $n \in \mathbb{N}$ large enough. By Taylor expansion, $c(t)=c(0) + tc'(0)+ o(t)$ as $t \downarrow 0$, hence  $\frac{x_n-x}{\norm{x_n-x}}\to \frac{c'(0)}{\norm{c'(0)}}.$ This yields $c'(0) \in T^B_x\bar{\Omega}$ as desired.

Let us prove the second statement dealing with the LCDD case. By \eqref{eq:epi2}, if $v \in int(T^B_x\bar{\Omega}) = int(\tilde{F}_x(\Omega))$, then there exists $t_v>0$ such that $x+tv \in \bar{\Omega}$ for any $t \in (0,t_v)$, so $v = c'(0) \in I_x \bar{\Omega}$ with $c(t) \df x+tv$ for any $t \in (-t_v,t_v)$. This implies that $int(T^B_x\bar{\Omega}) \subset I_x \bar{\Omega}$. But we know from \eqref{eq:epi3} that $\mathcal{L}^d(\partial T^B_x\bar{\Omega})=0$. This gives the conclusion.
\end{proof}

Before closing this subsection, let us explain how the Bouligand tangent cone provides us with a convenient way to define cusps.

\begin{definition}
    Let $\Omega\subset \ddim$ be open with a non-empty boundary, and $x\in \partial{\Omega}$ be a $\cC^0$ boundary point. We say that $x$ is a cusp if $$ \mathcal{L}^d( T_x^B \bar{\Omega}) = 0.$$
\end{definition}

\begin{remark}
    This definition also applies to the general case of a subset $A \subset \setR^d$ : a cusp would be a point $x \in A$ such that $\mathcal{L}^d( T_x^B A) = 0$.
\end{remark}

\begin{examples}
    According to the previous definition, the origin $0_3$ is a cusp for both
    \[
    A \df  \{(x,y,z) \in \setR^3 : z > (x^2 + y^2)^{1/4}\} \qquad \text{and} \qquad B \df \{(x,y,z) \in \setR^3 : z > \sqrt{|x|}\},
    \]
  see Figure \ref{cusps}. Indeed, the half-line $T_{0_3}^B\bar{A} = \{(x,0,z) : z \ge 0\}$ and the half-space $T_{0_3}^B \bar{B} = \{(0,0,z) : x \in \setR, z \ge 0\}$ are both $\mathcal{L}^3$-negligible.

\begin{figure}[!h] \centering \begin{subcaptionblock}{.4\textwidth} \centering \includegraphics[scale=0.5]{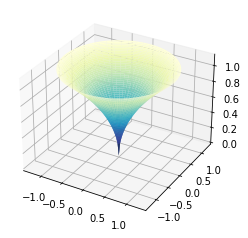} \caption{} \end{subcaptionblock}
\begin{subcaptionblock}{.4\textwidth} \centering \includegraphics[scale=0.5]{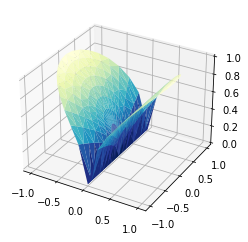} \caption{}\end{subcaptionblock}
\caption{Sets $A$ and $B$ both have $0_3$ as a cusp}\label{cusps}\end{figure}
\end{examples}

\subsection{Non-fluctuating boundary and local blow-ups}

Consider an open subset $\Omega\subset \ddim$ with a non-empty boundary, and a boundary point $x\in \partial \Omega.$
For any $t>0$, set $$\Omega_{x,t}:=\frac{\Omega-x}{{t}} \, \cdot$$  We are interested in finding the limit, in the sense of convergence of characteristic functions, of the sets $\Omega_{x,t}$ as $t\to 0$. When this limit exists, we call it the blow up of $\Omega$ at $x.$ With this in mind, we introduce the following.

\begin{definition}\label{def:domain with non-fluctuating boundary}
Consider an open subset $\Omega\subset \ddim$ with a non-empty boundary, and $x\in \partial \Omega.$
\begin{enumerate}[label=(\roman*)]
    \item We say that $v \in T_x^B\bar{\Omega} \backslash \tilde{F}_x\Omega$ is a non-fluctuating direction at $x$ if the following holds : for any $t_n\to 0$ such that $x+t_nv\notin \Omega$ for any $n$ there exists $t_v>0$ such that $x+tv\notin \Omega$ for any $t \in (0,t_v)$.
    Otherwise we say that $v$ is a fluctuating direction at $x$.
    \item We say that $\Omega$ has (a.e.) non-fluctuating boundary at $x$ if any ($\sigma$-a.a.) $v \in S_x^B\bar{\Omega}$ is a non-fluctuating direction. Otherwise we say that $\Omega$ has a fluctuating boundary at $x$.
\end{enumerate}

\end{definition}

\nd Roughly speaking, a domain with non-fluctuating boundary at a boundary point $x$ is so that whenever there exist points on a line spanned by $v$ arbitrary close to $x$ that stay outside of $\Omega,$ then there must be a segment of that line that also remains outside $\Omega.$ This explains the terms fluctuating and non-fluctuating.

\begin{examples}
    For any $k \ge 0$, consider $\Omega_k:=\{(x,y)\in \cl^2: y < x^k \sin (1/x), 0<x<1\}.$ For $k=0,1$, the set $\Om_k$ has a fluctuating boundary at the boundary point $0_2,$ however for $k\ge 2$, the sole fluctuating direction of $\Om_k$ at $0_2$ is $(1,0)$, and thus the set has a.e.~non-fluctuating boundary, see Figure \ref{fluctu}.

\begin{figure}[!h] \centering \begin{subcaptionblock}{.4\textwidth} \centering \includegraphics[scale=0.45]{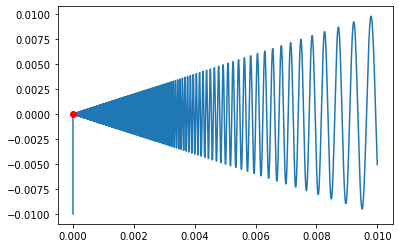} \caption{Fluctuating $\partial \Omega_1$} \end{subcaptionblock}
\begin{subcaptionblock}{.4\textwidth} \centering \includegraphics[scale=0.45]{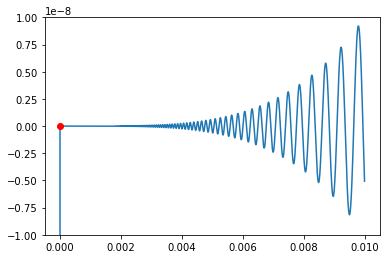} \caption{Non-fluctuating $\partial \Omega_4$}\end{subcaptionblock}
\caption{Fluctuating and non-fluctuating boundaries}\label{fluctu}
\end{figure}

\end{examples}

We shall use the following lemma.

\begin{lemma}\label{lem:general}
Let $\Omega\subset \ddim$ be open with a non-empty boundary, and $x \in \partial \Omega$. For any non-fluctuating direction $z$ at $x$,
  \[
 1_{\Omega_{x,t}}(z) \to 1_{\tilde{F}_x(\Omega)}(z) \qquad \text{as } t \downarrow 0.
\]
\end{lemma}

\begin{proof}
If $1_{\tilde{F}_x(\Omega)}(z) = 1$, then there exists $t_0 > 0$ such that for any $t \in (0,t_0)$ we have $x + tz \in \Omega$. The latter is obviously equivalent to $z \in \Omega_{x,t}$, hence we get $1_{\Omega_{x,t}}(z) = 1$ for any $t \in (0,t_0)$, thus $1_{\Omega_{x,t}}(z) \to 1$ as $t \downarrow 0$. Now if $1_{\tilde{F}_x(\Omega)}(z)=0$, then $z \notin \tilde{F}_x(\Omega)$ which implies that for any positive integer $n$ there exists $t_n \in (0, 1/n)$ such that $x+t_nz \notin  \Omega$ for any $n$. By definition of a non-fluctuating direction, we obtain that $x+sz \notin \Omega$ for any $s \in (0,t_z)$ for some $t_z>0$. Then $z \notin \Omega_{x,s}$ for any such a $s$, hence $1_{\Omega_{x,s}}(z) \to 0 $ as $s \to 0$.
\end{proof}

\nd Now we focus on particular boundary points.

\begin{lemma}\label{lem:negligible}
    Let $\Omega\subset \ddim$ be open with a non-empty boundary, and $x \in \partial \Omega$ be LCDD. Then the fluctuating directions at $x$ form a $\mathcal{L}^d$-negligible set.
\end{lemma}

\begin{proof}
     It follows from the definition of fluctuating direction that the set of such directions at $x$ belongs to $T^B_x\bar{\Omega}\setminus \tilde{F}_x(\Omega)$. But
    \begin{equation*}
        T^B_x\bar{\Omega}\setminus \tilde{F}_x(\Omega) \subset T^B_x\bar{\Omega}\setminus int(\tilde{F}_x(\Omega))
        = T^B_x\bar{\Omega}\setminus int(T^B_x\bar{\Omega})
        = \partial T^B_x\bar{\Omega}
    \end{equation*}
    where the first equality follows from the second equality in \eqref{eq:epi2}. Then \eqref{eq:epi3} gives the desired result. 
\end{proof}

We are now in a position to state and prove the following important property.

\begin{proposition}\label{prop:gen}
    Let $\Omega\subset \ddim$ be open with a non-empty boundary, and $x \in \partial \Omega$ be either an interior point, an LCDD border point, or a cusp. Then:
     \[
1_{\Omega_{x,t}}(z) \to 1_{T^B_x\bar{\Omega}}(z) \quad \text{ as $t \downarrow 0$,}  \quad \text{ for $\mathcal{L}^d$-a.e.~}z \in \mathbb{R}^d.
   \]
\end{proposition}

\begin{proof}
    Assume first that $x$ is interior. Consider $r>0$ such that $B_r(x) \subset \Omega$. Then for any $t>0$ the blow-up ${B_{r}(x)}_{x,t}$ is contained in $\Omega_{x,t}$. Thus 
    \[
    1_{\Omega_{x,t}} \ge 1_{{B_{r}(x)}_{x,t}} \stackrel{t \downarrow 0}{\longrightarrow} 1_{\setR^d} = 1_{T_x^B(\bar{\Omega})},
    \]
    so $1_{\Omega_{x,t}} \to 1$.
    
    Assume now that $x$ is LCDD. From Lemma \ref{lem:general} and Lemma \ref{lem:negligible} we get that $1_{\Omega_{x,t}}(z) \to 1_{\tilde{F}_x(\Omega)}(z)$ as $t \downarrow 0$ for $\mathcal{L}^d$-a.e.~$z \in \mathbb{R}^d$, and \eqref{eq:epi3} yields that $1_{\tilde{F}_x(\Omega)}(z) = 1_{T^B_x\bar{\Omega}}(z)$ for $\mathcal{L}^d$-a.e.~$z \in \mathbb{R}^d$.

If $x$ is a cusp, then we must prove that $1_{\Omega_{x,t}}(z) \to 0$ as $t \downarrow 0$ for $\mathcal{L}^d$-a.e.~$z \in \mathbb{R}^d$. For any $z \in \setR^d$ and $t>0$ :
\[
x + tz \in \Omega \iff z \in \Omega_{x,t} \iff 1_{\Omega_{x,t}}(z) = 1,
\]
so that $\tilde{F}_x(\bar{\Omega})$ rewrites as $\{z \in \setR^d : \exists \, t_z>0, \forall t \in (0,t_z), 1_{\Omega_{x,t}}(z) = 1\}$. Since $\tilde{F}_x(\bar{\Omega}) \subset T_x^B(\bar{\Omega})$, we know by definition of a cusp that $\mathcal{L}^d(\tilde{F}_x(\bar{\Omega}))=0$. Thus $\mathcal{L}^d$-a.e.~$z \in \setR^d$ there exists $t_z>0$ such that  $1_{\Omega_{x,t}}(z) = 0$ for any $t \in (0,t_z)$.
\end{proof}

\begin{remark}
There exist domains with a fluctuating boundary which do not blow up at a boundary point to their open feasible direction cone. 
Consider $\Om:=\{(x_1,x_2): x_2 \ne x_1 \sin(1/x_1), -1< x_1 <1\}.$ Then $\tilde{F}_{0_2}(\Omega)=\{(x_1,x_2): x_2 > |x_1| \text{ or } x_2 < - |x_1|, -1< x_1 <1\},$ but $1_{\Omega_{0_2,t}}(z) \to 1_{\setR^2}(z)$ for a.e.~$z \in \setR^2$, because the graph of $x \mapsto x \sin(1/x)$ has zero two-dimensional Lebesgue measure.

\begin{figure}[!h] \centering
\includegraphics[scale=0.5]{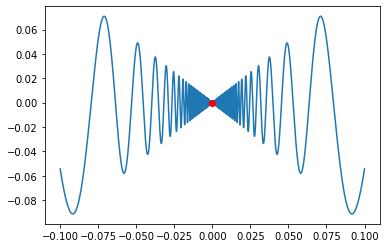}
\caption{Fluctuating boundary with blow-up $\setR^2\neq \tilde{F}_{0_2}(\Omega)$}
\label{Im22}
\end{figure}

\end{remark}

Let us conclude with providing a large class of domains with non-fluctuating boundary, though we will not focus on it in this paper. We recall that a set $A \subset \setR^d$ is convex if the segment $[x,y]\df \{(1-t)x+ty : t \in [0,1]\}$ belongs to $A$ for any $x,y \in A$.

\begin{definition}\label{local convexity}
We say that $S \subset \mathbb{R}^d$ is locally convex at $x \in S$ if there exists $\delta>0$ such that $S\cap \mathbb{B}_\delta(x)$ is convex.
\end{definition}

Then the following holds.

\begin{proposition}\label{Convergence of transalated and scaled domains}
Let $\Omega$ be an open subset of $\ddim$ with a non-empty boundary, and $x \in \partial \Omega$. If either $\bar{\Omega}$ or $\mathbb{R}^d \setminus \bar{\Omega}$ is locally convex at $x$, then $\Omega$ has a non-fluctuating boundary at $x$.
\end{proposition}

\begin{proof}
    Fix $v\in \ddim.$ Assume that there exists $t_n\to 0$ such that $x+t_nv\notin \Omega$ for any $n$. We need to show that there exists $t_v>0$ such that $x+tv\notin \Omega$ for all $t \in (0,t_v)$. Since $x+t_nv \in \ddim \setminus \Omega$, if $\ddim \setminus \Omega$ is locally convex at $x$, then the claim obviously follows with $t_v$ equal to $t_0$.  Assume now that $\Omega$ is locally convex at $x.$ If the claim were not true, then there would exist $s_n\to 0$ such that $t_{n+1}<s_n<t_n$ and $x+s_nv\in \Omega$ for any $n$. But this is not possible because once $x+s_nv \in \Omega$ for some $n,$ all the points $x+sv$ with $s \in (0,s_n)$ must belong to $\Omega$ by convexity, in particular $x+t_{n+1} v$ must belong to $\Omega$, which is in contradiction with our initial assumption.
\end{proof}

\section{Manifolds with kinks}\label{sec:manifolds}

In this section, we introduce the notion of manifold with kinks.

\subsection{Topological manifolds with kinks}

Recall that a topological space is called paracompact if any open cover has a locally finite open refinement, and that a Hausdorff topological space is paracompact if and only if it admits partitions of unity subordinate to any open cover. For two topological spaces $X,Y$ and a subset $A \subset X$, a map $\varphi : A \to Y$ is a topological embedding if it is an homeomorphism onto its image.

\begin{definition}\label{interior and border charts}
    Let $M$ be a paracompact Hausdorff topological space.
    \begin{enumerate}[label=(\roman*)]
        \item A $d$-dimensional chart around $x \in M$ is a pair $(U,\phi)$ where $U$ is an open neighborhood of $x$ in $M$ and $\phi : U \to \ddim$ is a topological embedding. We say that  $(U,\phi)$ is centered at $x$ if $\phi(x)=0_d.$
        \item A $d$-dimensional interior chart around $x$ is a $d$-dimensional chart $(U,\phi)$ such that $\phi(U)$ is an open subset of $\ddim.$ Any $x$ admitting such a chart is called an interior point of $M.$ The set of interior points of $M$ is denoted by $int(M).$
        \item A $d$-dimensional border chart around $x$ is a $d$-dimensional chart $(U,\phi)$ such that $\phi : \bar{U} \to \ddim$ is a topological embedding satisfying $int(\phi(U))\neq \emptyset$ and $\phi(x)$ is a $\cC^0$ boundary point of $int(\phi(U))$.
         Any $x$ admitting a border chart is called a border point. The set of border points of $M$ is denoted by $\partial M.$
    \end{enumerate}
\end{definition}

\begin{remark}
    We use the word \textit{border} as a synonym for possibly singular boundary of manifolds.
\end{remark}

\begin{remark}\label{border points don't depend on the chart chosen}
    The definition of a border point $x$ is \textit{independent} of the choice of a border chart around $x$. Indeed, let $(U,\phi)$ and $(V,\psi)$ be two such charts. Since  $\phi, \psi$ are defined on $\bar{U}, \bar{V}$ respectively, the transition map $\psi \circ \phi^{-1}$ is a homeomorphism from $\phi(\bar{U}\cap \bar{V}) = \overline{\phi(U \cap V)}$ onto $\psi(\bar{U}\cap \bar{V}) = \overline{\psi(U \cap V)}$, in particular it maps $\partial [\phi(U \cap V)]$ to $\partial [\psi(U \cap V)].$ Then $\phi(x)$ is a $\cC^0$ boundary point of $\phi(U)$ if and only if $\psi(x)$ is a $\cC^0$ boundary point of $\psi(U),$ because $\psi \circ \phi^{-1}$ and its inverse are both $\cC^0$, so their composition with a continuous function $\gamma$ as in \eqref{eq:boundary_points} is still continuous.
\end{remark}
    
    Note that a border chart cannot be an interior chart, since for $(U,\phi)$ to be an interior chart, $\phi(x)$ has to be an interior point of $\phi(U),$ contradicting the fact that for a border chart, $\phi(x)$ is a boundary point of $\phi(U).$ What may not be immediate is that a point cannot be simultaneously an interior point with respect to one chart and a border point with respect to another chart. This is precisely the point of the next lemma.

\begin{lemma}\label{lem:either}
    Let $M$ be a paracompact Hausdorff topological space. Then $int(M)\cap \partial M= \emptyset.$
\end{lemma}

\begin{proof}
   It follows from Proposition \ref{Small ball near the border are homeomorphic} that for any border chart $(U,\phi)$ centered at a border point \( x \in \partial M \), there exists \( \delta > 0 \) such that the set \( int(\phi(U)) \cap \mathbb{B}_\delta \) is homeomorphic to a small relatively open ball centered at $0_d$ in the closed upper half-space. But the latter cannot be homeomorphic to a small relatively open ball centered at $0_d$ in the open upper half-space as can be shown by a classical relative homology argument.
\end{proof}

We are now in a position to introduce our definition of topological manifold with kinks.

\begin{definition}\label{d-dimensioanl atlas with kinks and maximal atlas}
     Let $M$ be a paracompact Hausdorff topological space. A $d$-dimensional atlas with kinks on $M$ is a collection $\{(U_i, \phi_i)\}_{i \in \mathcal{I}}$ of $d$-dimensional interior or border charts such that $X=\cup_{i \in \mathcal{I}}U_i$. Such an atlas is called maximal if it is not a proper subcollection of any other atlas. If $M$ admits a maximal $d$-dimensional atlas with kinks, then we say that $M$ is a $d$-dimensional topological manifold with kinks.
\end{definition}

The next proposition shows that, from a topological point of view, there is no difference between a manifold with kinks and a manifold with boundary, just like there is no topological difference between a manifold with corners and a manifold with boundary. The proof is an immediate consequence of Proposition \ref{Small ball near the border are homeomorphic} which implies that the codomain of border charts are all locally homeomorphic.

\begin{proposition}
    Any $d$-dimensional topological manifold with kinks $M$ is locally homeomorphic to a $d$-dimensional topological manifold with boundary.
\end{proposition}

\subsection{Differentiable manifolds with kinks}

Let us now consider differentiable structures on topological manifolds with kinks. We let $k$ be a positive integer that we keep fixed for the whole subsection.

\begin{definition}\label{chart compatibility}
    Let $M$ be a topological manifold with kinks. We say that two charts $(U,\phi), (V,\psi)$ around $x\in M$ are $\cC^k$ compatible if the transition map  $\psi \circ \phi^{-1}:\phi(U\cap V)\to \psi(U\cap V)$ is a $\cC^k$  diffeomorphism in the sense of Definition \ref{Diffeomorphism of sets in ddim}.
    \end{definition}

\nd The previous definition is classical for interior points and becomes relevant for border points only. In particular, it allows to define the regularity of border points thanks to the next key result.
    
\begin{lemma}\label{prop:invariance of regularity of border points}
    Let $M$ be a topological manifold with kinks, and $x$ a border point of $M.$
    Let $(U,\phi), (V,\psi)$ be two $\cC^1$ compatible border charts around $x.$  Then $\phi(x)$ is a Lipschitz (resp.~CDD, $\cC^1$, cusp) boundary point of $\phi(U)$ if and only if it is a Lipschitz (resp.~CDD, $\cC^1$, cusp) boundary point of $\psi(V)$.
\end{lemma}    

\begin{proof}
    For Lipschitz, CDD and $\cC^1$ boundary points, this is immediate from noticing that any transition map $\psi \circ \phi^{-1}:\phi(U\cap V)\to \psi(U\cap V)$ extends to a local $\cC^1$ diffeomorphism $F$ which preserves boundary (Lemma \ref{lem:prop_diffeo}), and that the composition of a $\cC^1$ map with a Lipschitz (resp.~CDD, $\mathcal{C}^1$) map is Lipschitz (resp.~CDD, $\mathcal{C}^1$). For cusps, the result follows from two facts : the push-forward of $\cL^d$ through $F$ is mutually absolutely continuous with $\cL^d$, and the Bouligand tangent cone $T_{\varphi(x)}\phi(\bar{U})$ is mapped to $T_{\psi(x)}\psi(\bar{V})$ by $F$.
\end{proof}

\nd Note that the previous claim fails if the transition maps are Hölder continuous.\\

We can now define differentiable manifolds with kinks.

\begin{definition}\label{d-dimensional differentiable atlas with kinks}
 Let $M$ be a topological manifold with kinks. A $d$-dimensional $\cC^k$ atlas with kinks on $M$ is a collection $\{(U_i, \phi_i)\}_{i \in \mathcal{I}}$ of pairwise $\cC^k$ compatible $d$-dimensional interior or border charts such that $X=\cup_{i \in \mathcal{I}}U_i$. Such an atlas is called maximal if it is not a proper subcollection of any other atlas. If $M$ admits a maximal $d$-dimensional $\cC^k$ atlas with kinks, then we say that $M$ is a $d$-dimensional $\cC^k$ manifold with kinks.
\end{definition}

The word ``kink'' suggests particular boundary points where the manifold presents some type of singularity. With this in mind, we classify border points of $\cC^1$ manifolds with kinks as follows.

\begin{definition}
    Let $M$ be a $\cC^1$ manifold with kinks.

\begin{enumerate}[label=(\roman*)]
    \item We say that $x \in \partial M$ is a $\cC^1$-boundary point of $M$ if there exists a local border chart $(U,\phi)$ around $x$ such that $\phi(x)$ is a $\cC^1$ boundary point of $int(\phi(U))$. Otherwise, we say that $x$ is an essential kink.
    
    \item We say that an essential kink $x \in \partial M$ is an essential corner of depth $k \in \{2,\ldots,d\}$ if there exists a border chart $(U,\phi)$ centered at $x$ such that, up to a rigid motion, the image $\phi(U)$ locally writes as $\setR^{d}_k$.
    
    \item We say that an essential kink $x\in \partial M$ is an LCDD (resp.~Lipschitz) border point if there exists a border chart $(U,\phi)$ centered at $x$ such that $0_d$ is an LCDD (resp.~Lipschitz) boundary point of $\phi(\bar{U})$.

     \item We say that an essential kink $x\in \partial M$ is a cusp if there exists a border chart $(U,\phi)$ centered at $x$ such that $\cL^d(T_{0_d}\phi(\bar{U}))=0$.
\end{enumerate}
Finally, we say that $M$ has an LCDD (resp.~Lipschitz) border if the essential kinks in $\partial M$ are all LCDD (resp.~Lipschitz).
\end{definition}

\begin{remark}
    It is worth pointing out that essential corners are Lipschitz border points, as the set $\setR^d_k$ is the epigraph of a Lipschitz function $\gamma : H \to \setR$ where $H$ is a $(d-1)$-dimensional subspace of $\setR^d$. Indeed, consider the vector $n = (1,\ldots,1) \in \setR^k$ and its scaled version $\hat{n}=n/\|n\| \in \mathbb{S}^{k-1} \subset \setR^k$. Then $P \df \hat{n}^\perp$ is a $(k-1)$-dimensional subspace of $\setR^k$, and any $x \in \setR^k$ writes as $\xi + t \hat{n}$ with $\xi = (\xi_1,\ldots,\xi_{k-1}) \in P$ and $t \in \setR$, where the coordinates of $\xi$ are taken with respect to any orthonormal basis of $P$. Therefore
    \[
    x = \xi + t \hat{n} \in \setR_k^k \quad \iff  \quad t \ge - \sqrt{k} \min_{1 \le i \le d-1} \xi_i =: g_{k,k}(\xi).
    \]
    In other words, $\setR_k^k = epi (g_{k,k})$, and $g_{k,k}$ is obviously Lipschitz because it is the pointwise minimum of a family of Lipschitz functions. Setting
    \[
    g_{d,k} (\zeta) := g_{k,k}(\pi_P(\zeta)) \quad \text{ for any $\zeta \in P \oplus \setR^{d-k}$},
    \]
    where $\pi_P$ is the Euclidean orthogonal projection onto $P$, we get that $\setR_k^d = epi (g_{d,k})$, and $g_{d,k}$ is Lipschitz too.
\end{remark}

\vspace{2mm}  

Note that, according to our definition, any point on the boundary of a smooth manifold with boundary is a $\cC^1$ boundary point, and any corner on a manifold with corner is an essential corner. But the notion of essential kink captures wilder singularities, as illustrated below. Therefore, the category of manifolds with kinks is strictly larger than the ones of manifolds with corners, of manifolds with boundary, and of manifolds without boundary, and contains them all.

\begin{examples}
The following examples are displayed in Figure \ref{boundaries}.
\begin{enumerate}[label=(\roman*)]
    \item The pyramid $M:=\{(x,y,z)\in \cl^3: z\ge \max(|x|,|y|)\}$ is a manifold with kinks that is not a manifold with corner, because the point $0_3\in \partial M$ is an essential kink which is not a corner.
    \item The epigraph $M:=\{(x,y)\in \cl^2: y < \sqrt{|x|} \}$ is a manifold with kinks that is not not a manifold with corner, because the point $0_2\in \partial M$ is a cusp.
\end{enumerate}

\begin{figure}[!h] \centering \begin{subcaptionblock}{.4\textwidth} \centering \includegraphics[scale=0.5]{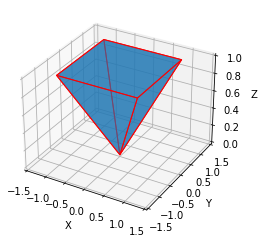} \caption{Pyramid} \end{subcaptionblock}
\begin{subcaptionblock}{.4\textwidth} \centering \includegraphics[scale=0.35]{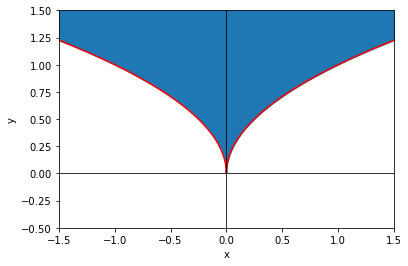} \caption{Cusp}\end{subcaptionblock}
\caption{Essential kinks that are not essential corners}\label{boundaries} \end{figure}
\end{examples}

We conclude this section with a definition regarding non-fluctuating border points which is the analogue of Definition \ref{def:domain with non-fluctuating boundary} in the context of manifolds with kinks.

\begin{definition}\label{def:non-fluctuating border point}
Let $M$ be a $\cC^1$ manifold with kinks. Then $x \in \partial M$ is called a (a.e.) non-fluctuating border point if for any chart $(U,\phi)$ centered at $x$, the set $int(\phi(U))$ has (a.e.) non-fluctuating boundary at $0_d.$  
\end{definition}

\begin{remark}
Any LCDD border point is an a.e.~non-fluctuating border point. This follows from Lemma \ref{lem:negligible}.
\end{remark}

\subsection{Tangent space and inward sector}\label{subsec:tangent}

In Subsection \ref{sec:boundary_cones}, we recalled the concepts of Bouligand tangent cone, feasible direction cone and open feasible direction cone at a point of a subset of $\ddim$, and we derived some results on these cones at suitably regular boundary points. The goal of the present subsection is to introduce the corresponding notions for smooth manifolds with kinks and to derive peculiar properties at essential kinks.

\subsubsection{Tangent space}

\nd From now on, we work only with smooth manifolds with kinks, i.e.~those manifolds with kinks for which the transition maps are all $\cC^{\infty}$. This is to make sure that the tangent space defined below as the space of smooth derivations (i.e.~linear forms on the space of smooth functions) is finite dimensional. Indeed, for any $\cC^k$ manifold with $k < +\infty$, the space of $\cC^k$ derivations is infinite dimensional : see e.g.~\cite[Theorem 4.2.41]{AMR}.

In order to introduce smooth derivations on manifolds with kinks, we must first define what a smooth function is. This is what we do in the next definition, largely inspired by the context of manifolds with corners (see e.g.~\cite[Definition 2.2]{Joyce}). 

\begin{definition}\label{def:smooth functions}
Let $M$ be a $d$-dimensional smooth manifold with kinks. For any open set $O \subset M$, a function $f:O\to \cl$ is called smooth in a neighborhood of $x\in O$ if it is smooth in every chart $(U,\phi)$ 
 centered at $x$, namely :
 \begin{enumerate}[label=(\roman*)]
     \item if $x \in int (M)$, then there exists $\rho>0$ such that $f\circ \phi^{-1} : \phi(U) \cap \mathbb{B}_\rho^d \to \setR$ is smooth in the classical sense,
     \item if $x \in \partial M$, then there exist $\rho>0$ and an open neighborhood $V\subset \ddim$ of $0_d$ containing $\phi(\bar{U}) \cap \mathbb{B}_\rho^d$ to which $f\circ \phi^{-1} : \phi(\bar{U}) \cap \mathbb{B}_\rho^d \to \setR$ extends to a smooth function.
 \end{enumerate}
 We denote by $\cC^\infty(O)$ the space of smooth functions on $O$, that is to say, the functions which are smooth at any $x \in O$.
\end{definition}

\begin{remark}
    If $x \in \partial M$ and $f\circ \phi^{-1} : \phi(\bar{U}) \cap \mathbb{B}_\rho^d \to \setR$ admits Taylor polynomials of arbitrary high order, then the existence of a smooth extension of $f\circ \phi^{-1}$ to $\setR^d$ is ensured by Whitney's extension theorem \cite[Theorem I]{Whitney}, see also \cite[VI.2.]{Stein}. Note that in the setting of manifolds with corners, the latter upgrades into the well-posedness of Seeley's linear extension operator \cite{Seeley}, see \cite[Section 1.4]{Melrose}.
\end{remark}

We are now in a position to define smooth derivations and tangent spaces as follows.

\begin{definition}\label{def:tangent space} \label{def:tangent sp at a kink}
     Let $M$ be a smooth manifold with kinks. Then the tangent space of $M$ at $x \in M$ is the real vector space
     \[
     T_xM \df \{\mathcal{D}: \cC^{\infty}(M)\to \cl \text{ linear }: \mathcal{D}(fg)= f(x) \mathcal{D}(g) + g(x)\mathcal{D}(f)  \text{ for any $f,g \in \cC^{\infty}(M)$}\}.
     \]
     Any element $\mathcal{D} \in T_xM$ is called a smooth derivation at $x$. 
\end{definition}

We can now define the differential of a smooth function between smooth manifolds with kinks.  There is no change with the case of classical manifolds, see e.g.~\cite[p.55]{Lee}.

\begin{definition}\label{def:diffeoMFK}
Let $M$ and $N$ be smooth manifolds with kinks of dimension $d$ and $d'$ respectively, and $O\subset M$ an open set.
\begin{enumerate}[label=(\roman*)]
    \item  A function $\Phi:O\to N$ is called smooth in a neighborhood of $x\in O$ if its local expression is smooth in any couple of charts $(U,\phi)$  and $(U',\psi)$ centered at $x$ and $\Phi(x)$ respectively, namely:
 \begin{enumerate}[label=(\roman*)]
     \item if $x \in int (M)$, then there exists $\rho>0$ such that $\psi \circ f\circ \phi^{-1} : \phi(U) \cap \mathbb{B}_\rho^d \to \setR^{d'}$ is smooth in the classical sense,
     \item if $x \in \partial M$, then there exist $\rho>0$ and an open neighborhood $V\subset \ddim$ of $0_d$ containing $\phi(\bar{U}) \cap \mathbb{B}_\rho^d$ to which $\psi \circ f\circ \phi^{-1} : \phi(\bar{U}) \cap \mathbb{B}_\rho^d \to \setR$ extends to a smooth function.
 \end{enumerate}
 We denote by $\cC^\infty(O,N)$ the space of functions $O\to N$ which are smooth in a neighborhood of any $x \in O$.

    \item For any $\Phi \in \cC^\infty(O,N)$, the differential of $\Phi$ at $x \in M$ is the linear map
\[
\di_x \Phi : T_x M \to T_{\Phi(x)}N
\] 
sending a derivation $\cD \in T_xM$ to the derivation $\di_x \Phi(\cD)  \in T_{\Phi(x)}N$ defined by :
\[
\di_x \Phi(\cD)( h) \df \cD(h\circ \Phi) \qquad \forall h \in \cC^\infty(N).
\]
 \item A function $\Phi \in \cC^\infty(O,N)$ is called a smooth diffomorphism onto its image if it is a smooth bijection with smooth inverse.
\end{enumerate}
\end{definition}

\begin{remark}
One can check with no harm that the usual chain rule holds in this context.
\end{remark}

Let us now establish the following natural result.

\begin{lemma}\label{prop:dimension of tangent space}
    Let $M$ be a smooth manifold with kinks and $x\in \partial M.$ For any local chart $(U,\phi)$ centered at $x$, the differential $$\di_x\phi:T_x M \to T_{0_d}\phi(\bar{U})$$ is a linear isomorphism, and $dim(T_xM)=dim(M).$
\end{lemma}

\begin{proof}
    When $x$ is an interior point, a $\cC^1$ boundary point or an essential corner, this is already known, see \cite{Joyce}. Let us then assume that $x$ is an essential kink which is not an essential corner. Consider a local chart $(U,\phi)$ centered at $x$. Like for manifolds without boundary, we obtain from the chain rule applied to the identities $\phi \circ \phi^{-1} = \mathrm{id}_{\phi(U)}$ and $\phi^{-1} \circ \phi = \mathrm{id}_{U}$ that the differential $\di_x\phi:T_x M \to T_{0_d}\phi(\bar{U})$ is a linear isomorphism with inverse $\di_{0_d}\phi^{-1}$. But the space of derivations at the $\cC^0$ boundary point $0_d$ of $\phi(U)$ coincides with $\setR^n$, since a basis of this space is given by the classical partial differential operators $\partial \cdot /\partial x_1,\ldots, \, \partial\cdot /\partial x_n$ defined on the open set $\phi(U)$ and naturally extended to any open set $V \subset \setR^d$ containing $\overline{\phi(U)}$.
\end{proof}

\begin{remark}
    Like for manifolds without/with boundary or corners, the preceding lemma and its proof show that any local chart $(U,\phi)$ centered at a point $x$ in a manifold with kinks provides a linear isomorphism $T_xM \simeq \setR^d$.
\end{remark}

\subsubsection{Inward sector}

Let us now introduce the notion of inward tangent sector for smooth manifolds with kinks.  This is analogous to the inward tangent sector of Euclidean domains as discussed in Section \ref{subsec:diffeo_inward}. Our definition builds upon the classical characterization of the tangent space in terms of initial velocities of curves, which holds true for manifolds without/with boundary (see e.g.~\cite[p.~68--70]{Lee}). For manifolds with boundary, one can identify inward tangent vectors at boundary points by specifying the domain of the curves we choose. This is how we came up with the following natural definition.

\begin{definition}\label{def: inward sector of the tangent space}
    Let $M$ be a smooth manifold with kinks. For any $x \in M$, we define the equivalence relation $\sim$ on the set of smooth curves $c : I \to M$ such that $I=[0,\epsilon)$ for some $\epsilon >0$ and $c(0)=x$ by setting :
    \begin{equation}\label{eq:equiv}
    c_1 \sim c_2 \iff (f \circ c_1)'(0)= (f\circ c_2)'(0) \, \, \forall f \in C^\infty(M).
    \end{equation}
    Then the inward tangent sector of $M$ at $x$ is the space of equivalent classes of such curves under $\sim$ :
    \[
    I_xM := \{c \colon I \to M \text{ smooth such that $I=[0,\epsilon)$ for some $\epsilon > 0$ and $c(0)=x$}\}/\sim .
    \]
\end{definition}

 It follows from \eqref{eq:equiv} that any element $[c] \in I_xM$ canonically defines a smooth derivation $f \mapsto (f\circ c)'(0)$ belonging to $T_xM$. As such, \begin{equation}\label{eq:inclusion}
I_xM \subset T_xM.
     \end{equation}
The converse is obvious when $x \in int M$. We discuss the case $x \in \partial M$ in Lemma \ref{lem:chara_border} below.
     
We shall use the natural convention which denotes equivalent classes $[c]$ as $c'(0)$, and think of these objects as initial velocities pointing towards the interior of $M$.

\begin{lemma}\label{prop:inward sectors correspond to inward sectors via charts}
Let $M$ be a smooth manifold with kinks and $(U,\phi)$ a local chart centered at some $x \in \partial M$. 
    Then for any $y \in U \cap \partial M$, $$I_yM=[\di_{y}\phi]^{-1}(I_{\phi(y)}\phi(\bar{U})).$$ 
\end{lemma}

\begin{proof}
Let us establish $\subset$. If $c'(0) \in I_yM$ for some smooth $c : [0,\epsilon) \to M$ such that $c(0)=y$, extend $\tilde{c}\df \phi \circ c : [0,\epsilon) \to \phi(\bar{U})$ to a $\cC^1$ curve $\bar{c} : (-\epsilon,\epsilon) \to \phi(\bar{U})$ in any way, for example by symmetrizing $\tilde{c}$ with respect to $\phi(y)$. The chain rule yields that $\bar{c}'(0) = d_y \phi(\tilde{c}'(0))$ so that $\tilde{c}'(0) = [d_y \phi]^{-1}(\bar{c}'(0)) = [d_{\phi(y)} \phi^{-1}](\bar{c}'(0))$. Since $\bar{c}'(0) \in I_{\phi(y)} \phi(U)$ we get that $\tilde{c}'(0) \in [d_{\phi(y)} \phi^{-1}] (I_{\phi(y)} \phi(U))$ as desired. The converse inclusion $\supset$ is proved along similar lines that we skip for brevity.
\end{proof}

\begin{remark}\label{rem:pre-distinction}
The same proof shows that if $(U,\phi)$ is a local chart centered at some interior point $x \in M$, then $I_y M = [\di_{y}\phi]^{-1}(I_{\phi(y)}\phi(U))$ for any $y \in U$. Since in this case $\phi(y)$ belongs to the open set $\phi(U)$, the inward tangent sector $I_{\phi(y)}\phi(U)$ clearly coincides with $\setR^d$, hence we get $I_yM = T_yM$. The same holds when $(U,\phi)$ is a local chart centered at a border point  $x$ and $y$ belongs to $U \backslash \partial M$.
\end{remark}

We are now in a position to characterize border points in terms of their inward tangent sector.

\begin{lemma}\label{lem:chara_border}
Let $M$ be a smooth manifold with kinks, and $x \in M$. Then the following holds, where $\simeq$ means that there exists a bijection that preserves multiplication by positive real numbers.
\[
\begin{array}{lll}
    x \in int(M)  & \iff &  I_xM \simeq T_xM.\\
    \text{$x$ is a $\cC^1$ boundary point of $M$}  & \iff & I_xM \simeq \mathbb{H}^d.\\
    \text{$x$ is an essential corner of depth $k$ of $M$}  & \iff & I_xM \simeq \setR^d_k.
\end{array}
\]
\end{lemma}

\begin{proof}
    It follows from Remark \ref{rem:pre-distinction} that if $x \in int(M)$ then $I_xM \simeq T_xM$. If $x \in \partial M$, then up to rigid motion $\phi(\bar{U}) \subset \bar{\mathbb{H}}^d$ and the latter is a cone, thus $I_{0_d}\phi(\bar{U}) \subset \bar{\mathbb{H}}^d$. This prevents $I_x M$ from being isomorphic to $T_xM \simeq \setR^d$. Since $x$ can only be border or interior (Lemma \ref{lem:either}), the first equivalence is established.  To prove the second and third ones,  notice that Lemma \ref{prop:inward sectors correspond to inward sectors via charts} implies in both cases that $I_xM \simeq I_{0_d} \phi(\bar{U})$.  The conclusion follows from the fact that for small enough $\rho>0$, the set $\phi(\bar{U})\cap \mathbb{B}_\rho^d$ is an open neighborhood of $0_d$ in $\mathbb{H}^d$ and $\mathbb{R}^d_k$ respectively, which yields that $I_{0_d} \phi(\bar{U})$ is $\mathbb{H}^d$ in the first case and $\mathbb{R}^d_k$ in the second one. 
\end{proof}

\subsubsection{Strictly inward sector}

This section extends the Euclidean open feasible direction cone to the setting of manifolds with kinks. To do so, it might be natural to consider
\[
\{ c'(0) : c  \, \colon I \to M \text{ smooth, }  I = [0,\epsilon)\text{ for some $\epsilon > 0$, $c(t) \in int(M)$ for any $t>0$, and } c(0)=x  \}/\sim
\]
where $x$ is a border point of a smooth manifold with kinks, and $\sim$ is like in \eqref{eq:equiv}. A problem with this set is that even if a curve $c$ entirely lies within $int(M)$, the initial velocity vector $c'(0)$ might still be a boundary vector.  For instance,  take $M:=[0,\infty)\times [0,\infty) \subset \cl^2,  x:=0_2$, and $c(t):=(t+t^2,t^2)$ for $t > 0$. Then $c'(0)=(1,0)$ belongs to the previous set, but if we consider $M$ as a subset of $\cl^2,$ its open feasible direction cone at $0_2$ is $(0,+\infty)\times (0,+\infty)$ that does not contain the boundary vector $(1,0)$. See Figure \ref{Im14}.

\begin{figure}[!h] \centering
\includegraphics[scale=0.6]{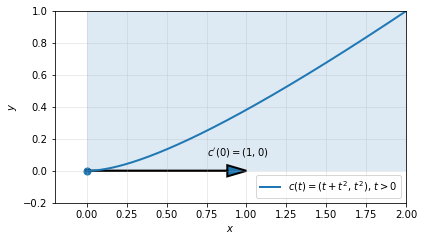}
\caption{Counter-example}
\label{Im14}
\end{figure}

For this reason, we use the following definition.

\begin{definition}\label{def:Strictly inward sector of the tangent space}
Let $M$ be a smooth manifold with kinks. Then the strictly inward sector of $M$ at $x \in M$ is the subset of $T_xM$ defined as
\[
\tilde{I}_xM := 
\begin{cases}
 I_xM & \text{if $x$ is a cusp}, \\
int (I_xM) & \text{otherwise.}
\end{cases}
\]  
Here $T_xM$ is endowed with the natural topology coming from the identification with $\setR^n$ induced by any local chart.
\end{definition}


\subsection{Riemannian manifolds with kinks}\label{subsec:metric}

In this section, we develop a suitable notion of Riemannian metric for manifolds with kinks. There is no particular difference compare to the case of manifolds without/with boundary, but we provide details for completeness.

\subsubsection{Tangent bundle} Let us first define the tangent bundle on a smooth manifold with kinks. The definition is basically the same as for manifolds without/with boundary.

\begin{definition}\label{tangent bundle}
   Let $M$ be a smooth manifold with kinks. Then the tangent bundle of $M$ is the vector bundle
   $$TM:=\bigsqcup_{p\in M}T_pM = \{(p,v):v\in T_pM\}.$$
\end{definition}
  
  Let us check the following natural result.
  
  \begin{lemma}\label{prop:tangent bundle is a smooth manifold with kinks}
    Let $M$ be a $d$-dimensional smooth manifold with kinks. Then $TM$ is a $2d$-dimensional smooth manifold with kinks.
  \end{lemma}

  \begin{proof}
  Denote by $\pi$ the quotient map 
  $TM \to M$ mapping $(x,v)$ to $x$. Consider a maximal smooth atlas with kinks $\mathcal{A} = \{(U_i,\phi_i)\}$ on $M$. For a local chart $(U_i,\phi_i)$ in this atlas, we define a corresponding $2d$-dimensional chart $(\mathcal{U}_i,\Phi_i)$ for $TM$ by setting
  \[
  \mathcal{U}_i\df \pi^{-1}(U_i), \qquad \Phi_i(x,v) := (\phi_i(x), \di_x\phi_i(v)) \,\,\, \text{ for all $(x,v) \in \mathcal{U}_i$.}
  \]
  For any chart $(U_j,\phi_j) \in \mathcal{A}$ such that $U_j \cap U_i \neq \emptyset $, the charts $(\mathcal{U}_i,\Phi_i)$ and $(\mathcal{U}_j,\Phi_j)$ are $\cC^\infty$ compatible because $\Phi_j \circ \Phi^{-1} : (x,v) \mapsto (\phi_j\circ \phi^{-1}(x), d_{\phi_i(x)}\phi_j \circ d_x \phi_i^{-1}(v))$ is smooth from $\Phi_i(\mathcal{U}_i \cap \mathcal{U}_j)$ to $\Phi_j(\mathcal{U}_i \cap \mathcal{U}_j)$.

Let us show that, with respect to these charts, the interior and border of $M$ match up with the interior and border of the tangent bundle $TM$. This is obvious for the interior since the local charts are defined as in the case of manifolds without boundary. Let us then consider $p \in \partial M$ and a border chart $(U,\phi)$ centered at $p$. Then $\phi(p)=0_d$ is a $\mathcal{C}^0$ boundary point of $int(\phi(U))$. Let us show that for any $v \in T_pM$,
\[
\Phi(p,v) = (\phi(p), \di_p \phi(v))
\]
is a $\mathcal{C}^0$ boundary point of $int(\Phi(\mathcal{U}))$, where $\mathcal{U}\df \pi^{-1}(U)$. Shrinking $U$ if necessary, we can identify $\mathcal{U}$ with $U \times \mathbb{R}^d$, so that $\Phi(\mathcal{U})$ is $\phi(U) \times \mathbb{R}^d$. Then $\partial \Phi(\mathcal{U}) = \partial \phi(U) \times \mathbb{R}^d$. Since $\phi(p)$ is a $\mathcal{C}^0$ boundary point of $\phi(U)$, we get that any $(\phi(p), v)$ is a $\mathcal{C}^0$ boundary point of $\phi(U) \times \mathbb{R}^d$.
\end{proof}

\begin{remark}
    We could also define the cotangent bundle on a smooth manifold with kinks as $T^*M \df \bigsqcup_{p\in M}T_p^*M = \{(p,\omega):\omega\in T_p^*M\}$, where each $T_p^*M$ is the dual of $T_pM$. We do not delve on this notion since we don't need it in the rest of the paper.
\end{remark}

\subsubsection{Covariant $k$-tensor bundle}

Let $k$ be a positive integer. Recall that a covariant $k$-tensor on a vector space $V$ is an element of the $k$-fold tensor product $V^* \otimes \dots \otimes V^*$ or, equivalently, a $k$-linear map $V \times V \times \dots \times V \to \cl.$ We denote by $T^k(V^*)$ the space of all covariant $k$-tensors on $V$. Then we can define the covariant $k$-tensor bundle on a smooth manifold with kinks as follows.

\begin{definition}
  Let $M$ be a smooth manifold with kinks. Then the space of covariant $k$-tensors on $M$ is defined as $$T^k(T^*M):=\bigsqcup_{p\in M}T^k(T_p^*M).$$
\end{definition}

Acting like in the previous subsection, one can easily show the following. We omit the proof for brevity.

\begin{lemma}\label{prop:covariant k-tensor bundle is a manifold with kinks}
     Let $M$ be a smooth manifold with kinks. Then $T^k(T^*M)$ is a smooth manifold with kinks of dimension $2dk$.
\end{lemma}

\begin{remark}
Likewise, we could define $(k,r)$-tensors on smooth manifolds with kinks, but we do not need them in the present paper so we skip them.
\end{remark}

\subsubsection{Riemannian metrics}

Let us define covariant tensor fields on manifold with kinks.

\begin{definition}
    Let $M$ be a smooth manifold with kinks.  A $k$-tensor field on $M$ is a section $g$ of the covariant $k$-tensor bundle ${T}^k(T^*M)$. Such a field is of $\cC^\ell$ regularity if  $g:M\to {T}^k(T^*M)$ is a $\cC^{\ell}$ map w.r.t. the smooth structures introduced in the previous section.
 \end{definition}

We are now in a position to define Riemannian metrics on smooth manifolds with kinks.

 \begin{definition}\label{def:Riemannian metrics on manifolds with Lipschitz kinks}
    Let $M$ be a smooth manifold with kinks, and $k$ a positive integer. A $\cC^k$ Riemannian metric on $M$ is a $\cC^k$ symmetric, positive definite, section of the covariant $2$-tensor bundle ${T}^2(T^*M)$.
\end{definition}

\subsubsection{Riemannian distance and volume measure}

The Riemannian distance defined via a length-minimizing problem extends with no change to the context of manifolds with kinks. We recall the definition for completeness and refer to \cite[p.337--341]{Lee}, for instance, for more details.

\begin{definition}
    Let $M$ be a smooth connected manifold with kinks admitting a $\cC^1$ Riemannian metric $g$. The associated Riemannian distance is defined by
    \[
    \di(x,y) \df \inf \left\{\int_0^1 g_{c(t)}(c'(t),c'(t)) \di t : c \in \cC^1([0,1],M) \text{ s.t.~}c(0)=x \text{ and } c(1)=y \right\},
    \]
    for any $x,y \in M$.
\end{definition}

In the same way, the definition of the Riemannian volume measure carries over to manifolds with kinks and behaves like in the case of manifolds without/with boundary.

\begin{definition}
    Let $M$ be a smooth manifold with kink admitting a $\cC^1$ Riemannian metric $g$. Then the Riemannian volume measure is defined by
    \[
    \mathrm{vol}_g(A) \df \sum_{\alpha} \int_{\phi_\alpha(U_\alpha)} \chi_\alpha \circ \phi_\alpha^{-1}\sqrt{\det g}
    \]
    for any Borel set $A \subset M$, where $\{(U_\alpha, \phi_\alpha\}$ is an atlas compatible with the smooth structure of $M$ and $\{\chi_\alpha\}$ is a partition of unity subordinate to this atlas.
\end{definition}

\subsection{Extension of Riemannian manifolds with kinks}\label{subsec:submanifolds}

For our purposes, we need to extend beyond the border any $\cC^2$ Riemannian metric defined on a smooth manifold with kinks $M$. To this aim, we shall flow $M$ into its interior using a suitable semiflow. We adapt an argument for manifolds with corners that goes back to \cite{DouadyHerault} at least, see also \cite[Section 2.7]{Michor}.

\subsubsection{Vector fields and semiflows} Let us begin by defining vector fields and their associated semiflows on smooth manifolds with kinks.

\begin{definition}\label{inner and strictly inner vector fields} Let $M$ be a smooth manifold with kinks. A smooth vector field on $M$ is a smooth section of the tangent bundle $TM.$ Such a vector field $\xi$ is called inward-pointing (resp.~strictly inward pointing) if $\xi_x\in I_xM$ (resp.~$\tilde{I}_xM$) for any $x\in M.$
\end{definition}

Let us ensure that any smooth manifold with kinks admits a stricly inward-pointing vector field.

\begin{lemma}\label{lem:strictly_inward}
    Let $M$ be a smooth manifold with kinks. Then $M$ admits a smooth strictly inward-pointing vector field.
\end{lemma}

\begin{proof}
    Let $(U,\phi)$ be a border chart centered at some $x \in \partial M$. Up to composing $\phi$ with a rigid motion, we may assume that there exist $\delta>0$ and $\gamma \in \cC^0(\setR^{d-1})$ such that 
    \[
        \begin{cases}
        \phi(\bar{U}) \cap  \setB_\delta^d  = epi(\gamma)  \cap \setB_\delta^d, \nonumber\\
        int (\phi(\bar{U})) \cap \setB_\delta^d  = \mathring{epi}(\gamma)  \cap \setB_\delta^d.
        \end{cases}
     \]
     Then we set $\xi_{(U,\phi)}(y) \df (\di_{y}\phi)^{-1}(e_d)$
     for any $y \in \bar{U} \cap \phi^{-1}(\setB_\delta^d)$. Since $e_d$ belongs to $I_{\phi(y)}\phi(\bar{U})$, we get from Lemma \ref{prop:inward sectors correspond to inward sectors via charts} that $\xi_{(U,\phi)}(y) \in I_yM$. If $y$ is a cusp, this implies that $\xi_{(U,\phi)}(y) \in \tilde{I}_yM$. If not, notice that $e_d$ belongs to the interior of $I_{\phi(y)}\phi(\bar{U})$, which is mapped to the interior of $I_yM$ by the linear isomorphism $\di_{y}\phi)^{-1}$. Thus $\xi_{(U,\phi)}(y) \in \tilde{I}_yM$ in this case too. Consider now an atlas $\{(U_\alpha, \phi_\alpha)\}$ of $M$ and a partition of unity $\{\chi_\alpha\}$ subordinate to this atlas. Define
     \[
     \xi \df \sum_{\substack{(U_\alpha,\phi_\alpha)\\ \text{border charts}}} \chi_\alpha \, \xi_{(U_\alpha,\phi_\alpha)} : M \to TM.
     \]
     Then $\xi$ is a global smooth vector field on $M$, and it is strictly inward-pointing by construction.
\end{proof}

Recall the definition of integral curve.

\begin{definition}
     Let $M$ be a smooth manifold with kinks, and $\xi$ a smooth vector field on it. An integral curve of $\xi$ is a smooth curve $c : I \to M$ such that $c'(t) = \xi(c(t))$ for any $t \in I$.
\end{definition}

Then the following existence result holds.

\begin{theorem}\label{th:semiflow}
     Let $M$ be a smooth manifold with kinks, and $\xi$ the smooth striclty inward-pointing vector field given by Lemma \ref{lem:strictly_inward}. Then there exists a smooth function $\delta : \partial M \to (0,+\infty)$ and a smooth embedding $\Phi : \mathcal{P}_\delta \to M$, with $\mathcal{P}_\delta = \{ (t,x) : x \in \partial M, t \in [0,\delta(x)) \} \subset \setR \times \partial M$, such that for any $x \in \partial M$ the map $[0,\delta(x)) \ni t \mapsto \Phi(t,x)$ is an integral curve of $\xi$ starting at $x$.
\end{theorem}

\begin{proof}
    The function $\delta$, the set $\mathcal{P}_\delta$ and the embedding $\Phi$ are first defined locally, and then patched together by means of a partition of unity. To define these objects locally around some point $x \in \partial M$, consider a border chart $(U,\phi)$ centered at $x$. By definition of border chart, there exists $\rho>0$ such that, up to a rigid motion, the set $\phi(U) \cap \mathbb{B}_\rho^d$ writes as the local epigraph of some continuous function $\gamma:\setR^{d-1}\to \setR$ such that $\gamma(0_{d-1})=0$. Since $\phi(U \cap \partial M)$ is mapped to the local graph of this function, and since $d_y \phi(\xi)\equiv e_d$, we can apply the Cauchy--Lipschitz theorem in $\mathbb{R}^d$ to get existence of a smooth function $\delta : \phi(U \cap \partial M) \to \setR_+$ such that for any $y \in U \cap \partial M$ there exists an integral curve $c_{\phi(y)} : [0,\delta(y))] \to \setR^d$ of $\di \phi(\xi)$ starting at $\phi(y)$, so that the map $[0,\delta(y)) \ni t \mapsto \Phi(t,y) \df \phi^{-1}(c_{\phi(y)}(t))$ is an integral curve of $\xi$ starting at $y$.
\end{proof}

\subsubsection{Riemannian submanifolds with kinks}

Let us now provide a definition of submanifold adapted to the context of manifolds with kinks, inspired by \cite[p.19]{Michor} who introduced submanifolds with corners.

\begin{definition}\label{submanifolds with kinks}
    Let $M$ be a $d$-dimensional manifold with kinks, and $k$ a positive integer at most equal to $d$.
    \begin{enumerate}
        \item[(i)] We say that $N \subset M$ is a $k$-dimensional submanifold with kinks of $M$ if for every $p\in N,$ there is a chart $(U,\phi)$ of $M$ centered at $p$ so that $\phi(U \cap N) \subset \cl^k \times \{0\}^{d-k} \subset \ddim$ and $\phi(p)=0_d$ is an interior or a $\cC^0$ boundary point of $\phi(U\cap N)$. We call $(U,\phi)$ a slice chart centered at $p$ of $N$
    \item[(ii)] Assume now that $M$ is endowed with a $\cC^2$ Riemannian metric $g$. Then $g$ induces a $\cC^2$ Riemannian metric on $N$ like in the setting of usual submanifolds (i.e.~through the inclusion $TN \subset TM$).
    \end{enumerate}
\end{definition}

We say that a smooth manifold is \textit{open} if it is non-compact without boundary. Our next result is that any smooth manifold with kinks $M$ can be embedded into an open smooth manifold $\tilde{M}$ having same dimension, and that any $\cC^2$ Riemannian metric on $M$ extends to $\tilde{M}$ provided $\partial M$ is Lipschitz. The proof actually embeds a smooth manifold with kinks into its interior.

\begin{theorem}
    \label{lem:Embedding a manifold with kinks into a manifold without boundary}
    Let $M$ be a smooth $d$-dimensional manifold with kinks.
    \begin{enumerate}
        \item[(i)] Then there exists a smooth open $d$-dimensional manifold $\tilde{M}$ such that $M \subset \tilde{M}$ is a smooth submanifold with kinks.
        \item[(ii)] Consider a $\cC^2$ Riemannian metric $g$ on $M$. Then there exist a neighborhood $O$ of $M$ in $\tilde{M},$ i.e. an open subset $O\subset \tilde{M}$ containing $M,$ and a $\cC^{2}$ Riemannian metric $\tilde{g}$ on $O$ so that $\tilde{g}$ restricts to $g$ on $M.$
    \end{enumerate}
\end{theorem}

\begin{proof}
    Let us prove (i). Consider a strictly inward pointing vector field $\xi$ on $M$ as given by Lemma \ref{lem:strictly_inward}. Then the semiflow of $\xi$ given by Theorem \ref{th:semiflow} maps $\partial M$ into $int(M),$ and $int(M)$ into $int(M).$

    Let us now prove (ii). As in many extension theorems, the idea is to extend $g$ on local neighborhoods of $M$, and then use a smooth partition of unity to patch all these local extensions to a global one. Note that we need only to extend the metric at border points of $M$, since interior points are interior for $\tilde{M}$ too. Let us then consider a maximal $d$-dimensional atlas compatible with the smooth structure of $M$, a Lipschitz border point $x \in \partial M$, and a border chart $(U,\phi)$ from the previous atlas such that $x \in U$. Let  $\tilde{U}$ be an open neighborhood of $U$ in $\tilde{M}$, and $\tilde{\phi}$ an extension of $\phi$ from $U$ to $\tilde{U}$. The steps to extend $g$ from $U$ to $\tilde{U}$ are the following.

\begin{enumerate}[label=(\roman*)]
    \item[a)] Pull back $g$ to $\phi(U)$ by $\phi^{-1}$, i.e.~consider $(\phi^{-1})^*g$ on $\phi(U)$. This is a $\mathcal{C}^2$ Riemannian metric on $\phi(U)$.
    
    \item[b)] Use Whitney's theorem \cite{Whitney} to extend each coordinate of $(\phi^{-1})^*g$ to form a $\cC^2$ Riemannian metric $h$ on an open subset $V$ of $\mathbb{R}^d$ containing $\phi(x)$, such that $\phi^{-1}(V) \subset \tilde{U}$. Note that this extension may not be unique and depends on the choice of the border chart $(U,\phi)$.
    
    \item[c)] Pull back $h$ by $\tilde{\phi}$ to $\tilde{U}$, i.e.~consider $\tilde{\phi}^* h$ on $\tilde{U} \supset U$. Since the extension in the previous step is not unique, this pullback may not be unique either. Nevertheless, regardless of the specific extension used in b), the tensor $g$ extends to a $\mathcal{C}^{2}$ Riemannian metric $\tilde{g}$ on $\tilde{U}$ whose restriction to $U$ agrees with $g$, by the contravariant functoriality of pullbacks applied to the composition $\phi \circ \phi^{-1}=Id$.
\end{enumerate}
\end{proof}

\section{Asymptotic behavior of the intrinsic Gaussian Operator}\label{sec:intrinsic}

In this section, we prove Theorem \ref{th:1} without the refined estimates on the error term $\mathrm{Err}(t)$. We refer to Section \ref{sec:refined} for these refined estimates. Consider a smooth $d$-dimensional Riemannian manifold with kinks $M$ endowed with a $\cC^2$ Riemannian metric $g$, a density $p \in \cC_{\ge 0}^2(M)$, a function $f \in \cC^3(M)\cap L^1(M,p\,\mathrm{vol}_g)$, an exponent $\eta  \in (0,1/2)$, and a point $x \in M$. For the sake of clarity, let us highlight the main steps of our proof, each of which being dedicated a subsection.

\textit{Step 1.} We establish that
\begin{equation}\label{eq:1!}
L_tf(x) = \bar{L}_{t,\eta } f (x) + O(t^{-d/2-1}e^{-t^{2 \eta -1}}) \qquad \text{as $t \downarrow 0$}
\end{equation}
with
\begin{equation}\label{eq:localized}
\bar{L}_{t,\eta } f (x) \df \frac{1}{t^{d/2+1}} \int_{B_{t^\eta }(x)} \exp\left(-\frac{\di_g^2(x,y)}{t}\right) (f(x)-f(y)) p(y) \di \mathrm{vol}_g(y).
\end{equation}

\textit{Step 2.} We write
\begin{equation}\label{eq:error}
\bar{L}_{t,\eta } f (x) = I(t) + II(t)
\end{equation}
where $I(t)$ is a term suited for an exponential change of variable, and we show that
\[
II(t) = o\left( \frac{1}{\sqrt{t}}\right) =: \frac{\mathrm{Err(t)}}{\sqrt{t}}  \qquad \text{as $t \downarrow 0$.}
\]

\textit{Step 3.} We prove an Euclidean version of the expansion
\[
I(t) = -\frac{c_d}{\sqrt{t}} p(x) \, \partial_{v_g(x)}f(x) - c_{d+1} \bigg( p(x) A_gf(x) +  [p,f]_g(x) \bigg) + O(\sqrt{t}) \qquad \text{as $t \downarrow 0$.}
\]

\textit{Step 4.} We conclude by change of variable in $I(t)$.

\subsection{Localisation} We perform the first step of our proof in the general context of a metric measure space $(Z,\di,\mu)$. In this case, the intrinsic $d$-dimensional Gaussian operator at time $t>0$ associated with a density $q \in L^1(Z,\mu)$ is defined by
\[
L_t h(z) \df \frac{1}{t^{d/2+1}}\int_{Z}  \exp\left(-\frac{\di^2(z,y)}{t}\right) (f(z)-f(y)) q(y)\di \mu(y)
\]
for any $h \in L^1(X,p\mu)$ and $\mu$-a.e.~$z \in Z$. Then the following holds.

\begin{lemma}\label{lem:localisation}
   For any $h \in L^1(Z,p \mu)$, $\mu$-a.e.~$z \in Z$, and $t>0$,
\[
\left| \frac{1}{t^{d/2+1}} \int_{X\backslash B_{t^\eta }(z)} \exp\left(-\frac{\di^2(z,y)}{t}\right) (h(z)-h(y)) p(y)\di \mathrm{\mu}(y) \right| 
\leq [|h(z)|\|p\|_{1} + \norm{hp}_{1}] \, \frac{1}{t^{d/2 +1}} e^{-t^{2\eta -1}}
\]
As a consequence,
\[
L_th(z) = \bar{L}_th(z) + O(t^{-d/2-1}e^{-t^{2 \eta -1}}) \qquad \text{as $t\downarrow 0$,}
\]
with
\[
\bar{L}_th(z) \df \frac{1}{t^{d/2+1}} \int_{B_{t^\eta }(z)} \exp\left(-\frac{\di^2(z,y)}{t}\right) (h(z)-h(y)) p(y)\di \mathrm{\mu}(y).
\]
\end{lemma}

\begin{proof}
    By triangle inequality, for any $z,y \in X$,
    \[
    |h(z)-h(y)| \le |h(z)||p(y)| + |hp|(y).
    \]
    Multiply by $\exp\left(-\frac{\di^2(z,y)}{t}\right)$ and integrate over $y \in X \backslash B_{t^\eta}(z)$. The result follows from there since $\di^2(z,y) \ge t^{2 \eta}$ for any such $y$.
\end{proof}

Applying the previous lemma with $(Z,\di,\mu) = (M,\di_g,\mathrm{vol}_g)$, $h=f$, $z=x$, gives \eqref{eq:1!}.

\subsection{Exponential coordinates} Let $(\tilde{M},\tilde{g})$ be the open Riemannian manifold of which $(M,g)$ is a submanifold, as obtained in Theorem \ref{lem:Embedding a manifold with kinks into a manifold without boundary}. We identify $(T_x \tilde{M},\tilde{g}_x) = (T_x M,g_x)$ with $(\setR^d,\cdot)$ by choosing a $g(x)$-orthonormal basis $(w_1,\ldots,w_d)$ of $T_x \tilde{M}$ and mapping each $w_i$ to the $i$-th element of the canonical basis of $\setR^d$. We let $$\iota : \setR^d \to T_xM$$ denote this isometric identification. We denote by $\tilde{B}_r(x)$ the $\tilde{g}$-ball of radius $r$ centered at $x$, by $\tilde{\exp}_x$ the $\tilde{g}$-exponential map at $x$, and by $\tilde{\gamma}_{x,v}$ the unique maximal $\tilde{g}$-geodesic with initial point $x$ and initial velocity $v \in T_x \tilde{M}$.

Since $x$ is an interior point for $\tilde{M}$, there exists $R>0$ such that $\tilde{\exp}_x$ is a diffeomorphism from $\{\|\cdot\|_{\tilde{g}_x} < R \} \subset T_x \tilde{M}$ onto $\tilde{B}_R(x)$. This implies, in particular, that $(B_R(x), (\iota \circ \tilde{\exp}_x)^{-1})$ is a $d$-dimensional chart centered at $x$ of $M$. This chart is interior (resp.~border) if $x$ is interior (resp.~border) for $M$.

Since $(\tilde{M},\tilde{g})$ extends $(M,g)$, the set of $\cC^1$ curves joining $x$ to $y \in \tilde{B}_R(x)$ and lying entirely in $M$ is a subset of those curves lying in $\tilde{M}$, so that
\begin{equation}\label{eq:ineq_distance}
\tilde{\di}(x,y) \le \di(x,y).
\end{equation}
As a consequence, for any $r \in (0,R)$,
\[
B_r(x) \subset \tilde{B}_r(x).
\]
Consider the open subset of $\setR^d$ defined as
\begin{equation}\label{eq:omega}
\Omega \df (\iota \circ \tilde{\exp}_x)^{-1}(B_R(x)).
\end{equation}
Note that if $x$ is interior then $\Omega = \mathbb{B}_R^d$. However, if $x$ is border, then $\Omega$ is a proper subset of $\mathbb{B}_R^d$ admitting $0_d$ as a $\cC^0$ boundary point. Moreover, the regularity of $x$ as a border point transfers to the regularity of $0_d$ as a $\cC^0$ boundary point of $\Omega$.

We define
 \begin{align}\label{eq:W}
    W & \df \{v \in T_x \tilde{M} : \text{ there exists } t_v >0 \text{ s.t.~}\tilde{\gamma}_{x,v}(t) \in M \text{ for all $t \in (0,t_v)$}\}, \\
    \mathbb{W} & \df \iota^{-1}(W), \nonumber
\end{align}
and for any $t>0$ we introduce
 \begin{align}\label{eq:II}
 I(t) & \df \frac{1}{t^{d/2+1}} \int_{B_{t^\eta}(x)\cap \tilde{\exp}_x(W\cap \mathbb{B}_R^d)} \exp\left(-\frac{\di^2(x,y)}{t}\right) (f(x)-f(y)) \, p(y) \, d\mathrm{vol}_g(y),\nonumber\\
 II(t) & \df \frac{1}{t^{d/2+1}} \int_{B_{t^\eta}(x)\backslash \tilde{\exp}_x(W\cap \mathbb{B}_R^d)} \exp\left(-\frac{\di^2(x,y)}{t}\right) (f(x)-f(y)) p(y) \, d\mathrm{vol}_g(y).
 \end{align}
 Obviously,
 \[
 \bar{L}_tf(x) = I(t) + II(t).
 \]
 Moreover, $W$ is a cone : indeed, if $v \in W$ and $\lambda>0$, then $\tilde{\gamma}_{x,\lambda v}(t) = \tilde{\gamma}_{x,v}(\lambda t)$ for any $t \in (0,t_v/\lambda)$, so that $\lambda v \in W$ with $t_{\lambda v}=t_v/\lambda$. This obviously implies that $\mathbb{W}$ is a cone too.

Finally, recall that for any $y \in \setR^d$ we set $y^{(k)} \df y\otimes \ldots \otimes y \in (\setR^d)^{\otimes^k}$. If $h \in \cC^{\infty}(U)$ for some open subset $U \subset \setR^d$, we let
$\di^{(k)}_zh$ denote the differential of order $k$ of $h$ at $z \in \setR^d$, which is understood here as a $k$-linear symmetric map from $(\setR^d)^{\otimes^k}$ to $\setR$, and we set
\begin{equation}\label{eq:previous}
\|\di^{(k)}_zf\|_{op} \df \sup_{y \in \setR^d\backslash \{0_d\}} \frac{\di^{(k)}_zf(y^{(k)})}{\|y\|^{k}} \, \cdot
\end{equation}
We shall also write $\di_z^{(0)}f(y^{(0)})$ for $f(z)$, in which case $\|\di^{(0)}_zf\|_{op}= f(z)$.

\begin{proposition}\label{prop:localized}
Set $\tilde{f}\df f \circ \tilde{\exp}_x \circ \iota$ and $\tilde{p} \df (p \circ \tilde{\exp}_x \circ \iota) \sqrt{\det g}$, where $\sqrt{\det g}$ is the Radon--Nikodym derivative of $(\tilde{\exp}_x \circ \iota)^{-1}_\#\mathrm{vol}_g$ with respect to $\mathcal{L}^d$. Then
    \[
    \bar{L}_{t,\eta} f (x)  = \frac{1}{t} \int_{\mathbb{B}^d_{t^{\eta}}\cap \tilde{F}_{0_d}(\bar{\Omega})}e^{-\frac{\|\xi\|^2}{t}} (\tilde{f}(0_d) - \tilde{f}(\xi))\tilde{p}(\xi)\di \xi + o\left( \frac{1}{\sqrt{t}}\right) \qquad \text{as $t\to 0$.}
    \]
\end{proposition}

\begin{proof}

\textbf{Step 1.}  We show that for any small enough $t>0$,
\begin{align}\label{eq:step1!}
I(t) = \frac{1}{t} \int_{\mathbb{B}^d_{t^{\eta}}\cap \tilde{F}_{0_d}(\bar{\Omega})} e^{-\frac{\|\xi\|^2}{t}} (\tilde{f}(0_d) - \tilde{f}(\xi))\tilde{p}(\xi)\di \xi.
\end{align}

For any $v \in W$, set $s_v \df \sup \{s >0 : \tilde{\gamma}_{x,v}(s') \in M \text{ for any $s' \in [0,s]$}\} \in (0,+\infty]$. Let us show that
\begin{equation}\label{eq:technical}
B_{t^{\eta}}(x) \cap \tilde{\exp}_x(W) = B_{t^{\eta}}(x) \cap  \{\tilde{\exp}_x(sv) : v \in W \text{ with } \|v\|_{g_x} = 1 \text{ and } s \in (0,s_v) \}.
\end{equation}
Since $W$ is a cone,
\[
W = \bigsqcup_{\substack{v \in W\\ \|v\|_{g_x} = 1}} \setR_+ v,
\]
thus
\[
B_{t^{\eta}}(x) \cap \tilde{\exp}_x(W) = \bigsqcup_{\substack{v \in W\\ \|v\|_{g_x} = 1}} B_{t^{\eta}}(x) \cap \tilde{\exp}_x(\setR_+ v).
\]
Now for any $v \in W$ such that $\|v\|_{g_x} = 1$,
\[
 B_{t^{\eta}}(x) \cap \tilde{\exp}_x(\setR_+ v) = \{\tilde{\exp}_x(sv) : s \in (0,s_v) \} \cap B_{t^{\eta}}(x).
 \]
Then 
\[
B_{t^{\eta}}(x) \cap \tilde{\exp}_x(W) = \bigsqcup_{\substack{v \in W\\ \|v\|_{g_x} = 1}} \{\tilde{\exp}_x(sv) : s \in (0,s_v) \} \cap B_{t^{\eta}}(x)
\]
hence we get \eqref{eq:technical}.

Let us now prove that for any $y \in B_{t^{\eta}}(x) \cap \tilde{\exp}_x(W)$,
\begin{equation}\label{eq:equal_distance}
    \tilde{\di}(x,y) = \di(x,y).
\end{equation}
Thanks to \eqref{eq:technical}, we know that there exist $v \in W$ with $\|v\|_{g_x}=1$ and $s \in [0,s_v)$ such that $y = \tilde{\exp}_x(sv)$ and $\tilde{\exp}_x(s'v)\in M$ for any $s' \in [0,s]$. Since $g$ and $\tilde{g}$ coincide on $M$, the $g$-geodesic joining $x$ to $y$ coincides with $[0,s] \ni s' \mapsto \tilde{\exp}_x(s'v)$. This yields \eqref{eq:equal_distance}.

Lastly, we point out that
\begin{equation}\label{eq:step1_1}
    \mathbb{W} = \tilde{F}_{0_d} (\bar{\Omega})
\end{equation}
and
\begin{equation*}
   \tilde{\exp}_x^{-1}(B_{t^\eta}(x)) \cap W = \iota(\mathbb{B}^d_{t^\eta} \cap \mathbb{W}).
\end{equation*}
These are directly resulting from the fact that $\tilde{\gamma}_{x,v}(s) = \tilde{\exp}_x (sv)$ and $t\mapsto t\iota^{-1}(v)$ is the Euclidean geodesic in $\setR^d$ starting at $0_d$ with initial velocity $\iota^{-1}(v)$.

We are now in a position to obtain \eqref{eq:step1!}. We successively use \eqref{eq:equal_distance} and the change of variable $y = (\exp_x\circ \iota)(\xi)$ to obtain
 \begin{align*}
     I(t) &  = \frac{1}{t^{d/2+1}} \int_{B_{t^\eta}(x)\cap \tilde{\exp}_x(W\cap \mathbb{B}_R^d)} \exp\left(-\frac{\tilde{\di}^2(x,y)}{t}\right) (f(x)-f(y)) \, p(y)d\mathrm{vol}_g(y)\\
     & = \frac{1}{t^{d/2+1}} \int_{\mathbb{B}_{t^\eta}^d\cap \mathbb{W}} \exp\left(-\frac{\|\xi\|^2}{t}\right) (\tilde{f}(0_d)-\tilde{f}(\xi)) \tilde{p}(\xi)\, d \xi.
 \end{align*}
Then we apply \eqref{eq:step1_1} to get \eqref{eq:step1} as sought.

 \textbf{Step 2.}
 We show that
 \begin{align}\label{eq:step4}
 II(t) = o\left(\frac{1}{\sqrt{t}}\right)  \qquad \text{as $t \downarrow 0$.}
 \end{align}
To this purpose, let us first establish that as $t \downarrow 0$,
\begin{equation}\label{eq:step2_2}
    1_{\frac{(\iota \circ \tilde{\exp_x})^{-1}(B_{t^\eta}(x)) }{\sqrt{t}}\backslash \mathbb{W}} \to 0 \qquad \text{$\mathcal{L}^d$-a.e.~on $\setR^d$.}
\end{equation}
For any $t >0$ such that $t^{\eta} \le R$,
\begin{align*}
    1_{\frac{(\iota \circ \tilde{\exp_x})^{-1}(B_{t^\eta}(x)) \backslash W}{\sqrt{t}}} & \le 1_{\frac{(\iota \circ \tilde{\exp_x})^{-1}(B_{R}(x)) \backslash W}{\sqrt{t}}} = 1_{\frac{\Omega \backslash \mathbb{W}}{\sqrt{t}}}.
\end{align*}
Since $\mathbb{W}$ is a cone, the latter characteristic function is equal to
$$
1_{\frac{\Omega}{\sqrt{t}} \backslash \mathbb{W}} = 1_{\frac{\Omega}{\sqrt{t}}} - 1_{\mathbb{W}} = 1_{\frac{\Omega}{\sqrt{t}}} - 1_{\tilde{F}_{0_d}(\bar{\Omega})}  \stackrel{t \downarrow 0}{\longrightarrow} 1_{\tilde{F}_{0_d}(\bar{\Omega})} - 1_{\tilde{F}_{0_d}(\bar{\Omega})} = 0 \qquad \text{$\mathcal{L}^d$-a.e.~on $\setR^d$.}
$$
Here we use \eqref{eq:step1_1} to get the second equality and Proposition \ref{prop:gen} for the convergence a.e. as $t\downarrow 0$.

Let us now estimate
\begin{align}\label{eq:obtained}
     |II(t)| &  \le  \frac{1}{t^{d/2+1}} \int_{B_{t^\eta}(x)\backslash \tilde{\exp}_x(W\cap \mathbb{B}_R^d)} \exp\left(-\frac{\di^2(x,y)}{t}\right) |f(x)-f(y)| p(y)d\mathrm{vol}_g(y) \nonumber \\
     & \le \frac{1}{t^{d/2+1}}\int_{B_{t^\eta}(x)\backslash \tilde{\exp}_x(W\cap \mathbb{B}_R^d)} \exp\left(-\frac{\tilde{\di}^2(x,y)}{t}\right) |f(x)-f(y)| p(y)d\mathrm{vol}_g(y) \nonumber \\
     & = \frac{1}{t^{d/2+1}} \int_{(\iota \circ \tilde{\exp}_x)^{-1}(B_{t^\eta}(x))\backslash \mathbb{W}} \exp\left(-\frac{\|\xi\|^2}{t}\right) |\tilde{f}(0)-\tilde{f}(\xi)| \tilde{p}(\xi)d\xi  \\
     & = \frac{1}{t} \int_{\frac{(\iota \circ \tilde{\exp}_x)^{-1}(B_{t^\eta}(x))}{\sqrt{t}}\backslash \mathbb{W}} \exp\left(-\|\zeta\|^2) \right) |\tilde{f}(0_d)-\tilde{f}(\sqrt{t}\zeta)| \tilde{p}(\sqrt{t}\zeta)d\zeta \nonumber
 \end{align}
where we use \eqref{eq:ineq_distance} to get the second inequality, the change of variable $y = (\tilde{\exp}_x \circ \iota) (\xi)$ to get the penultimate line, and $\zeta = \xi/\sqrt{t}$ to get the last one. Now we use the Taylor theorem with Laplace remainder : for any $\zeta \in \frac{(\iota \circ \tilde{\exp}_x)^{-1}(B_{t^{\eta}}(x))}{\sqrt{t}}\backslash \mathbb{W}$ there exist $s_1,s_2 \in (0,\sqrt{t})$ such that
\[
\tilde{f}(0_d)-\tilde{f}(\sqrt{t}\zeta) = - \di_{0_d} \tilde{f} (\sqrt{t}\zeta)  -  \frac{1}{2}  \di^{(2)}_{\sqrt{s_1}\zeta} \tilde{f}  ((\sqrt{t}\zeta)^{(2)})  ,
\]
\[
\tilde{p}(\sqrt{t}\zeta) = \tilde{p}(0_d) + \di_{\sqrt{s_2}\zeta} \tilde{p}(\sqrt{t}\zeta),
\]
so that 
\begin{align*}
|\tilde{f}(0_d)-\tilde{f}(\sqrt{t}\zeta)|\tilde{p}(\sqrt{t}\zeta) & \le \sqrt{t} \, \tilde{p}(0_d) |\di_{0_d} \tilde{f} (\zeta)|  +  \frac{t \, \tilde{p}(0_d)}{2}  |\di^{(2)}_{\sqrt{s_1}\zeta} \tilde{f}  (\zeta^{(2)})| \\
& \phantom{=} + t \di_{\sqrt{s_2}\zeta} \tilde{p} (\zeta) |\di_{0_d} \tilde{f} (\zeta)|  +  \frac{t^{3/2}}{2}  \di_{\sqrt{s_2}\zeta} \tilde{p} (\zeta) |\di^{(2)}_{\sqrt{s_1}\zeta} \tilde{f}  (\zeta^{(2)})| \\
& \le \sqrt{t} \, \tilde{p}(0_d) \| \di_{0_d} \tilde{f}\|_{op} \|\zeta\|  +  \frac{t \, \tilde{p}(0_d) \|\zeta\|^2 }{2}  \sup_{z \in (\iota \circ \tilde{\exp}_x)^{-1}(B_{t^\eta}(x))\backslash W }\|\di^{(2)}_{z} \tilde{f}\|_{op}  \\
& \phantom{=} + \sup_{z \in (\iota \circ \tilde{\exp}_x)^{-1}(B_{t^\eta}(x))\backslash W } \left( t \|\zeta\|^2 \|\di_{0_d} \tilde{f} \|_{op} \| \di_{z} \tilde{p} \|_{op} +  \frac{t^{3/2}\|\zeta\|^3 }{2}  \| \di_{z} \tilde{p}\|_{op}  \|\di^{(2)}_{z} \tilde{f}\|_{op} \right) \\
& \le \sqrt{t} \, \tilde{p}(0_d) \| \di_{0_d} \tilde{f}\|_{op} \|\zeta\|  + t\|\zeta\|^2 C_{x,f,p}(t)
\end{align*}
where we have set
\begin{align*}
C_{x,f,p}(t) & \df \sup_{z \in (\iota \circ \tilde{\exp}_x)^{-1}(B_{t^\eta}(x))\backslash \mathbb{W} } \left( \frac{\tilde{p}(0_d) }{2}  \|\di^{(2)}_{z} \tilde{f}\|_{op}  + \|\di_{0_d} \tilde{f} \|_{op} \| \di_{z} \tilde{p} \|_{op}\right. \\
&  \qquad \qquad \qquad \qquad \qquad \qquad \qquad \qquad \left. + \frac{t^{1+\eta} }{2}  \| \di_{z} \tilde{p}\|_{op}  \|\di^{(2)}_{z} \tilde{f}\|_{op} \right).
\end{align*}
Note that as $t \downarrow 0$,
\[
C_{x,f,p}(t) \to \frac{\tilde{p}(0_d) }{2}  \|\di^{(2)}_{0_d} \tilde{f}\|_{op}  + \|\di_{0_d} \tilde{f} \|_{op} \| \di_{0_d} \tilde{p} \|_{op}.
\]
Then we get 
\begin{align*}
     |II(t)| &  \le  \frac{\tilde{p}(0_d) \| \di_{0_d} \tilde{f}\|_{op} }{\sqrt{t}} \int_{\frac{\tilde{\exp}_x^{-1}(B_{t^\eta}(x))}{\sqrt{t}}\backslash W} \exp\left(-\|\zeta\|^2 \right) \|\zeta\| d\zeta \\
     & + C_{x,f,p}(t) \int_{\frac{\tilde{\exp}_x^{-1}(B_{t^\eta}(x))}{\sqrt{t}}\backslash W} \exp\left(-\|\zeta\|^2 \right) \|\zeta\|^2 d\zeta.
\end{align*}
It follows from \eqref{eq:step2_2} and the dominated convergence theorem that the two previous integrals converge both to $0$ as $t \downarrow 0$. This yields \eqref{eq:step4} as desired. 
\end{proof}

\subsection{Euclidean calculation}  With a view to apply it to the expression of $I(t)$ established in \eqref{eq:I(t)}, we prove the next Euclidean result.  Recall that $c_\ell \df \Gamma((\ell + 1)/2)/2$ for any positive integer $\ell$, where $\Gamma$ is the Gamma function.

\begin{proposition}\label{prop:key}
Let $q \in \cC^{2}_{\ge 0}(\setB^d)$ and $h \in \cC^{3}(\setB^d)$ be such that
\begin{equation}\label{eq:p_f}
C_{q,h} \df \left( \max_{0 \le j \le 2 }  \sup_{\xi \in \setB^d}\|\di^{(j)}_{\xi}q \|_{op} + \max_{1 \le i \le 3 }  \sup_{\xi \in \setB^d}\|\di^{(i)}_{\xi}h \|_{op} \right)  < +\infty.
\end{equation}
Consider a cone $\cC \subset \setR^d$. For $\eta \in (0,1/2)$ and $t \in (0,1)$, define
\[
L_t^{\mathcal{C},\eta}h(0_d)= \frac{1}{t} \int_{ \setB^d_{t^\eta}\cap \mathcal{C}} e^{-\frac{\|y\|^2}{t}} (h(0_d)-h(y))q(y) \di y.
\]
Set $S \cC \df \cC \cap \mathbb{S}^{d-1}.$ Then as $t \downarrow 0$,
\begin{align}\label{eq:eucl_result}
    L_t^{\mathcal{C},\eta}h(0_d) & = - \frac{c_d}{\sqrt{t}} q(0_d) \partial_{v_{\mathcal{C}}} h(0_d)  - c_{d+1} \bigg( q(0_d) A_\cC h(0_d) +  [q,h]_\cC(0_d)\bigg)  + O(\sqrt{t}),
\end{align}
where 
\[
\partial_{v_{\mathcal{C}}}h(0_d) \df d_{0_d} h(v_\cC) \quad \text{with} \quad v_{\mathcal{C}} \df \int_{S^g \cC} \theta \di \sigma (\theta), 
\]
\[
    A_{\cC}h(0_d) \df \frac{1}{2}\int_{S^g\cC} \di_{0_d}^{(2)}h \left( \theta^{(2)} \right) \di \sigma(\theta) \qquad \text{and} \qquad 
    [q,h]_\cC(x) \df  \int_{S^g\cC} \di_{0_d}h(\theta)\di_{0_d}q(\theta) \di \sigma(\theta).
\]
\end{proposition}

To prove this proposition, we need a preliminary lemma. We provide a proof for completeness. To this aim, we recall that the upper incomplete Gamma function is defined by
\[
\Gamma (s,x) \df \int _{x}^{\infty }t^{s-1}\,e^{-t}\,dt
\]
for any $s,x >0$. 

\begin{lemma}\label{lem:O}
Let $h \in \cC(\setR^d)$ be such that there exists $m \in \mathbb{N}$ for which
\[
C_h \df \sup_{z \in \setR^d \backslash \{0\}} \frac{|h(z)|}{\|z\|^m} < +\infty.
\]
Then for any $a\in(-1/2,0)$,
\[
\left| \int_{\cC \backslash \setB^d_{t^a}}  e^{-\|z\|^2}   h(z) \, \di z \right| = \sigma(S\cC) O ( t^{\frac{1}{2}}) \qquad \text{as $t \downarrow 0$.}
\]
\end{lemma}

\begin{proof}
We have
\begin{align*}
\left| \int_{\cC \backslash \setB^d_{t^a}}  e^{-\|z\|^2}   h(z) \, \di z \right| & \le  \int_{\cC \backslash \setB^d_{t^a}}  e^{-\|z\|^2}   |h(z)|  \, \di z \le C_g \int_{\cC \backslash \setB^d_{t^a}}   e^{-\|z\|^2} \|z\|^m   \, \di z.
\end{align*}
Using polar coordinates and then the change of variable $\tau=r^2$, we get
\begin{align*}
\int_{\cC \backslash \setB^d_{t^a}}   e^{-\|z\|^2}  \|z\|^m \, \di z & = \sigma(S\cC) \int_{t^{a}}^{+\infty} e^{-r^2} r^{m+d-1} \, \di r   \\
& =  \frac{\sigma(S\cC)}{2} \int_{t^{2a}}^{+\infty} e^{-\tau} \tau^{\frac{m+d}{2}-1} \, \di \tau \\
& = \frac{\sigma(S\cC)}{2} \, \Gamma \left( \frac{m+d}{2} , t^{2a} \right).
\end{align*}
By asymptotic property of the incomplete Gamma function (see e.g.~\cite{Pinelis})
\[
\Gamma \left( \frac{m+d}{2} , t^{2a} \right) \underset{t \downarrow 0}{\sim}  t^{a(m+d-2)} e^{-t^{2a}}.
\]
Therefore, for small enough $t>0$,
\[
\left| \int_{\cC \backslash \setB^d_{t^a}}  e^{-\|z\|^2}   F(z) \, \di z \right| \le C_g \, \sigma(S\cC) \,  t^{a(m+d-2)} e^{-t^{2a}}.
\]
The result follows from the fact that $C_h t^{a(m+d-2)} e^{-t^{2a}} = O(\sqrt{t})$ as $t \downarrow 0$.
\end{proof}

\begin{remark}
Our proof yields the more precise estimate:
\[
\left| \int_{\cC \backslash \setB^d_{t^a}}  e^{-\|z\|^2}   g(z) \, \di z \right| = \sigma(S\cC) O \left( t^{a(m+d-2)} e^{-t^{2a}}\right) \qquad \text{as $t \downarrow 0$.}
\]
\end{remark}

We are now in a position to prove Proposition \ref{prop:key}.

\begin{proof}
\textbf{Step 1}. Let us prove that, as $t \downarrow 0$,
\begin{align}\label{eq:step1}
    L_t^{\mathcal{C},\eta}h(0_d) & = - t^{-\frac{1}{2}} q(0_d) \di_{0_d}h\left(  \int_\cC e^{-\|z\|^2} z \di z\right) - \int_\cC e^{-\|z\|^2} \di_{0_d}h(z)\di_{0_d}q(z) \di z \nonumber \\
    & \phantom{=} -  \frac{q(0_d)}{2}\di_{0_d}^{(2)} h \left(\int_\cC e^{-\|z\|^2} z^{(2)} \di z \right) + O(\sqrt{t}).
\end{align}
The Taylor theorem with Laplace remainder implies that for any $y \in \setB^d$ there exist $\xi_y, \zeta_y \in \setB^d$ such that
\[
h(0_d)-h(y) = - \di_{0_d} h (y)  -  \frac{1}{2}  \di^{(2)}_{0_d} h  (y^{(2)})  - \frac{1}{6}  \di^{(3)}_{\xi_y} h (y^{(3)}),
\]
\[
q(y) = q(0_d) + \di_{0_d} q (y) +  \frac{1}{2}  \di ^{(2)}_{\zeta_y} q (y^{(2)}).
\]
Then for any $t \in (0,1)$,
\begin{align}\label{eq:key1}
L_t^{\cC,\eta} h(0_d) & = \sum_{i=1}^{3}\sum_{j=0}^{2} I_{t}(i,j) 
 \end{align}
 where for any $i \in \{1,2,3\}$ and $j \in \{0,1,2\}$,
 \[
 I_{t}(i,j)   \df - \frac{1}{i!j!t^{\frac{d}{2}+1}} \int_{\mathcal{C} \cap \setB^d_{t^\eta}} e^{-\frac{|y|^2}{t}} h_{i,j}(y) \di y.
 \]
Here we have defined, for any $y \in \setB^d$,
\[
\begin{array}{ll}
  h_{i,j}(y) \df \di_{0_d}^{(i)} h (y^{(i)})   \di_{0_d}^{(j)} q(y^{(j)})   &  \text{for $(i,j)  \in \{1,2\} \times \{0,1\}$}, \\
h_{3,j}(y) \df \di_{\xi_y}^{(3)} h (y^{(3)})   \di_{0_d}^{(j)} q(y^{(j)})  & \text{for $j\in \{0,1\}$},\\ 
h_{i,2}(y) \df \di_{0_d}^{(i)} h (y^{(i)})   \di_{\zeta_y}^{(2)} q(y^{(2)}) & \text{for $i \in \{1,2\}$},\\
h_{3,2}(y) \df \di_{\xi_y}^{(3)} h (y^{(3)})   \di_{\zeta_y}^{(2)} q(y^{(2)}). & 
\end{array}
\]
On one hand, if $i=3$ or $j=2$, we know from \eqref{eq:p_f} that for any $y \in \setB^d$,
\[
|h_{i,j}(y)| \le C_{q,h}|y|^{i+j}
\]
hence 
\begin{align*}
    |I_{t}(i,j)|  & \le\frac{1}{t^{\frac{d}{2}+1}} \int_{\mathcal{C} \cap \setB^d_{t^\eta}} e^{-\frac{|y|^2}{t}} |h_{i,j}(y)| \di y \le \frac{C_{q,h}}{t^{\frac{d}{2}+1}} \int_{\mathcal{C} \cap \setB^d_{t^\eta}} e^{-\frac{|y|^2}{t}} |y|^{i+j} \di y \\
    & = C_{q,h} t^{\frac{i+j}{2}-1} \int_{\mathcal{C} \cap \setB^d_{t^{\eta-1/2}}} e^{-\|z\|^2} \|z\|^{i+j} \di z \le C_{q,h} C_{\mathcal{C},i+j} t^{\frac{i+j}{2}-1}
\end{align*}
where we have used the change of variable $z= y/\sqrt{t}$ and defined the constant
\[
 C_{\mathcal{C},k} \df \int_{\mathcal{C}} e^{-\|z\|^2} \|z\|^{k} \di z = \sigma(S\cC) \Gamma((k+d)/2)
\]
for any integer $k$. Since $i=3$ or $j=2$ implies $i+j\ge 3$, we have $(i+j)/2-1 \ge 3/2-1 = 1/2$, so that
\begin{align}\label{eq:key2}
\sum_{(i,j) \notin \{1,2\} \times  \{0,1\}} I_{t}(i,j) = O(\sqrt{t}) \qquad \text{as $t \downarrow 0$.}
\end{align}

On the other hand, if $(i,j) \in \{1,2\} \times  \{0,1\}$, then for any $z \in \setR^d$ and $t>0$,
\[
h_{i,j}(\sqrt{t}z) = t^{\frac{i+j}{2}} h_{i,j}(z) \quad \text{and} \quad  |h_{i,j}(z)| \le C_{q,h} \|z\|^{i+j} \text{ by \eqref{eq:p_f}}.
\]
Then the change of variable $z= y/\sqrt{t}$ yields that
\begin{align*}
    I_{t}(i,j)  & = \frac{1}{t} \int_{\mathcal{C} \cap \setB^d_{t^{\eta-1/2}}} e^{-\|z\|^2} h_{i,j}(\sqrt{t}z) \di z = t^{\frac{i+j}{2}-1} \int_{\mathcal{C} \cap \setB^d_{t^{\eta-1/2}}} e^{-\|z\|^2} h_{i,j}(z) \di z.
\end{align*}
By Lebesgue dominated convergence theorem,
\[
\int_{\mathcal{C} \cap \setB^d_{t^{\eta-1/2}}} e^{-\|z\|^2} h_{i,j}(z) \di z \to \int_{\mathcal{C}} e^{-\|z\|^2} h_{i,j}(z) \di z \qquad \text{ as $t \downarrow 0$.}
\]
Moreover, by Lemma \ref{lem:O},
\begin{align*}
\left| \int_{\mathcal{C} \backslash  \setB^d_{t^{\eta-1/2}}} e^{-\|z\|^2} h_{i,j}(z) \di z \right| = O( t^{\frac{1}{2}} ).
\end{align*}
Then 
\[
I_{t}(i,j) = t^{\frac{i+j}{2}-1}  \int_{\mathcal{C} } e^{-\|z\|^2} h_{i,j}(z) \di z + O( \sqrt{t} ).
\]
Now
\begin{align*}
\int_{\mathcal{C} } e^{-\|z\|^2} h_{1,0}(z) \di z & = q(0_d)  \int_\cC e^{-\|z\|^2} \di_{0_d}h(z) \di z \\
\int_{\mathcal{C} } e^{-\|z\|^2} h_{1,1}(z) \di z & =  \int_\cC e^{-\|z\|^2} \di_{0_d}h(z)\di_{0_d}q(z) \di z \\
\int_{\mathcal{C} } e^{-\|z\|^2} h_{2,0}(z) \di z & = q(0_d)\int_\cC e^{-\|z\|^2} \di_{0_d}^{(2)}h(z^{(2)}) \di z .
\end{align*}
This implies the desired result.

\textbf{Step 2.}
Using polar coordinates, we can write
\begin{align*}  \int_\cC e^{-\|z\|^2} z \di z =  \int_0^{+\infty} \int_{S^g\cC} e^{-\rho^2} \rho \theta \rho^{d-1} \di \rho \di \sigma(\theta) & = \left(\int_0^{+\infty} e^{-\rho^2} \rho^{d} \di \rho \right)\underbrace{\int_{S^g\cC} \theta \di \sigma(\theta)}_{=v_\cC}.
\end{align*}
The change of variable $\tau = \rho^2$ yields that 
\[
\int_0^{+\infty} e^{-\rho^2} \rho^{d}  \di \rho = \Gamma\left(\frac{d+1}{2}\right) = c_{d}.
\]
In the end, we get
\[
\int_\cC e^{-\|z\|^2} z \di z = c_{d} v_\cC.
\]
In a similar way,
\begin{align*}
\int_\cC e^{-\|z\|^2} z^{(2)} \di z & = \int_0^{+\infty} \int_{S^g\cC} e^{-\rho^2} (\rho \theta)^{(2)} \rho^{d-1} \di \rho \di \sigma(\theta)  = \int_0^{+\infty} \int_{S^g\cC} e^{-\rho^2} \theta^{(2)} \rho^{d+1} \di \rho \di \sigma(\theta)\\
 & = \left( \int_0^{+\infty} e^{-\rho^2} \rho^{d+1} \di \rho \right)\int_{S^g\cC} \theta^{(2)} \di \sigma(\theta) = c_{d+1} \int_{S^g\cC} \theta^{(2)} \di \sigma(\theta)
\end{align*}
and
\begin{align*}
\int_\cC e^{-\|z\|^2} \di_{0_d}h(z)\di_{0_d}q(z) \di z & = \int_0^{+\infty} \int_{S^g\cC} e^{-\rho^2} \di_{0_d}h(\rho \theta)\di_{0_d}q(\rho \theta) \rho^{d-1} \di \rho \di \sigma(\theta)\\
& =  \int_0^{+\infty} \int_{S^g\cC} e^{-\rho^2} \di_{0_d}h( \theta)\di_{0_d}q( \theta) \rho^{d+1} \di \rho \di \sigma(\theta)\\
& = \left( \int_0^{+\infty}  e^{-\rho^2} \rho^{d+1} \di \rho  \right) \int_{S^g\cC} \di_{0_d}h( \theta)\di_{0_d}q( \theta) \di \sigma(\theta)\\
& = c_{d+1}  \int_{S^g\cC} \di_{0_d}h( \theta)\di_{0_d}q( \theta) \di \sigma(\theta).
\end{align*}
Combined with \eqref{eq:step1}, the three previous calculations yield \eqref{eq:eucl_result} as desired.
\end{proof}

\begin{remark}\label{rem:further} The previous proof may be easily modified to get the following: for any integer $N \ge 2$,  let $q \in \cC^{N}_{\ge 0}(\setB^d)$ and $h \in \cC^{N+1}(\setB^d)$ be such that
\[
 \max_{0 \le j \le N}  \sup_{\xi \in \setB^d}\|\di^{(j)}_{\xi}q \|_{op} + \max_{1 \le i \le N+1 }  \sup_{\xi \in \setB^d}\|\di^{(i)}_{\xi}h \|_{op}   < +\infty.
\]
Then for any cone $\mathcal{C}$ in $\setR^d$ and $\eta \in (0,1/2)$, as $t \downarrow 0$
\[
L_t^{\cC,\eta}h(0_d) = \sum_{i=1}^{N+1} \sum_{j=0}^N \frac{t^{\frac{i+j}{2}-1}c_{d+i+j-1}}{i!j!} \int_{S^g\cC} \di_{0_d}^{(i)} h (\theta^{(i)})   \di_{0_d}^{(j)} q(\theta^{(j)}) \di \sigma(\theta) + O(t^{\frac{N-1}{2}}).
\]
Moreover, it should be pointed out that if $q(0_d)=0$, then $L_t^{\cC,\eta} h(0_d)$ always converges to $$ \int_\cC e^{-\|z\|^2} \di_{0_d}h(z)\di_{0_d}q(z) \di z.$$
\end{remark}

\subsection{Conclusion}

Let us now explain how to reach the conclusion of Theorem \ref{th:1}. From Lemma \ref{lem:localisation} and Proposition \ref{prop:localized}, we get
\[
L_tf(x) = \frac{1}{t} \int_{\mathbb{B}^d_{t^{\eta}}\cap \tilde{F}_{0_d}(\bar{\Omega})} e^{-\frac{\|\xi\|^2}{t}} (\tilde{f}(0_d) - \tilde{f}(\xi))q(\xi)\di \xi + o\left( \frac{1}{\sqrt{t}}\right) +  O(t^{-d/2-1}e^{-t^{2 \eta -1}}) \qquad \text{as $t\to 0$.}
\]

Since $\tilde{F}_{0_d}(\bar{\Omega}) \subset T_{0_d} \bar{\Omega}$, it follows from the definition of cusps that the previous integral is zero if $x$ is a cusp. This proves Theorem \ref{th:1} in this case, because all the integrals over $S^gI_xM$ vanish as well.

As for the other two cases, it follows from \eqref{eq:interiorF} for interior points, and from Proposition \ref{lem:tangent cones as epigraphs} and Proposition \ref{prop:at boundary of open sets, inward sector equals tangent cone} for LCDD border points, that we can replace $\tilde{F}_{0_d}(\bar{\Omega})$ by $I_{0_d}\bar{\Omega}$ in the previous domain of integration. Then we apply Proposition \ref{prop:key} with $\cC = I_{0_d}\bar{\Omega}$ to obtain that
\begin{align*}
    \frac{1}{t} \int_{\mathbb{B}^d_{t^{\eta}}\cap I_{0_d}\bar{\Omega}} e^{-\frac{\|\xi\|^2}{t}} (\tilde{f}(0_d) - \tilde{f}(\xi))q(\xi)\di \xi
\end{align*}
equals 
\[
- \frac{c_d}{\sqrt{t}} q(0_d) \partial_{v_{I_{0_d}\bar{\Omega}}} \tilde{f}(0_d)  - c_{d+1} \bigg(q(0_d) A_{I_{0_d}\bar{\Omega}} \tilde{f}(0_d) +  [q,\tilde{f}]_{I_{0_d}\bar{\Omega}}(0_d)\bigg)  + O(\sqrt{t})
\]
as $t \downarrow 0$. But
\[
q(0_d) = p(\tilde{\exp}_x(\iota(0_d))) \sqrt{\det g(\iota(0_d))} = p(x)
\]
and
\begin{align*}
\partial_{v_{I_{0_d}\bar{\Omega}}} \tilde{f}(0_d) = d_{0_d}\tilde{f}(v_{I_{0_d}\bar{\Omega}}) & = d_xf \left( d_{0_d}\tilde{\exp}_x\left( \int_{S^g I_{0_d}\bar{\Omega}} \theta \di \sigma (\theta)\right)\right)\\
& =
d_xf \left( \int_{S^g I_{0_d}\bar{\Omega}} d_{0_d}\tilde{\exp}_x(\theta) \di \sigma (\theta)\right)\\
& =
d_xf \left( \int_{S^gI_xM} \theta \di \sigma(\theta)\right) = d_xf(v_g(x)) = \partial_{v_g(x)}f(x),
\end{align*}
where we have used Lemma \ref{prop:inward sectors correspond to inward sectors via charts} to perform the change of variable yielding the last line. Likewise, we easily obtain that
\[
A_\cC \tilde{f}(0_d) = A_gf(x),
\]
\begin{align*}
[q,\tilde{f}]_{I_{0_d}\bar{\Omega}}(x) = [p,f]_g(x).
\end{align*}

\section{Asymptotic behavior of the extrinsic Gaussian Operator}\label{sec:extrinsic}

In this section, we prove Theorem \ref{th:2} without the refined estimates on the error term $\mathrm{Err}(t)$. See Section \ref{sec:refined} for these refined estimates. Let $M\subset \cl^D$ be a smooth $d$-dimensional submanifold with kinks endowed with the Riemannian metric $g$ induced by the Euclidean scalar product of $\cl^D$. Recall that the extrinsic Gaussian operator at time $t>0$  associated with a $\cC^2$ density $p$ on $M$ is given by
\[
L_t f(x) = \frac{1}{t^{d/2+1}} \int_M e^{-\frac{\|x-y\|_{\setR^D}^2}{t}} (f(x) - f(y)) p(y) \, \di \mathrm{vol_g}(y)
\]
for any $x \in M$ and $f \in \cC^3(M) \cap L^1(M,p\,\mathrm{vol}_g)$. Let $x \in M$ be either an interior point, a LCDD border point, or a cusp. 

Consider the orthogonal projection $\pi$ from $\setR^D$ onto the translated tangent space $x + T_xM$. Applying a rigid motion if needed, we can assume that $x=0_D$ and $x+T_xM=\setR^d \times \{0_{D-d}\} \simeq \setR^d$, so that $\pi$ can be seen as the orthogonal projection mapping a vector of $\setR^D$ onto its first $d$ coordinates. We let $\epsilon>0$ be such that $\pi$ is a smooth diffeomorphism, in the sense of (iii) in Definition \ref{def:diffeoMFK}, of $\mathbb{B}_\epsilon^D(x) \cap M$ onto its image
 $$\Omega \df \pi (\mathbb{B}_\epsilon^D(x) \cap M) \subset \ddim.$$
 Following \cite{CoifmanLafon2006}, for a generic $y \in \mathbb{B}_\epsilon^D \cap M$ we set $$u=(u_1 \dots u_d):=\pi(y).$$ Note that $\pi$ acts as a local chart centered at $x,$ so that $int(\Omega)$ is an open subset of $\ddim$ with $0_d=\pi(x)$ as interior or $\cC^0$ boundary point. In the latter case, $0_d$ is either a LCDD boundary point or a cusp for $\Omega$.
 
 Below, we express the Euclidean distance and the Riemannian volume measure on $M$ in the $u$-coordinates introduced above. We let $Q_{x,m}$ denote a generic homogeneous polynomial of degree $m.$

\begin{lemma}\label{lem:prep}
As $t \downarrow 0$, for any $u \in \pi(\mathbb{B}_{t^\eta}^D(x) \cap M)$,
 \begin{align*}
    \norm{\pi^{-1}(u)}_{\setR^D}^2 &= \norm{u}_{\ddim}^2 + Q_{x,4}(u) + Q_{x,5}(u) + O(t^{6\eta}) && \text{(metric comparison)} \\
    \rho(u) 
    &= 1 + Q_{x,2}(u) + Q_{x,3}(u) + O(t^{4\eta}) && \text{(infinitesimal volume comparison)}
\end{align*}
where $\rho \in L^1(\Omega,\mathcal{L}^d)$ is the density of the push-forward measure $\pi_\# \mathrm{vol}_g$ on $\Omega$ with respect to $\mathcal{L}^d$.
\end{lemma}
 
\begin{proof}
    It follows line by line from \cite[Appendix B, Lemma 7]{CoifmanLafon2006}, with the assumption that $\norm{y-x}< t^{\eta}$ instead of $t^{1/2}$ like they did, which and whose implications are both indeed \textit{weaker}, since $t^{\eta}>t^{1/2}$ for small $t>0$.
\end{proof}

We are now ready to prove Theorem \ref{th:2}. Recall that we consider $\eta \in (1/6, 1/2)$ and set $\tilde{f} = f \circ \pi^{-1}$ and $\tilde{p} = p \circ \pi^{-1}$.

\begin{proof}

By Lemma \ref{lem:localisation}, taking $\di$ to be the extrinsic Euclidean distance on $M,$ we can write 
\begin{equation}\label{eq:locali}
    L_tf(x)= \bar{L}_tf(x) + O(t^{-d/2-1}e^{-2\eta -1})
\end{equation}
with
\[
\bar{L}_tf(x):= \frac{1}{t^{d/2+1}}\int_{M \cap \setB^D_{t^{\eta}}(x)} e^{-\frac{\norm{x - y}_{\setR^D}^2}{t}}(f(x) - f(y)) p(y) \di \mathrm{vol}_g(y).
\]
Consider $t>0$ sufficiently small to guarantee that $t^{\eta}< \epsilon$. Since $\pi$ is a diffeomorphism from $M\cap \mathbb{B}^D_{t^{\eta}}(x) = M\cap \mathbb{B}^D_{\epsilon}(x) \cap \mathbb{B}^D_{t^{\eta}}(x)$ onto $\Omega \cap \pi (\mathbb{B}^D_{t_\eta}(x)) = \Omega \cap \mathbb{B}^d_{t_\eta}$, we can use the change of variable $\pi(y-x)=u$ and Lemma \ref{lem:prep} to get
\begin{align}
    \begin{split}
        t^{d/2+1} \bar{L}_t f(x)
        &= \int_{\Omega \cap \mathbb{B}^d_{t_\eta}} e^{-\left( \frac{\norm{u}_{\setR^d}^2}{t} + \frac{Q_{x,4}(u) + Q_{x,5}(u) + O(t^{6\eta})}{t} \right)} (\tilde{f}(0_d) - \tilde{f}(u))\tilde{p}(u)\\
        & \phantom{=} \times \left( 1 + Q_{x,2}(u) + Q_{x,3}(u) + O(t^{4\eta})\right) \, \di u.
    \end{split}
\end{align}
Set $\zeta:= u/\sqrt{t}$ and divide both sides by $t^{d/2+1}$ to obtain
\begin{align}
    \begin{split}
        \bar{L}_t f(x)
        &= \frac{1}{t} \int_{\frac{\Omega}{\sqrt{t}} \cap \mathbb{B}^d_{t^{\eta-1/2}}} e^{-\left( \norm{\zeta}^2 + \left( \frac{Q_{x,4}(\sqrt{t}\zeta) + Q_{x,5}(\sqrt{t}\zeta) + O(t^{6\eta})}{t} \right)  \right)} (\tilde{f}(0_d) - \tilde{f}(\sqrt{t}\zeta))\tilde{p}(\sqrt{t}\zeta)\\
        &\phantom{=}  \times \left( 1 + Q_{x,2}(\sqrt{t}\zeta) + Q_{x,3}(\sqrt{t}\zeta) + O(t^{4\eta}) \right)  \, \di \zeta \\
        &= \frac{1}{t} \int_{\frac{\Omega}{\sqrt{t}} \cap \mathbb{B}^d_{t^{\eta-1/2}}} e^{-\left( \norm{\zeta}^2 +  tQ_{x,4}(\zeta) + t^{3/2}Q_{x,5}(\zeta) + O(t^{6\eta-1}) \right)  } (\tilde{f}(0_d) - \tilde{f}(\sqrt{t}\zeta))\tilde{p}(\sqrt{t}\zeta)\\
        &\phantom{=} \times \left( 1 + tQ_{x,2}(\zeta) + t^{3/2}Q_{x,3}(\zeta) + O(t^{4\eta} \right)  \, \di \zeta
    \end{split}
\end{align}
where we use the homogeneity of the polynomials $Q_m$. Now
\begin{align*}
    e^{-\left( \norm{\zeta}^2 + \left( tQ_{x,4}(\zeta) + t^{3/2}Q_{x,5}(\zeta) + O(t^{6\eta-1})\right)  \right)} & =
e^{-\norm{\zeta}^2} \left( 1 + O(t^{min(1, 6\eta-1)})\right),\\
e^{-\norm{\zeta}^2} (1 + t Q_{x,2}(\zeta) + t^{3/2} Q_{x,3}(\zeta) + O(t^{4\eta})) & =e^{-\norm{\zeta}^2} ( 1 + O\big(t^{\min(1, 4\eta)}\big)),
\end{align*}
where both $O$ are independent of $\zeta \in \setR^d$. Set $\alpha:=\min(1, 6\eta-1)$ and $\beta:=\min(1,4\eta)$. Then $\alpha>0$ and $\beta>2/3$ because $\eta > 1/6.$  With the above simplification and symbols, we get 
\begin{align}\label{eq:I(t)}
    \begin{split}
        \bar{L}_t f(x)
        &= \int_{\frac{\Omega}{\sqrt{t}} \cap \mathbb{B}^d_{t^{\eta-1/2}}} e^{-\norm{\zeta}^2}\left(1 + O(t^{\alpha}) \right) \frac{\tilde{f}(0_d) - \tilde{f}(\sqrt{t}\zeta)}{t}\tilde{p}(\sqrt{t}\zeta)(1 + O(t^{\beta})) \di \zeta\\
        & = (1+O(t^{\alpha}))\int_{\frac{\Omega}{\sqrt{t}} \cap \mathbb{B}^d_{t^{\eta-1/2}}} e^{-\norm{\zeta}^2}\frac{\tilde{f}(0_d) - \tilde{f}(\sqrt{t}\zeta)}{t}\tilde{p}(\sqrt{t}\zeta)(1 + O(t^{\beta})) \di \zeta\\
        & =: (1+O(t^{\alpha}))I(t).
    \end{split}
\end{align}
Using the Taylor expansion of $\tilde{f}$ and $\tilde{p}$, we get
\begin{align*}
    \begin{split}
        I(t)
        &=\int_{\frac{\Omega}{\sqrt{t}} \cap \mathbb{B}^d_{t^{\eta-1/2}}} e^{-\norm{\zeta}^2}  \left(-\frac{1}{\sqrt{t}} \nabla \tilde{f}(0_d) \cdot \zeta - \frac{1}{2}\text{Hess} \, \tilde{f}(0_d)(\zeta, \zeta) + o(1) \norm{\zeta}^2 \right) \\
        &\quad \times  (\tilde{p}(0_d) + \sqrt{t} \nabla \tilde{p}(0_d) \cdot \zeta + o(\sqrt{t})\norm{\zeta} ) (1 + O(t^{\beta}))\, \di \zeta. \\
    \end{split}
\end{align*}
Applying the asymptotic equalities
\[
1_{\frac{\Omega}{\sqrt{t}}}(\zeta) \, 1_{\mathbb{B}^d_{t^{\eta-1/2}}}(\zeta) = (1_{T^B_0(\Om)}(\zeta)+o(1))(1+o(1)) = 1_{T^B_0(\Om)}(\zeta)+o(1) \qquad \text{as $t \downarrow 0$,}
\]
the previous rewrites as
\begin{align}\label{eq:II(t)}
    \begin{split}
        I(t)
        &= \int_{T^B_0(\Om)} e^{-\norm{\zeta}^2} \left(-\frac{1}{\sqrt{t}} \nabla \tilde{f}(0_d) \cdot \zeta - \frac{1}{2}\text{Hess} \, \tilde{f}(0_d)(\zeta, \zeta) + o(1)\norm{\zeta}^2 \right) \\
        &\quad \times  (\tilde{p}(0_d) + \sqrt{t} \nabla \tilde{p}(0_d) \cdot \zeta + o(\sqrt{t})\norm{\zeta} ) (1 + O(t^{\beta})) \, \di \zeta + o\left( \frac{1}{\sqrt{t}}\right) \\
        &= \int_{T^B_0(\Om)} e^{-\norm{\zeta}^2} \left(-\frac{1}{\sqrt{t}} \nabla \tilde{f}(0_d) \cdot \zeta - \frac{1}{2}\text{Hess} \, \tilde{f}(0_d)(\zeta, \zeta) + o(1)\norm{\zeta}^2 \right) \\
        &\quad \times  (\tilde{p}(0_d) + \sqrt{t} \nabla \tilde{p}(0_d) \cdot \zeta + o(\sqrt{t})\norm{\zeta} )  \, \di \zeta + O(t^{\beta-1/2}) + o\left( \frac{1}{\sqrt{t}}\right)\\
        & =: II(t) + O(t^{\beta-1/2}).
    \end{split}
\end{align}
where we use the regularity of $f$ and $p$, which transfers to $\tilde{f}$ and $\tilde{p}$, to justify the $o(1/\sqrt{t})$ term in the second line. Now
\begin{align}\label{eq:II(t)=}
    \begin{split}
      II(t)
      &= -\frac{1}{\sqrt{t}} \tilde{p}(0_d) \nabla \tilde{f}(0_d) \cdot \int_{T^B_0(\Om)} \zeta  e^{-\norm{\zeta}^2} \di \zeta - \frac{1}{2}\int_{T^B_0(\Om)} e^{-\norm{\zeta}^2} \tilde{p}(0_d)\, \text{Hess} \, \tilde{f}(0_d)(\zeta, \zeta)\di \zeta \\
      & \quad  - \int_{T^B_0(\Om)} e^{-\norm{\zeta}^2} (\nabla \tilde{f}(0_d) \cdot \zeta) (\nabla \tilde{p}(0_d) \cdot \zeta)  \di \zeta + O(\sqrt{t}) + o\left( \frac{1}{\sqrt{t}}\right)\\
      & =  -\frac{c_d}{\sqrt{t}} ( \tilde{p}(0_d) \, \partial_{M}\tilde{f}(0_d) + o(1)) - c_{d+1} \bigg( \tilde{p}(0_d) A_M\tilde{f}(0_d) +  [\tilde{p},\tilde{f}]_M(0_d) \bigg) + O(\sqrt{t}),
    \end{split}
\end{align}
where we proceed as in the end of the previous section to obtain the last line. From \eqref{eq:I(t)} and \eqref{eq:II(t)}, we get that
\begin{align*}
    \bar{L}_t f(x) & = (1+O(t^\alpha))(II(t)+O(t^{\beta-1/2}))= II(t)[1+o(1)] + O(t^{1/6}),
\end{align*}
where we use that $\alpha>0$ and $\beta >1/6$ to get the second equality. Together with \eqref{eq:locali} and \eqref{eq:II(t)=}, the latter implies the desired result.
\end{proof}

\section{Refined estimates}\label{sec:refined}

In this section, we establish the refined estimates on $\mathrm{Err}(t)$ appearing in Theorem \ref{th:1} and Theorem \ref{th:2}. We provide a proof in the context of Theorem \ref{th:1} only, because the proof for Theorem \ref{th:2} follows along the same lines. Consider a smooth $d$-dimensional Riemannian manifold with kinks $M$ endowed with a $\cC^2$ Riemannian metric $g$, a density $p \in \cC_{\ge 0}^2(M)$, a function $f \in \cC^3(M)\cap L^1(M,p\,\mathrm{vol}_g)$, an exponent $\eta  \in (0,1/2)$, and a point $x \in M$. Let $(\tilde{M},\tilde{g})$ be the open Riemannian manifold of which $(M,g)$ is a submanifold, given by Theorem \ref{lem:Embedding a manifold with kinks into a manifold without boundary}. Consider $R>0$ such that the associated exponential map $\tilde{\exp}_x$ is a diffeomorphism from the $\tilde{g}_x$-ball of radius $R$ in $T_xM$ onto $\tilde{B}_R(x)$. Let $W$ be the set of initial velocities in $T_xM$ that yield a $\tilde{g}$-geodesic that coincide with the $g$-one on a neighborhood of $0$, see \eqref{eq:W}. It follows from Section \ref{sec:intrinsic} (see \eqref{eq:error},\eqref{eq:II} in particular) that the error term $\mathrm{Err}(t)$ is equal to
\[
\sqrt{t} II(t) = \frac{1}{t^{(d+1)/2}} \int_{B_{t^\eta}(x)\backslash \tilde{\exp}_x(W\cap \mathbb{B}_R^d)} \exp\left(-\frac{\di^2(x,y)}{t}\right) (f(x)-f(y)) p(y) \, d\mathrm{vol}_g(y).
\]

\subsection{Interior points} Assume that $x$ is an interior point. Then $I_xM=T_xM$ so $v_g(x)=0$ by symmetry. Moreover, $W = T_xM$, so $$B_{t^\eta}(x) \subset \tilde{\exp}_x(W \cap \mathbb{B}^d_R),$$ then the previous integration takes place on the empty set. Therefore, $II(t)$ and then $\mathrm{Err}(t)$ is zero for any small enough $t>0$. Thus
 \begin{align*}
    L_tf(x) & = \frac{1}{2} \Gamma\left(\frac{d}{2}+ 1\right) \bigg( p(x) A_gf(x) +  [p,f]_g(x) \bigg) + O(\sqrt{t}) + O(t^{-d}e^{-t^{2\eta-1}})
    \end{align*}
    as $t \downarrow 0$. That 
\[
\frac{1}{2} \Gamma\left(\frac{d}{2}+ 1\right) \bigg( p(x) A_gf(x) +  [p,f]_g(x) \bigg) = -\frac{\pi^{d/2}}{4} \, \Delta_{M,p}f(x)
\]
is a consequence of a direct calculation involving the second moments (see e.g.~\cite[p.~33--34]{MT})
\[
\int_{S^g_xM} \theta_i \theta_j \sigma(\theta) =  \frac{\delta_{ij}\mathcal{H}^{d-1}(\mathbb{S}^{d-1})}{d} \, \cdot 
\]

\subsection{$\cC^1$ boundary points} If $x$ is a $\cC^1$ boundary point, we obtain
\[
I_xM = \{ v \in T_xM : g(\partial_\nu f(x),v)\ge 0\}
\]
where $\partial_{\nu} f(x)$ is the inner normal derivative of $f$ at $x$,  hence $\partial_{v_g(x)} f(x)$ coincides with $\partial_{\nu} f(x)$. Then we act as in the previous subsection to get the right constant in \eqref{eq:c1bound}

\subsection{LTCDD points}

Now, assume that $x$ is an LTCDD point. Then there exist $\delta>0$ and $\gamma : \setR^{d-1} \to \setR$ with $\gamma(0_{d-1})=0$ such that the set $\Omega \df (\iota \circ \tilde{\exp}_x)^{-1}(B_R(x))$ satisfies
$$
 \begin{cases}
        \bar{\Omega} \cap  \setB_\delta^d  = epi(\gamma)  \cap \setB_\delta^d, \nonumber\\
        \Omega \cap \setB_\delta^d  = \mathring{epi}(\gamma)  \cap \setB_\delta^d.
        \end{cases}
$$
Moreover, $\gamma$ admits first and second directional derivatives at $0_{d-1}$ in any direction $v' \in \setR^{d-1}$, and these derivatives are continuous with respect to $v'$. 
For any $\theta \in \mathbb{S}^{d-2}$, we set
\[
a(\theta):= \gamma'(0_{d_1};\theta), \qquad 
\kappa(\theta):=\gamma''(0_{d-1};\theta).
\]
We consider the number
\begin{equation}\label{eq:slicewise}
SC(x):=\int_{\mathbb{S}^{d-2}}
\frac{|\kappa(\theta)|}{\big(1+a(\theta)^{2}\big)^{\frac{d}{2}+1}}
\,d\sigma(\theta).
\end{equation}
Note that $SC(x)$ does not depend on the choice of $\gamma$ as above. We call this number the total slicewise curvature of $\partial M$ at $x$.

\textbf{Step 1.} We show that for any small enough $t>0$,
\begin{equation}\label{eq:step1sec6}
|\mathrm{Err}(t)| \le 2 p(x) \left( \|d_xf\|_{g_x} S_t + \sqrt{t}\|d_x^{(2)}f\|_{g_x}\tilde{S}_t \right)
\end{equation} 
where
\[
S_t \df \int_{\frac{\Omega \cap \mathbb{B}_\delta^d}{\sqrt{t}} \triangle T_{0_{d-1}}^B \bar{\Omega}} \|\zeta\| e^{-\|\zeta\|^2}\di \zeta, \qquad \tilde{S}_t \df \int_{\frac{\Omega \cap \mathbb{B}_\delta^d}{\sqrt{t}} \triangle T_{0_{d-1}}^B \bar{\Omega}} \|\zeta\|^2 e^{-\|\zeta\|^2}\di \zeta,
\]
 and $\triangle$ stands for the symmetric difference.

Acting as in \eqref{eq:obtained}, we get that for any small enough $t>0$,
\begin{align*}
|\mathrm{Err}(t)| = \sqrt{t} |II(t)| & \le  \frac{\|p\|_{L^\infty(B_{t^\eta}(x))}}{t^{(d+1)/2}} \int_{B_{t^\eta}(x)\backslash \tilde{\exp}_x(W\cap \mathbb{B}_R^d)} \exp\left(-\frac{\tilde{\di}^2(x,y)}{t}\right) |f(x)-f(y)| \, d\mathrm{vol}_g(y)
\\
& = \frac{\|p\|_{L^\infty(B_{t^\eta}(x))}}{t^{(d+1)/2}} \int_{(\iota \circ \tilde{\exp}_x)^{-1}(B_{t^\eta}(x))\backslash \mathbb{W}} \exp\left(-\frac{\|\xi\|^2}{t}\right) |\tilde{f}(0)-\tilde{f}(\xi)| d\xi.
\end{align*}
The second-order Taylor expansion  of $\tilde{f}$ at $0_d$ implies that for any $\xi \in (\iota \circ \tilde{\exp}_x)^{-1}(B_{t^\eta}(x))$,
\[
|\tilde{f}(0_d)-\tilde{f}(\xi)| \le  \|\nabla \tilde{f}(0_d)\|\|\xi\| + \|\xi\|^2 \Lambda(t)
\]
with $\Lambda(t) \df \sup_{(\iota \circ \tilde{\exp}_x)^{-1}(B_{t^\eta}(x))}\|\mathrm{Hess} \tilde{f}\|$. Since  $\|\nabla \tilde{f}(0_d)\|=\|d_xf\|_{g_x}$, we get that
\begin{align*}
|\mathrm{Err}(t)|  & \le  \frac{\|p\|_{L^\infty(B_{t^\eta}(x))}\|d_xf\|_{g_x}}{t^{(d+1)/2}} \int_{(\iota \circ \tilde{\exp}_x)^{-1}(B_{t^\eta}(x))\backslash \mathbb{W}} \exp\left(-\frac{\|\xi\|^2}{t}\right) \|\xi\| d\xi \\
& + \frac{\|p\|_{L^\infty(B_{t^\eta}(x))}\Lambda(t)}{t^{(d+1)/2}} \int_{(\iota \circ \tilde{\exp}_x)^{-1}(B_{t^\eta}(x))\backslash \mathbb{W}} \exp\left(-\frac{\|\xi\|^2}{t}\right) \|\xi\|^2 d\xi.
\end{align*}
Performing the change of variable $\zeta=\xi/\sqrt{t}$, and using that $\mathbb{W}$ is a cone to get
\[
\frac{(\iota \circ \tilde{\exp}_x)^{-1}(B_{t^\eta}(x))\backslash \mathbb{W}}{\sqrt{t}} = \frac{(\iota \circ \tilde{\exp}_x)^{-1}(B_{t^\eta}(x))}{\sqrt{t}}\backslash \mathbb{W} \subset \frac{\Omega \cap \mathbb{B}_\delta^d}{\sqrt{t}} \backslash \mathbb{W},
\]
we obtain
\begin{align*}
|\mathrm{Err}(t)|  & \le  \|p\|_{L^\infty(B_{t^\eta}(x))}\|d_xf\|_{g_x} \int_{\frac{\Omega \cap \mathbb{B}_\delta^d}{\sqrt{t}}\backslash \mathbb{W}} \exp\left(-\|\zeta\|^2\right) \|\zeta\| d\zeta\\
& + \sqrt{t}\|p\|_{L^\infty(B_{t^\eta}(x))}\Lambda(t)\int_{\frac{\Omega \cap \mathbb{B}_\delta^d}{\sqrt{t}}\backslash \mathbb{W}} \exp\left(-\|\zeta\|^2\right) \|\zeta\|^2 d\xi.
\end{align*}
Since $\mathbb{W}=T_{0_{d-1}}^B \bar{\Omega}$, we can bound the previous integrals from above by $S_t$ and $\tilde{S}_t$ respectively. Then the conclusion \eqref{eq:step1sec6} follows from the obvious facts :
\[
\|p\|_{L^\infty(B_{t^\eta}(x))} \to p(x), \qquad \Lambda(t) \to \|d_x^{(2)}f\|_{g_x}.
\]

 \textbf{Step 2.} We show that for small enough $t>0$,
\begin{equation}\label{eq:step2sec6}
\frac{S_t}{\sqrt{t}} \le \Gamma(d+2)SC(x).
\end{equation}

Write $\zeta=(\zeta',z)$ with $\zeta'=r\,\theta$ for some $r\ge 0$ and $\theta\in \mathbb{S}^{d-2}$.
Then the Lebesgue measure decomposes as $$d\zeta=r^{d-2}\,dr\,d\sigma(\theta)\,dz.$$
Since $\bar{\Omega} \cap \mathbb{B}_\delta^d = epi(\gamma)\cap \mathbb{B}_\delta^d$, we get that
\[
\frac{\Omega\cap \mathbb{B}_\delta^d}{\sqrt{t}} = \{ (\zeta',z) \in \setR^{d} : \zeta' = r\theta \text{ for some $r \in (0,\delta)$ and $\theta \in \mathbb{S}^{d-2}$, and } z \ge \gamma_t(r,\theta) \}
\]
where
$$
\gamma_{t}(r,\theta)  :=\frac{\gamma(\sqrt{t}\,r\theta)}{\sqrt{t}} \, \cdot 
$$
Moreover, by Proposition \ref{lem:tangent cones as epigraphs},
\[
T_{0_{d-1}} \Omega = \{(\zeta',z) \in \setR^{d} : \zeta' = r\theta \text{ for some $r>0$ and $\theta \in \mathbb{S}^{d-2}$, and } z\ge a(\theta)r \}.
\]
Therefore, the symmetric difference
$([\Omega\cap \mathbb{B}_\delta^d]/\sqrt{t})\,\triangle\,T_{0_{d-1}}\Omega$
intersects each slice $P_\theta \df \mathrm{Span}(\theta,e_d)$ in a narrow strip that can be itself decomposed into segments $S_t(r,\theta)$ with extremal values $a(\theta)r$ and $\gamma_{t}(r,\theta)$. Thus
\begin{equation}\label{eq:exp}
S_t = \int_{\mathbb{S}^{d-2}} 
\int_{0}^{\delta}
\int_{S_t(r,\theta)}
\,\rho(r,z)\,dz\,r^{d-2}dr\,d\sigma(\theta) \qquad \text{with $\rho(r,z) = \sqrt{r^2+z^2}\, e^{-r^2-z^2}$.}
\end{equation}
Denote by $\ell_t(r,\theta)$ the length of the segment $S_t(r,\theta)$, and set $$C_t(r,\theta)\df \sup_{\eta \in S_t(r,\theta)} |\partial_\eta\rho(r,\eta)|.$$
Apply the mean value to $\rho(r,\cdot)$ on the segment $S_t(r,\theta)$ to get
\[
\rho(r,z) \le \rho(r,a(\theta)r) +  C_t(r,\theta)\ell_t(r,\theta).
\]
Use this estimate in \eqref{eq:exp} to get
\begin{align*}
    S_t & \le \int_{\mathbb{S}^{d-2}} 
\int_{0}^{\delta} \bigg( 
\rho(r,a(\theta)r) \ell_t(r,\theta) + C_t(r,\theta)\ell_t^2(r,\theta)\bigg) \,r^{d-2}dr\,d\sigma(\theta).
\end{align*}

For any $\theta \in \mathbb{S}^{d-2}$, use the Taylor expansion of the $\cC^2$ map $r \mapsto \gamma(r\theta)$ to obtain that, for small enough $r>0$,
\begin{align*}
|\gamma(r\theta) - a(\theta)\,r | \le  \bigg(\tfrac{1}{2}\kappa(\theta)\, + \omega(r) \bigg) r^2
\end{align*}
where $\omega(r)>0$ is a non-decreasing function uniform in $\theta$ that goes to $0$ as $r$ goes to $0$. This implies 
$$
\ell_t(r,\theta) = |\gamma_{t}(r,\theta)-a(\theta)\,r|\le  \tfrac{1}{2}\kappa(\theta)\,\sqrt{t} r^{2} + o(\sqrt{t}) \qquad \text{as $t \downarrow 0$,}
$$
where $o(\sqrt{t})$ is uniform in $\theta$. Therefore,
\begin{equation}\label{eq:step2sec6_1}
S_t \le \frac{\sqrt{t}}{2}\int_{\mathbb{S}^{d-2}} 
\int_{0}^{\delta} \rho(r,a(\theta)r) \,\kappa(\theta)\, \,r^{d}dr\,d\sigma(\theta) + o(\sqrt{t}) \qquad \text{as $t \downarrow 0$.}
\end{equation}
Since $\rho(r,a(\theta) r) = r\sqrt{1 + a(\theta)}e^{-r\sqrt{1 + a(\theta)}}$, we obtain
\begin{align*}
    \int_{0}^{\delta} \rho(r,a(\theta)r) \, \,r^{d}dr  = \sqrt{1 + a(\theta)} \int_{0}^{\delta} \, e^{-r\sqrt{1 + a(\theta)}} \,r^{d+1}dr & =  \int_{0}^{\delta/\sqrt{1 + a(\theta)}} \,\frac{e^{-\tau }\tau^{d+1}d\tau}{(1 + a(\theta))^{\frac{d+1}{2}}} \le \frac{\Gamma(d+2)}{(1 + a(\theta))^{\frac{d}{2}+1}}\cdot
\end{align*}
Thus 
\[
\int_{\mathbb{S}^{d-2}} 
\int_{0}^{\delta} \rho(r,a(\theta)r) \,\kappa(\theta)\, \,r^{d}dr\,d\sigma(\theta) \le \Gamma(d+2) \int_{\mathbb{S}^{d-2}} 
\frac{\kappa(\theta)}{(1 + a(\theta))^{\frac{d}{2}+1}}\, d\sigma(\theta) = \Gamma(d+2) SC(x).
\]
Combined with \eqref{eq:step2sec6_1}, this yields that
\[
\frac{S_t}{\sqrt{t}} \le \frac{\Gamma(d+2) SC(x)}{2} + o(1) \qquad \text{as $t \downarrow 0$,}
\]
which implies the desired \eqref{eq:step2sec6}.

\textbf{Step 3.} Acting as in the previous step, we show that
\begin{equation}\label{eq:step3sec6}
\tilde{S}_t = O(\sqrt{t}) \qquad \text{as $t \downarrow 0$}.
\end{equation}

This follows from writing
\[
\tilde{S}_t = \int_{\mathbb{S}^{d-2}} 
\int_{0}^{\delta}
\int_{S_t(r,\theta)}
\,\tilde{\rho}(r,z)\,dz\,r^{d-2}dr\,d\sigma(\theta) \qquad \text{with $\tilde{\rho}(r,z) = (r^2+z^2)\, e^{-r^2-z^2}$}
\]
and proceed as previously. We left the details to the reader.

\textbf{Step 4.} We conclude. From \eqref{eq:step3sec6} we get that for any small enough $t>0$,
\[
\tilde{S}_t \le \Gamma(d+2) SC(x).
\]
Combined with \eqref{eq:step1sec6} and \eqref{eq:step2sec6}, this yields the desired \eqref{eq:refined}.

\section{Concentration estimates}\label{sec:concentration}

The goal of this section is to establish Theorem \ref{th:2}.

\subsection{\texorpdfstring{$\alpha$}{alpha}-Subexponential Random Variables}

We begin with defining the class of $\alpha$-subexponential random variables introduced by G\"otze, Sambale and Sinulis in \cite{GSS}.

\begin{definition}\label{def:alphasub}
Let $Z$ be a real-valued random variable. We say that $Z$ is $\alpha$-subexponential for some $\alpha > 0$ if there exist constants $K\ge 1$ and $C > 0$ such that :
\[
\mathbb{P}[|Z|\ge\epsilon] \le K\exp(-C{\epsilon}^{\alpha}) \qquad \forall \epsilon \ge 0.
\]
\end{definition}

\begin{remark}
    When $\alpha = 1$, the previous class coincides with classical subexponential random variables, while for $\alpha = 2$, it coincides with classical subgaussian random variables, see e.g.~\cite{Vershynin} for more details about these cases. Note also that $\alpha > 2$ implies subgaussian (i.e.~$\alpha = 2$). For values $\alpha < 1$, heavier tails are permitted, as in the case of Weibull random variables, for instance.
\end{remark}

The following lemma is useful for our purposes. We provide a quick proof for completeness.

\begin{lemma}\label{lem:alphasub}
    Let $Z$ be $\alpha$-subexponential for some $\alpha \in (0,2]$. Consider an a.s.~bounded real-valued random variable $W$, and $z\in \mathbb{R}$. Then the product random variable $WZ$ and the translated random variable $Z-z$ are both $\alpha$-subexponential.
\end{lemma}

\begin{proof}
Let $m>0$ be such that $|W|\le m$ a.s. Then for any $\epsilon \ge 0$,
\[
\mathbb{P}(|WZ| \ge \epsilon) = \mathbb{P}(|Z| \ge \epsilon/|W|) \le  \mathbb{P}(|Z| \ge \epsilon/m) \le K\exp(-C{(\epsilon/m)}^{\alpha}).
\]
This shows that $WZ$ is $\alpha$-subexponential. As for $Z-z$, note that $|Z|+|z| \ge |Z-z|$ implies that
\[
\mathbb{P}(|Z-z| \ge \epsilon ) \le \mathbb{P}(|Z| \ge \epsilon - |z|).
\]
Thus $\mathbb{P}(|Z-z| \ge \epsilon ) \le 1$ when $\epsilon \le |z|$, while the $\alpha$-subexponential property of $Z$ implies that if $\epsilon > |z|$,
\[
\mathbb{P}(|Z-z| \ge \epsilon ) \le K \exp(-C (\epsilon - |z|)^\alpha) \le
\begin{cases}
K & \text{if $|z|< \epsilon \le 2 |z|$,}\\
K \exp(-C (\epsilon/2)^\alpha) & \text{if $\epsilon > 2 |z|$.}
\end{cases}
\]
Setting $K' \df K \exp(C|z|^\alpha)$, we obtain that
\[
K' \exp(-C \epsilon^\alpha) \ge \begin{cases}
1 & \text{if $\epsilon \le |z|$,}\\
K & \text{if $|z|< \epsilon \le 2 |z|$,}\\
K \exp(-C(\epsilon/2)^\alpha) & \text{if $\epsilon > 2 |z|$,}
\end{cases}
\]
so that $\mathbb{P}(|Z-z| \ge \epsilon ) \le K' \exp(-C \epsilon^\alpha)$ for any $\epsilon \ge 0$, as desired.
\end{proof}

The next statement is taken from \cite[Corollary 1.4 and Theorem 1.5]{GSS}. It is a concentration result for sample means of $\alpha$-subexponential random variables.

\begin{theorem}\label{th:concentration}
    Let $Z_1, \dots, Z_n \sim Z$ be iid and mean-zero real-valued random variables that are $\alpha$-subexponential for some $\alpha \in (0,2]$. Then there exist constants $c, C> 0$ such that for every $\epsilon \ge 0$,
    \begin{enumerate}[label=(\roman*)]
        \item if $\alpha \in (0,1]$, then
        \[\mathbb{P}\left(\left|\frac{1}{n}\sum_{i=1}^n Z_i\right| \ge \epsilon\right) \le C \exp\left( -c (n\epsilon)^{\alpha} \right),\]
        \item if $\alpha \in (1,2]$, then
        \[
        \mathbb{P}\left(\left|\frac{1}{n}\sum_{i=1}^n Z_i\right| \ge \epsilon\right) \le C \exp\left( -c (\sqrt{n} \epsilon)^{\alpha} \right).
        \]
    \end{enumerate}
\end{theorem}

\begin{remark}
     Note that (1) matches up with the subexponential tail decay given by Bernstein's inequality in case $\alpha=1,$ while (2) coincides with the subgaussian tail decay of Hoeffding's inequality when $\alpha=2$. We refer to \cite{Vershynin} for a nice account on these two classical inequalities.
\end{remark}

\subsection{Convergence in probability}

In this section, we fall back onto our original setup, i.e.~we consider $\{X_i\}_{i\ge 1} \sim X$, a sequence of iid random variables taking values in a Riemannian $d$-dimensional manifold with kinks $(M,g)$. Assume the law $\mathbb{P}_X$ has a $\mathcal{C}^2$ density function $p$ with respect to $\mathrm{vol}_g$ and consider a function $f \in \cC^3(M)$ and a number $t>0$.

We are firstly interested in finding out the concentration of $L_{n,t}f(x)$ around its expected value $L_tf(x)$, for $x \in M$ being either an interior point, an LCDD border point, or a cusp. This is given by the following result.

\begin{proposition}
Assume \( f(X) \) is \(\alpha\)-subexponential for some \(\alpha \in (0,2]\). Then there exist constants \( C_1, C_2 > 0 \)  such that for all \( \epsilon \ge 0 \), 
\begin{enumerate}[label=(\roman*)]
        \item if $\alpha \in (0,1]$, then
        \[
        \mathbb{P} \left( \left| L_{n,t} f(x) - L_t f(x) \right| \ge \epsilon \right) 
        \le C_1 \exp \left( - C_2 \left( n\, t^{\frac{d}{2} + 1} \epsilon \right)^{\alpha} \right),
        \]
        \item if $\alpha \in (1,2]$, then
        \[
        \mathbb{P} \left( \left| L_{n,t} f(x) - L_t f(x) \right| \ge \epsilon \right) 
        \le C_1 \exp \left( - C_2 \left( \sqrt{n}\, t^{\frac{d}{2} + 1} \epsilon \right)^{\alpha} \right).
        \]
\end{enumerate}
Moreover, the constants \( C_1 \) and \( C_2 \) depend only on \( X \), \( f \), and \( \alpha \).
\end{proposition}

\begin{proof}
    Let us provide details for the case $\alpha \in (0,1]$ only, the other one being analogous. Consider the zero-mean real-valued random variables
    \[
    Z_i(t) = e^{-\di^2(x,X_i)/t} (f(x)-f(X_i)) - t^{d/2+1}L_tf(x), \qquad i \ge 1.  \]
    We know from Lemma \ref{lem:alphasub} that these are $\alpha$-subexponential, hence we can apply Theorem \ref{th:concentration} : for any $\epsilon \ge 0$,
    \[
    \mathbb{P} \left( \left| L_{n,t} f(x) - L_t f(x) \right| \ge \epsilon \right) = \mathbb{P} \left(  \frac{1}{  t^{d/2+1}} \left| \frac{1}{n}\sum_{i=1}^n Z_i(t) \right| \ge \epsilon \right) \le  C_1 \exp \left( - C_2 \left( n\, t^{\frac{d}{2} + 1} \epsilon \right)^{\alpha} \right).
    \]
    
\end{proof}

Let us recall the operators
\begin{align*}
\mathcal{L}_tf(x) & \df -\frac{c_d}{\sqrt{t}}\left(  p(x) \, \partial_{v(x)}f(x)  + o(1)\right) - c_{d+1} \bigg( p(x) A_gf(x) +  [p,f]_g(x) \bigg)\\
\mathcal{D}f(x) & \df -c_d\, p(x) \, \partial_{v(x)}f(x) 
\end{align*}
identified in Theorem \ref{th:1}. Since the latter differ from $L_tf(x)$ of a $o(\sqrt{t})$ term, we immediately deduce the following from the previous proposition.

\begin{corollary}
Assume \( f(X) \) is \(\alpha\)-subexponential for some \(\alpha >0\). Then there exist constants \( C_1, C_2 > 0 \)  such that for all \( \epsilon \ge 0 \), we have :
\begin{enumerate}[label=(\roman*)]
        \item if $\alpha \in (0,1]$, then
        \[
        \mathbb{P} \left( \left| L_{n,t} f(x) - \mathcal{L}_t f(x) \right| \ge \epsilon \right) 
        \le C_1 \exp \left( - C_2 \left( n\, t^{\frac{d}{2} + 1} \epsilon \right)^{\alpha} \right),
        \]
        \item if $\alpha \in (1,2]$, then
        \[
        \mathbb{P} \left( \left| L_{n,t} f(x) - \mathcal{L}_t f(x) \right| \ge \epsilon \right) 
        \le C_1 \exp \left( - C_2 \left( \sqrt{n}\, t^{\frac{d}{2} + 1} \epsilon \right)^{\alpha} \right).
        \]
\end{enumerate}
\end{corollary}

We are now in a position to prove (i) in Theorem \ref{th:3}. Consider $t_n \to 0$ such that \( \sqrt{n}\, t_n^{\frac{d+1}{2}} \to \infty \) as \( n \to \infty \). From the previous corollary, we obtain that for any $\epsilon \ge 0$ and $t>0$,
\[
\mathbb{P} \left( \left| \sqrt{t}L_{n,t} f(x) - \mathcal{D} f(x) \right| \ge \epsilon \right) 
        \le \begin{cases}
        C_1 \exp \left( - C_2 \left( n\, t^{\frac{d+1}{2}} \epsilon \right)^{\alpha} \right) & \text{if $\alpha \in (0,1]$},\\
        C_1 \exp \left( - C_2 \left( \sqrt{n}\, t^{\frac{d+1}{2}} \epsilon \right)^{\alpha} \right) & \text{if $\alpha \in (1,2]$}.
        \end{cases}
\]
The asymptotic relation between $n$ and $t_n$ implies that
\[
C_1 \exp \left( - C_2 \left( n\, t_n^{\frac{d+1}{2}} \epsilon \right)^{\alpha} \right)\to 0 \qquad \text{and} \qquad  C_1 \exp \left( - C_2 \left( \sqrt{n}\, t_n^{\frac{d+1}{2}} \epsilon \right)^{\alpha} \right) \to 0 \qquad \text{as $n \to +\infty$}.
\]
This yields the desired result.

\subsection{Almost sure convergence} Let us now prove (ii) in Theorem \ref{th:3}. To this aim, recall that a sequence of real-valued random variables $\{Z_n\}$ completely converges to another real-valued random variable $Z$ if for any $\epsilon>0$,
    \[
    \sum_{n} \mathbb{P}(|Z_n - Z| \ge \epsilon) < +\infty.
    \]
It is easily seen from the first Borel--Cantelli lemma that complete convergence implies almost sure convergence. We shall need an elementary lemma.

\begin{lemma}\label{lem:elem}
    Let $\{b_n\}$ be a sequence of positive real numbers such that
    \[
    \frac{b_n}{\ln(n)} \stackrel{n\to +\infty}{\longrightarrow} +\infty \quad \text{and} \quad \sum_{n\ge 1} e^{-b_n} < +\infty.
    \]
    Then for any $\delta \in (0,1)$,
    \[
    \sum_{n\ge 1} e^{-\delta b_n} < +\infty
    \]
\end{lemma}

\begin{proof}
Take $\delta \in (0,1)$. For any $c>0$ there exists $N \in \mathbb{N}$ such that for any $n \ge N$,
\[
\frac{b_n}{\ln(n)} \ge c.
\]
For these integers $n$,
\[
e^{-\delta b_n} \le e^{-\delta c \ln(n)} = \frac{1}{n^{\delta c}} \, \cdot 
\]
Therefore, if $c > 1/\delta$,
\[
\sum e^{-\delta b_n} \le \sum \frac{1}{n^{\delta c}} < +\infty.
\]
\end{proof}

We can now prove (ii) in Theorem \ref{th:3}. Consider $t_n \to 0$ such that $\left(\sqrt{n}\, t_n^{\frac{d+1}{2} } \right)^{\alpha}/\ln(n)\to \infty$ as $n \to \infty$. For any $\epsilon >0$,
\begin{align*}
\sum_{n} \mathbb{P} \left( \left| \sqrt{t_n} L_{n,t_n} f(x) - \mathcal{D} f(x) \right| \ge \epsilon \right) & \le  C_1 \sum_{n}  \exp \left( - C_2 \left( n\, t_n^{\frac{d+1}{2}} \epsilon \right)^{\alpha} \right)\\
& =  C_1 \sum_{n}  \exp \left( - b_n\epsilon^\alpha \right)
\end{align*}
with $b_n \df C_2 \left( n\, t_n^{\frac{d+1}{2} } \right)^{\alpha}$. The assumption on $\{t_n\}$ implies that $b_n/\ln(n) \to \infty$ as $n\to \infty$, hence Lemma \ref{lem:elem} yields that 
\[
\sum_{n} \mathbb{P} \left( \left| \sqrt{t_n}L_{n,t_n} f(x) - \mathcal{D} f(x) \right| \ge \epsilon \right) \le \sum_{n}  \exp \left( - b_n\epsilon^\alpha \right) < +\infty.
\]
Thus $|\sqrt{t_n}L_{n,t_n}f(x) - \mathcal{D}f(x)|$ completely converges to $0$, thus it converges almost surely.

\section{Numerical simulations}\label{sec:numerics}

In this section, we empirically compare $L_{n,t}$ and $L_t$ and their scaled versions $\sqrt{t} L_{n,t}$ and $\sqrt{t} L_t$ on three models of submanifolds with kinks : the three-dimensional unit ball, which is a manifold with boundary, the three-dimensional unit cube, which is a manifold with corners, and a two-dimensional cusp, which is a manifold with kinks that is not a manifold with corners.

 \subsection{Experiments for a 3D ball}


We consider the ball $\{(x,y,z) : x^2 +y^2 + z^2 \le 1\}$ and compute the values of $L_{n,t}$, $L_t$, $\sqrt{t} L_{n,t}$ and $\sqrt{t} L_t$ at $n:=10^8$ uniformly generated sample points. We let the kernel bandwidths \(t:=t_n\) vary logarithmically from \(0.05\) down to \(0.01\) (20 values). We consider the function
\[
f(x,y,z) = x+y+z.
\]
Note that this function is harmonic on $\setR^3$, i.e.~its Laplacian is constantly equal to zero.

 Theorem \ref{th:3} gives us a condition on the convergence in probability of $L_{n,t_n}$ for $n\to \infty, t_n\to 0,$ which is $\sqrt{n}t_n^{(d+1)/2} = \sqrt{n}t_n^{2} \to \infty.$ In our experiments, the values of $\sqrt{n}t_n^{2}$ vary between $\sqrt{10^8}\times (0.01)^{2}=1$ and $\sqrt{10^8}\times (0.05)^{2}=25.$ These values are not large, but one can already see with them that the behavior of $L_{n,{t_n}}$ matches up with the asymptotics provided by our theory.

The tables below show how the graph Laplace operator $L_{n,t}f$ and its expectation $L_tf$ behave asymptotically on the unit ball $\mathbb{B}^3$. Note that 'int' corresponds to an interior point, namely the origin $0_3$, and 'bd' corresponds to the boundary point $(1,0,0)$.\\ 

\noindent
\begin{minipage}{0.5\textwidth}
\centering
\captionof{table}{Discrete Laplacians}
\begin{adjustbox}{max width=\linewidth}
\begin{tabular}{rrrrr}
\toprule
$t$ & $L_{n,t}$ int & $L_{n,t}$ bd & $\sqrt{t} L_{n,t}$ int & $\sqrt{t} L_{n,t}$ bd \\
\midrule
0.050000 & -0.000183 & 1.634825 & -0.000041 & 0.365558 \\
0.047895 & -0.000193 & 1.672184 & -0.000042 & 0.365955 \\
0.045789 & -0.000201 & 1.712053 & -0.000043 & 0.366353 \\
0.043684 & -0.000204 & 1.754732 & -0.000043 & 0.366752 \\
0.041579 & -0.000200 & 1.800571 & -0.000041 & 0.367153 \\
0.039474 & -0.000185 & 1.849985 & -0.000037 & 0.367555 \\
0.037368 & -0.000154 & 1.903471 & -0.000030 & 0.367958 \\
0.035263 & -0.000101 & 1.961627 & -0.000019 & 0.368364 \\
0.033158 & -0.000015 & 2.025183 & -0.000003 & 0.368772 \\
0.031053 & 0.000117 & 2.095039 & 0.000021 & 0.369183 \\
0.028947 & 0.000312 & 2.172321 & 0.000053 & 0.369597 \\
0.026842 & 0.000594 & 2.258457 & 0.000097 & 0.370016 \\
0.024737 & 0.000996 & 2.355294 & 0.000157 & 0.370439 \\
0.022632 & 0.001563 & 2.465269 & 0.000235 & 0.370870 \\
0.020526 & 0.002355 & 2.591675 & 0.000337 & 0.371309 \\
0.018421 & 0.003461 & 2.739100 & 0.000470 & 0.371762 \\
0.016316 & 0.005012 & 2.914158 & 0.000640 & 0.372235 \\
0.014211 & 0.007220 & 3.126828 & 0.000861 & 0.372743 \\
0.012105 & 0.010473 & 3.393033 & 0.001152 & 0.373315 \\
0.010000 & 0.015566 & 3.740139 & 0.001557 & 0.374014 \\
\bottomrule
\end{tabular}
\end{adjustbox}
\end{minipage}%
\hfill
\begin{minipage}{0.5\textwidth}
\centering
\captionof{table}{Continuous Laplacians}
\begin{adjustbox}{max width=\linewidth}
\begin{tabular}{rrrr}
\toprule
$L_t$ int & $L_t$ bd & $\sqrt{t} L_t$ int & $\sqrt{t} L_t$ bd \\
\midrule
-7.5826e-16 & 1.638128 & -1.6955e-16 & 0.366297 \\
-1.0393e-15 & 1.675388 & -2.2746e-16 & 0.366657 \\
-9.1152e-16 & 1.715124 & -1.9505e-16 & 0.367010 \\
-1.4726e-15 & 1.757625 & -3.0779e-16 & 0.367357 \\
-1.3166e-15 & 1.803230 & -2.6847e-16 & 0.367695 \\
-1.0559e-15 & 1.852339 & -2.0979e-16 & 0.368022 \\
-1.2042e-15 & 1.905428 & -2.3278e-16 & 0.368337 \\
-1.5811e-15 & 1.963070 & -2.9690e-16 & 0.368635 \\
-2.0095e-15 & 2.025957 & -3.6591e-16 & 0.368913 \\
-1.8371e-15 & 2.094937 & -3.2373e-16 & 0.369165 \\
-2.3320e-15 & 2.171064 & -3.9676e-16 & 0.369383 \\
-2.4203e-15 & 2.255662 & -3.9653e-16 & 0.369558 \\
-2.7905e-15 & 2.350426 & -4.3889e-16 & 0.369674 \\
-2.8125e-15 & 2.457565 & -4.2311e-16 & 0.369711 \\
-3.7982e-15 & 2.580030 & -5.4417e-16 & 0.369641 \\
-4.1979e-15 & 2.721868 & -5.6975e-16 & 0.369423 \\
-4.9808e-15 & 2.888827 & -6.3622e-16 & 0.368999 \\
-5.8463e-15 & 3.089426 & -6.9692e-16 & 0.368284 \\
-7.7553e-15 & 3.336974 & -8.5326e-16 & 0.367147 \\
-1.1482e-14 & 3.653681 & -1.1482e-15 & 0.365368 \\
\bottomrule
\end{tabular}
\end{adjustbox}
\end{minipage}

\hfill \\

Below are the plots for the above tables.

\begin{figure}[htbp]
  \centering
  \begin{subfigure}[t]{0.48\textwidth}
    \includegraphics[width=\linewidth]{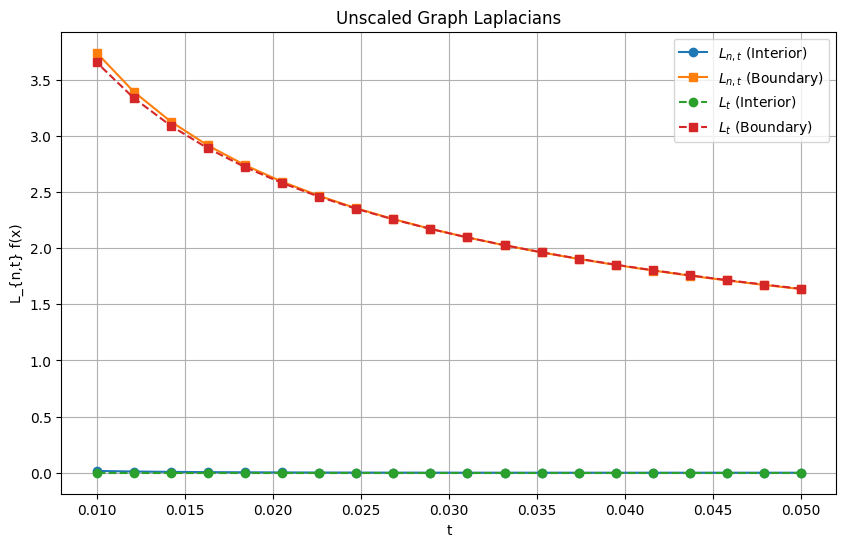}
    \caption{Unscaled graph Laplacian  – 3D ball}
  \end{subfigure}
  \hfill
  \begin{subfigure}[t]{0.48\textwidth}
    \includegraphics[width=\linewidth]{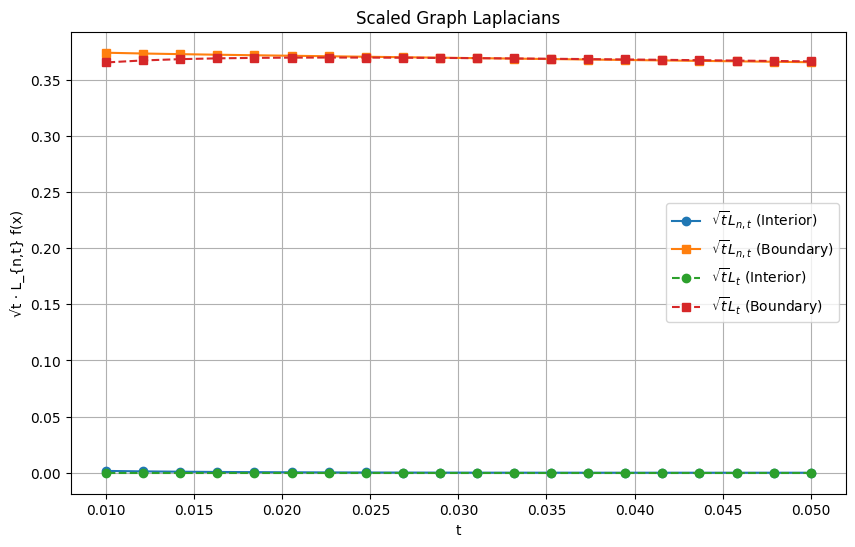}
    \caption{Scaled graph Laplacian – 3D ball}
  \end{subfigure}
  \caption{Comparison of graph Laplacians on a 3D ball}
\end{figure}

\subsection{Experiments for a 3D cube}

We consider the unit hypercube \([0,1]^3\) with vertices 
\[
a = (0,0,0), \qquad b = (1,0,0), \qquad c = (0,1,0), \qquad d = (0,0,1),
\]
\[
e = (1,1,0), \qquad f = (1,0,1), \qquad g = (0,1,1), \qquad h = (1,1,1),
\]
and the points $I,F,E,V$ defined below : 
\[
\begin{array}{rll}
I & \;=\; (a+b+c+\dots+h)/8 & \text{(interior point),}\\
F & \;=\;\displaystyle (a+b+d+e)/4 & \text{(face midpoint),} \\
E & \;=\;\displaystyle (a+b)/2 & \text{(edge midpoint),}\\
V & \; =\; a & \text{(vertex).} 
\end{array}
\]


We use the same bandwidths and the same function as in the previous section. Our numerical results are presented in Table \ref{table3Dhypercube} and Figure \ref{figure3Dhypercube}. \\

\vspace{-7mm}

\begin{table}[H]
\centering\scriptsize
\vspace{-0.5\baselineskip}

\begin{subtable}[t]{0.48\textwidth}
\centering
\caption{Interior point $I$}
\begin{tabular}{@{}crrrr@{}}
\toprule
$t$ & $L_{n,t}$ & $L_t$ & $\sqrt t\,L_{n,t}$ & $\sqrt t\,L_t$ \\
\midrule
5.000e-02 & -4.708e-16 & -4.708e-16 & -1.053e-16 & -1.053e-16 \\
3.737e-02 & -3.290e-15 & -3.290e-15 & -6.360e-16 & -6.360e-16 \\
2.474e-02 & -6.425e-15 & -6.425e-15 & -1.010e-15 & -1.010e-15 \\
1.421e-02 & -1.760e-14 & -1.760e-14 & -2.098e-15 & -2.098e-15 \\
1.000e-02 & -2.116e-14 & -2.116e-14 & -2.116e-15 & -2.116e-15 \\
\bottomrule
\end{tabular}
\end{subtable}
\hfill
\begin{subtable}[t]{0.48\textwidth}
\centering
\caption{Face midpoint $F$}
\begin{tabular}{@{}crrrr@{}}
\toprule
$t$ & $L_{n,t}$ & $L_t$ & $\sqrt t\,L_{n,t}$ & $\sqrt t\,L_t$ \\
\midrule
5.000e-02 & 7.002e+00 & 7.016e+00 & 1.566e+00 & 1.569e+00 \\
3.737e-02 & 8.120e+00 & 8.142e+00 & 1.570e+00 & 1.574e+00 \\
2.474e-02 & 9.985e+00 & 1.002e+01 & 1.570e+00 & 1.577e+00 \\
1.421e-02 & 1.318e+01 & 1.326e+01 & 1.571e+00 & 1.581e+00 \\
1.000e-02 & 1.573e+01 & 1.586e+01 & 1.573e+00 & 1.586e+00 \\
\bottomrule
\end{tabular}
\end{subtable}

\vspace{0.5\baselineskip}

\begin{subtable}[t]{0.48\textwidth}
\centering
\caption{Edge midpoint $E$}
\begin{tabular}{@{}crrrr@{}}
\toprule
$t$ & $L_{n,t}$ & $L_t$ & $\sqrt t\,L_{n,t}$ & $\sqrt t\,L_t$ \\
\midrule
5.000e-02 & -7.012e+00 & -7.027e+00 & -1.568e+00 & -1.571e+00 \\
3.737e-02 & -8.121e+00 & -8.144e+00 & -1.570e+00 & -1.574e+00 \\
2.474e-02 & -9.984e+00 & -1.002e+01 & -1.570e+00 & -1.577e+00 \\
1.421e-02 & -1.317e+01 & -1.326e+01 & -1.570e+00 & -1.581e+00 \\
1.000e-02 & -1.570e+01 & -1.586e+01 & -1.570e+00 & -1.586e+00 \\
\bottomrule
\end{tabular}
\end{subtable}
\hfill
\begin{subtable}[t]{0.48\textwidth}
\centering
\caption{Vertex $V$}
\begin{tabular}{@{}crrrr@{}}
\toprule
$t$ & $L_{n,t}$ & $L_t$ & $\sqrt t\,L_{n,t}$ & $\sqrt t\,L_t$ \\
\midrule
5.000e-02 & -5.264e+00 & -5.278e+00 & -1.177e+00 & -1.180e+00 \\
3.737e-02 & -6.088e+00 & -6.110e+00 & -1.177e+00 & -1.181e+00 \\
2.474e-02 & -7.450e+00 & -7.519e+00 & -1.177e+00 & -1.183e+00 \\
1.421e-02 & -9.873e+00 & -9.948e+00 & -1.177e+00 & -1.186e+00 \\
1.000e-02 & -1.177e+01 & -1.189e+01 & -1.177e+00 & -1.189e+00 \\
\bottomrule
\end{tabular}
\end{subtable}

\caption{Numerical values of the graph Laplacian operators at various points.}
\label{table3Dhypercube}
\end{table}

\vspace{-5mm}

\begin{figure}[H]
  \centering
  \begin{subfigure}{0.48\linewidth}\centering
    \includegraphics[width=\linewidth]{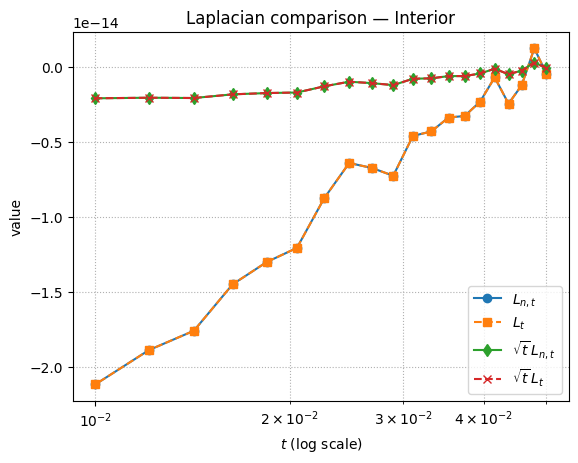}
    \caption{Interior}
  \end{subfigure}%
  \hfill
  \begin{subfigure}{0.48\linewidth}\centering
    \includegraphics[width=\linewidth]{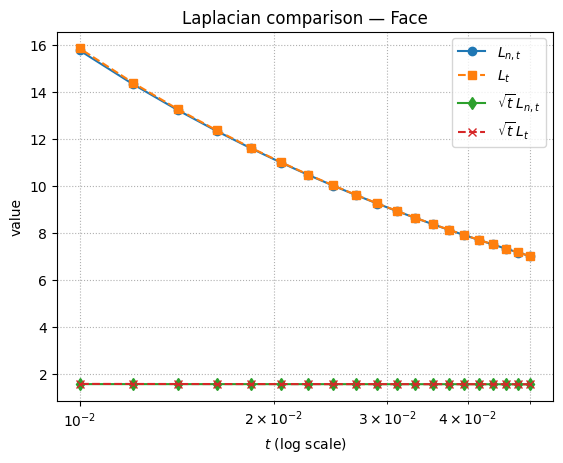}
    \caption{Face midpoint}
  \end{subfigure}

  \vspace{1em}

  \begin{subfigure}{0.48\linewidth}\centering
    \includegraphics[width=\linewidth]{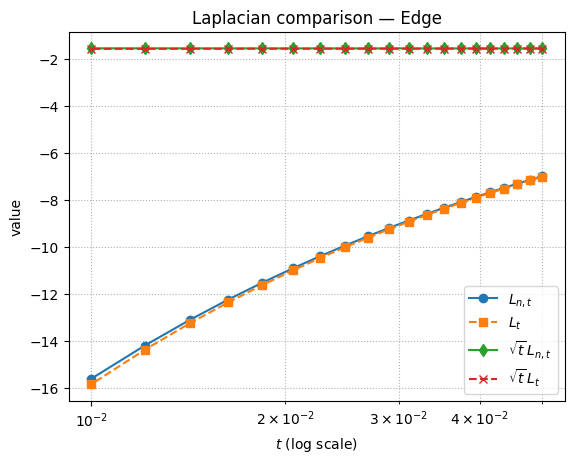}
    \caption{Edge midpoint}
  \end{subfigure}%
  \hfill
  \begin{subfigure}{0.48\linewidth}\centering
    \includegraphics[width=\linewidth]{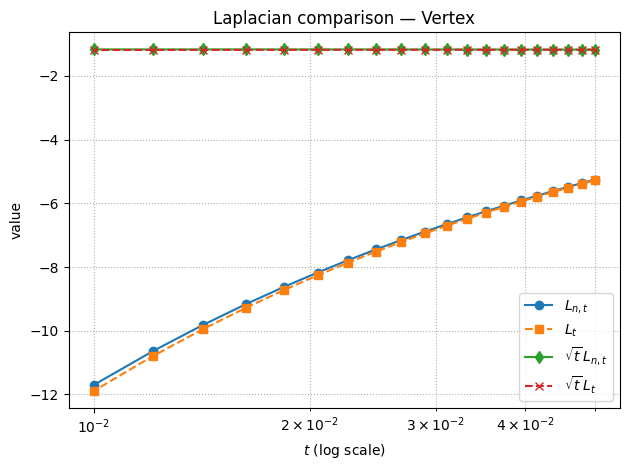}
    \caption{Vertex}
  \end{subfigure}
  
  \caption{Comparaison of graph Laplacians on a 3D hypercube}
  \label{figure3Dhypercube}
\end{figure}

\subsection{Experiments for a 2D cusp} 

In this experiment, we take sample $n:=10^8$ points on a cusp region $\{(x,y): y\le x^2\},$ which is a two dimensional ($d:=2$) manifold with an essential kink at $0_2$ that do not come from a uniform distribution, unlike in the previous examples. Then we take $20$ equispaced bandwidth ($t$) values in the range $[0.001, 0.01],$ find, tabulate and graph the values of discrete graph Laplacian $L_{n,t}f(0_2).$ We choose the harmonic function
\[
f(x,y)=x+y.
\]
Below are the table and the graph. We notice that the minimum value of $\sqrt{n}{t_n}^{(d+1)/2}=\sqrt{10^8}({0.001})^{3/2}=10^{-1/2},$ which is indeed pretty small (so this means for large $n$ it will be bigger, as required by Theorem \ref{th:3}), and \textit{yet} we notice the expected asymptotic behavior of $L_{n,t}f(0_2)$ given by Theorem \ref{th:3}, namely its value is very small for small $t$ since $f$ is harmonic and the Bouligand tangent cone/inward sector of the cusp at $0_2$ is one dimensional and hence has Lebesgue measure zero, so both $L_tf(0_2), L_{n,t}f(0_2)$ should approach zero for the pairs $(n,t_n)$ given by Theorem \ref{th:3}. The plot and the table are below:
\vspace{5mm}
\begin{figure}[H]
    \centering
    \begin{minipage}[t]{0.49\textwidth}
        \centering
        \vspace*{-1.5\baselineskip}  
        \includegraphics[width=\textwidth]{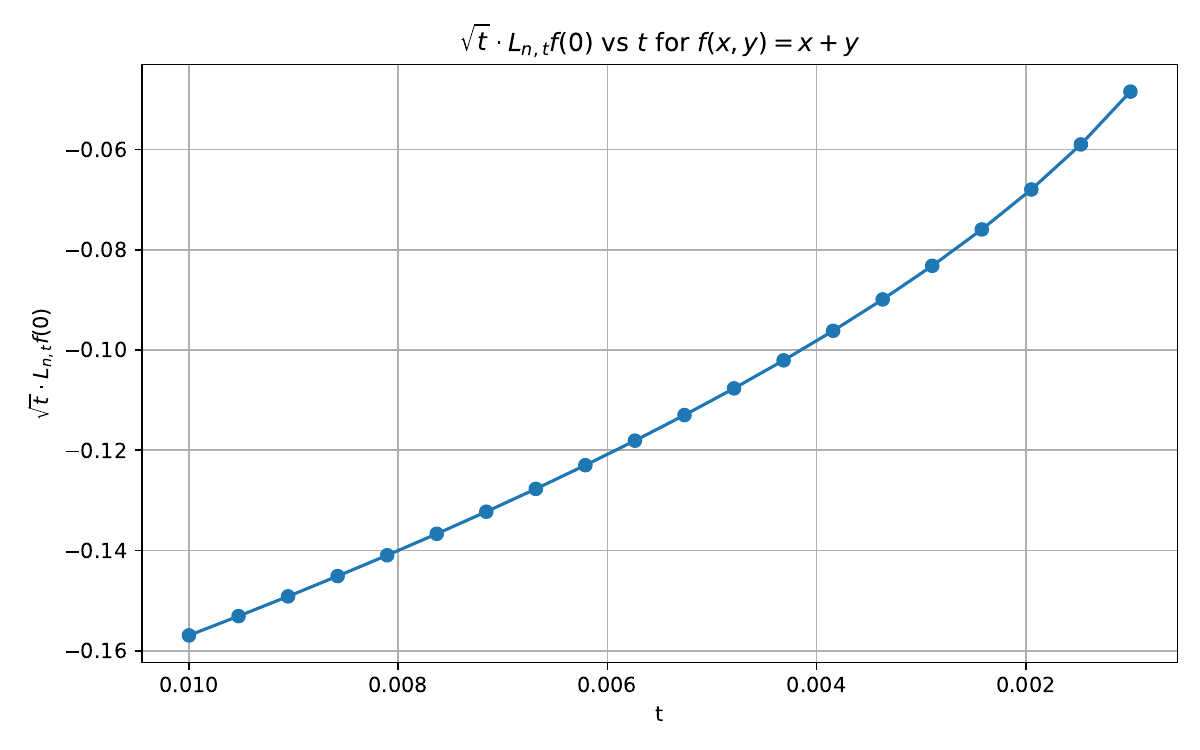}
        \captionof{figure}{Plot of $\sqrt{t} \cdot L_{n,t}f(0)$ vs $t$ for $f(x, y) = x + y$.}
        \label{fig:laplacian-plot}
    \end{minipage}%
    \hfill
    \begin{minipage}[t]{0.49\textwidth}
        \centering
        \captionof{table}{Values of $\sqrt{t} \cdot L_{n,t}f(0)$ for various $t$.}
        \label{tab:laplacian-table}
        {\footnotesize
        \renewcommand{\arraystretch}{0.85}%
        \input{laplacian_table.tex}
        }
    \end{minipage}
\end{figure}


\section{Acknowledgments.} The authors are funded by the Research Foundation – Flanders (FWO) via the Odysseus II programme no.~G0DBZ23N. The first author has also been partially supported by MBZUAI SU Fund for this work. Both authors thank Dominic Joyce for helpful answers concerning manifolds with corners, Iosif Pinelis who suggested the proof of \eqref{eq:epi}, and Laurent Bessières for a comment on the first version of the paper.

\bibliographystyle{unsrt}
\bibliography{biblio}

\end{document}

%% file: laplacian_table.tex
\begin{tabular}{r r}
\toprule
t & $\sqrt{t} L_{n,t}f(0)$ \\
\midrule
0.0100 & -1.569476e-01 \\
0.0095 & -1.531040e-01 \\
0.0091 & -1.491638e-01 \\
0.0086 & -1.451197e-01 \\
0.0081 & -1.409632e-01 \\
0.0076 & -1.366848e-01 \\
0.0072 & -1.322733e-01 \\
0.0067 & -1.277157e-01 \\
0.0062 & -1.229968e-01 \\
0.0057 & -1.180983e-01 \\
0.0053 & -1.129983e-01 \\
0.0048 & -1.076696e-01 \\
0.0043 & -1.020786e-01 \\
0.0038 & -9.618175e-02 \\
0.0034 & -8.992210e-02 \\
0.0029 & -8.322148e-02 \\
0.0024 & -7.596797e-02 \\
0.0019 & -6.799105e-02 \\
0.0015 & -5.900677e-02 \\
0.0010 & -4.846162e-02 \\
\bottomrule
\end{tabular}